\documentclass{amsart}
\usepackage[utf8]{inputenc}
\usepackage{setspace}
\usepackage[margin=1.25in]{geometry}
\usepackage{graphicx}
\graphicspath{ {./figures/} }
\usepackage{subcaption}
\usepackage{amsmath}
\usepackage{amssymb}
\usepackage{mathtools}
\usepackage{lineno}
\usepackage{amsfonts}
\usepackage{xfrac} 
\usepackage{graphicx} 
\usepackage{amsthm}
\usepackage[all]{xy}

\usepackage[pagebackref=true]{hyperref}
\usepackage{float}
\usepackage{braket}

\usepackage{kbordermatrix} 

\usepackage{multirow}
\usepackage{diagbox}
\usepackage{slashbox}

\usepackage{tikz-cd}
\usetikzlibrary{cd}

\newtheorem{theorem}{Theorem}

\newtheorem{example}[theorem]{Example}
\newtheorem{definition}[theorem]{Definition}
\newtheorem{lemma}[theorem]{Lemma}
\newtheorem{corollary}[theorem]{Corollary}
\newtheorem{remark}[theorem]{Remark}
\newtheorem{proposition}[theorem]{Proposition}
\newtheorem*{theorem*}{Theorem}

\newcommand {\rat}   {\ensuremath{\mathbb{Q}}}
\newcommand {\pt}    {\ensuremath{\operatorname{pt}}}
\newcommand {\Hom}   {\ensuremath{\operatorname{Hom}}}

\begin{document}

\title{Link Bundles and Intersection Spaces of Complex Toric Varieties}

\author{Markus Banagl}

\address{Institut f\"ur Mathematik, Universit\"at Heidelberg,
  Im Neuenheimer Feld 205, 69120 Heidelberg, Germany}

\email{banagl@mathi.uni-heidelberg.de}

\author{Shahryar Ghaed Sharaf}

\address{Fakultät für Mathematik und Informatik, Friedrich-Schiller-Universität Jena,
	Ernst-Abbe-Platz 2 , 07743 Jena, Germany}

\email{shahryar.ghaed.sharaf@uni-jena.de}

\date{May 31, 2024}

\subjclass[2020]{55N33, 14M25, 57S12, 57N80}

\keywords{Toric Varieties, Intersection Homology, Stratified Spaces}

\begin{abstract}
There exist several homology theories for singular spaces that satisfy generalized 
Poincar\'e duality, including 
Goresky-MacPherson's intersection homology, Cheeger's $L^2$ cohomology 
and the homology of intersection spaces. The intersection homology and $L^2$ cohomology 
of toric varieties is known. Here, we compute the rational homology of 
intersection spaces of complex 3-dimensional toric varieties and compare it
to intersection homology.
To achieve this, we analyze cell structures and topological 
stratifications of these varieties and determine compatible structures on 
their singularity links. 
In particular, we compute the homology of links in 3-dimensional toric varieties.
We find it convenient to use the concept of a rational homology stratification.
It turns out that the intersection space 
homology of a toric variety, contrary to its intersection homology, is not 
combinatorially invariant and thus retains more refined information on the defining fan.
\end{abstract}

\maketitle

\tableofcontents
	
	
	\section{Introduction}
	
	A toric variety is a complex algebraic variety containing an
	algebraic torus as an open dense subset, such that the action of the
	torus on itself extends to the whole variety. Such varieties
	can be described by combinatorial data, called fans.
	Toric varieties are generally singular, and their homology
	does not usually satisfy Poincar\'e duality.
	
	We focus here on compact complex $3$-dimensional singular toric varieties associated to complete fans.
	For these, the fourth Betti number only depends on the number $f_1$ of
	$1$-dimensional cones in the fan, but the second Betti number depends
	in addition on a parameter $b$, which then measures the discrepancy
	between these Betti numbers. The parameter $b$ is not ``combinatorial'',
	i.e. it depends on the precise form of the fan, not just on the
	number of cones in the various dimensions. If $b\not= 0$, then
	Poincar\'e duality is violated.
	
	There are several cohomology theories available that
	restore duality. The intersection cohomology $IH^*$ of Goresky and
	MacPherson \cite{goresky1980intersection},
	the $L^2$-cohomology of Cheeger \cite{cheeger1980hodge}, \cite{cheeger1979spectral},
	\cite{cheeger1983spectral} for appropriately conical metrics
	on the regular part, and the homotopy-theoretic method of intersection
	spaces $IX$ \cite{banagl2010intersection}, 
	yielding a cohomology theory $HI^* (X) := H^* (IX)$.
	Using a conical metric, Cheeger's cohomology is
	isomorphic to intersection cohomology for a space with, say, only
	even-dimensional strata. While intersection homology of a complex
	$3$-dimensional toric variety does repair Poincar\'e duality, it retains
	only
	a relatively small amount of the actual topology: The middle
	Betti number (indeed all odd Betti numbers) vanishes (Stanley \cite{stanley1987generalized}), and the
	second (and thus fourth) Betti number only depends on $f_1$, as has been shown by Fieseler in \cite{fieseler1991rational}.
	
	It is well-known that $IH^*$ and $HI^*$ are generally not isomorphic.
	Does the homology of intersection spaces perhaps retain more information
	of the toric variety (while also repairing duality)?
	We find this indeed to be the case:
	One of our main results is the computation of $HI^*$ for compact
	complex $3$-dimensional toric varieties. We show that
	$HI^*$ depends on $f_1,$ $f_2$ and $b$; more precisely:
	\begin{theorem*}
		(Theorem \ref{IntSecBetti})\\
		Let $X_{\mathcal{P}}$ be a compact complex $3$-dimensional toric variety with $m$ $\mathbb{Q}$-isolated singularities, $m \geq 1$, where $\mathcal{P}$ is the underlying polytope. Let $\Sigma$ be the dual fan to $\mathcal{P}$. We denote the number of $1$-dimensional and $2$-dimensional cones of $\Sigma$ by $f_{1}$ and $f_{2}$. Then
		\begin{align*}
			\operatorname{rk}(\widetilde{H}_{6}(IX))&=0 \\
			\operatorname{rk}(\widetilde{H}_{5}(IX))&=m-1 \\
			\operatorname{rk}(\widetilde{H}_{4}(IX))&= f_{1}-3-b\\
			\operatorname{rk}(\widetilde{H}_{3}(IX))&=2(3f_{1}-f_{2}-b-6) \\
			\operatorname{rk}(\widetilde{H}_{2}(IX))&=f_{1}-3-b \\
			\operatorname{rk}(\widetilde{H}_{1}(IX))&=m-1\\
			\operatorname{rk}(\widetilde{H}_{0}(IX))&=0,
		\end{align*}
where $b=\operatorname{rk}(H_{4}(X))-\operatorname{rk}(H_{2}(X))$ and
$\widetilde{H}_* (-)$ denotes reduced singular homology. 
	\end{theorem*}

Table \ref{table.hixihx} of the concluding Section \ref{conclusion} compares the above ranks
to the ranks of intersection homology.	
The latter is obtained by imposing restrictions on how chains intersect 
the strata. Alternatively, one may wish to implement similar local modifications on 
the spatial level rather than the chain level. This motivates in part the theory of 
intersection spaces introduced by the first named author in \cite{banagl2010intersection}. 
To a stratified pseudomanifold $X$ and perversity $\bar{p}$, one wishes to associate a space
\[
I^{\bar{p}}X,
\]
an \emph{intersection space} of $X$, such that the ordinary 
reduced rational homology $\widetilde{H}_{\ast}(I^{\bar{p}}X; \mathbb{Q})$ 
satisfies generalized Poincar\'e 
duality across complementary perversities when $X$ is closed and oriented. 
In the absence of odd-co-dimensional strata and if $\bar{p}$ is the middle 
perversity $\bar{p}=\bar{m}$, then $IX=I^{\bar{m}}X$ satisfies Poincar\'e self-duality 
rationally.
The idea in forming intersection spaces
	is to use the homotopy cofiber of Moore approximations of singularity links.
	Thus we carry out a detailed analysis of these links in the toric case.
	An important structural feature
	that has been particularly emphasized by
	Agust\'in and de Bobadilla in their work on intersection space pairs \cite{vicente2020intersection}
	is that the link bundles of toric varieties
	can be trivialized, see Proposition \ref{trivlink} in the present paper. 
	In \cite{fischli1992toric}, Fischli describes a procedure to endow compact toric 
	varieties with CW structures. However, due to the resulting complexity in higher 
	dimensions, the computation was carried out only up to dimension 2. 
	We construct CW structures on complex 3-dimensional
	toric varieties and on their links. These structures allow us to
	determine the homology of all links, particularly of the
	real $5$-dimensional ones (Proposition \ref{Link3D}). This material may be of independent interest, notwithstanding the theory of intersection spaces.  
	The CW structure on the links also allows us to construct
	their Moore approximations explicitly.
	We use a concept of $\mathbb{Q}$-homology stratified pseudomanifolds in order to
	be able to form the intersection space based on isolated singularity
	techniques. The observation here is that the real $3$-dimensional links
	are rational homology spheres, so do not disturb duality. Homology stratifications are treated by Rourke and Sanderson in \cite{rourkesanderson}.

The theory $HI^{\ast}$ has
had applications in fiber bundle theory and computation of equivariant cohomology
(\cite{banagl2013isometric}), K-theory (\cite[Chapter~2.8]{banagl2010intersection},\cite{spiegel2013k}), algebraic geometry (smooth deformation of
singular varieties \cite{banagl2012intersection}, perverse sheaves 
\cite{banagl2014intersection}, mirror symmetry \cite[Chapter~3.8]{banagl2010intersection}),
 theoretical Physics (\cite[Chapter~3]{banagl2010intersection}, 
\cite{banagl2014intersection}). An $L^2$ Hodge theorem for $HI^{\ast}$ has been 
provided in \cite{banagl2019hodge} based on the use of the scattering metric. 
That theorem then provides a purely topological interpretation of spaces
of extended weighted $L^2$-harmonic forms with respect to fibered scattering metrics.
A de Rham 
complex describing $HI^{\ast}$ with $\mathbb{R}$ coefficients was developed in \cite{banagl2016foliated} for stratification depth one and by Essig in 
\cite{essig2016rham} for higher stratification depth, see also
\cite{essigmatcontemp}.
Schl\"oder and Essig obtained multiplicative 
de Rham theorems for intersection space cohomology in \cite{schloder2019multiplicative}. 
Higher stratification depth was addressed through the notion of intersection space pairs, 
introduced by Agust\'in Vicente and De Bobadilla in  \cite{vicente2020intersection}. 
Verdier self-dual intersection space complexes are treated by  Agust\'in Vicente, Essig and 
De Bobadilla  in  \cite{agustin2020self}. For isolated singularities, 
Klimczak \cite{klimczak2017poincare} explained how to attach a top dimensional cell 
to $IX$ so as to make it into a geometric Poincar\'e complex whose Poincar\'e duality 
isomorphism is given by capping with a fundamental class. Wrazidlo has generalized this 
to non-isolated singularities of depth one in \cite{wrazidlo2021fundamental}.

\section{Preliminaries}
\label{sec.preliminaries}

Let us recall the definition of a rational polyhedral cone in $\mathbb{R}^{n}$.
\begin{definition}\label{Def:Polyhedra}
	We call $\sigma \subset \mathbb{R}^{n}$ a rational polyhedral cone, if there 
	exist finitely many vectors $\nu_{1}, \dots ,\nu_{k} \in \mathbb{Z}^{n}$ 
	such that 
	\begin{align}
	\sigma = \{ x_{1} \nu_{1}+ \dots + x_{k} \nu_{k} ~\vert ~ x_{1}, \dots, x_{k} 
	  \in \mathbb{R}_{\geq 0} \}.
	\end{align} 
	$\sigma$ is \textbf{proper} if it is not spanned by any proper subset of $\{\nu_{1}, \dots , \nu_{k} \}$,	and $\nu_{1}, \dots , \nu_{k}$ lie strictly on one side of some hyperplane in $\mathbb{R}^{n}$.
\end{definition}
	
The dimension of $\sigma$ is defined to be the dimension of $\operatorname{span}(\sigma) \subset \mathbb{R}^{n}$. Note also that $\sigma$ is called simplicial if the generating vectors $\nu_{1}, \dots ,\nu_{k} \in \mathbb{Z}^{n}$ can be chosen linearly independent. This implies that $k=\operatorname{dim}(\sigma)$.
\begin{definition}
We call a cone $\tau$ a (proper) face of cone $\sigma$ if it is spanned by a (proper) subset of $\{ \nu_{1}, \dots ,\nu_{k} \}$ and contained in the topological boundary of $\sigma$. In this case, we write $\tau \preceq \sigma$ ($\tau \prec \sigma$ if $\tau$ is a proper face of $\sigma$). 
\end{definition}

\begin{definition}[Complete rational polyhedral cone] \label{ConeComplex}
	A \textbf{complete rational polyhedral cone complex} (\textbf{complete fan}) is a nonempty set $\Sigma$ of proper rational polyhedral cones in $\mathbb{R}^{n}$ satisfying the following conditions:
	\begin{enumerate}
		\item If $\tau$ is a face of a cone $\sigma \in \Sigma$, then $\tau \in \Sigma$.
		\item If $\sigma , \sigma^{\prime} \in \Sigma$, then $\sigma \cap \sigma^{\prime}$ is a face of both $\sigma$ and $\sigma^{\prime}$.
		\item $\bigcup_{\sigma \in \Sigma} \sigma = \mathbb{R}^{n}$.
	\end{enumerate}  
\end{definition} 
Note that $\{0\}$ is in $\Sigma$ as $\Sigma$ contains at least one cone and $\{0\}$ is a face on any cone.
\begin{definition}
Let $\Sigma$ be a complete rational polyhedral cone complex in $\mathbb{R}^n$. 
Reversing the inclusions and the dimensional grading in the face lattice of $\Sigma$, we obtain 
the face lattice of an \textit{abstract} polyhedron which can be realized in 
$\mathbb{R}^{n}$ as a regular (polyhedral) cell complex. We denote this abstract 
polyhedron by $\mathcal{P}(\Sigma)$ and its geometrical realization by 
$\vert \mathcal{P}(\Sigma) \vert$. The abstract polyhedron $\mathcal{P}(\Sigma)$ 
is called \textbf{ the dual polyhedron} or \textbf{the dual polytope}. If the fan 
$\Sigma$ is understood we will simply write $\mathcal{P}$. By abuse of notation 
we will frequently write $\mathcal{P}$ for $\vert \mathcal{P} \vert$. The dimension of $\mathcal{
P}(\Sigma)$ is $n$, since the top dimensional cell of $\mathcal{
P}(\Sigma)$ is associated to the cone $\{0\}$ and it is $n$-dimensional. 
\end{definition}

\begin{remark}\label{Bijection}
The faces in $\mathcal{P}(\Sigma)$ are in one-to-one correspondence 
(in complementary dimension) with the cones in $\Sigma$. This defines a 
bijection $\delta: \mathcal{P} \longrightarrow \Sigma$ and we denote the dual 
cone to $\tau \in \mathcal{P}$ by $\delta(\tau)\in \Sigma$. In particular $\Sigma$ can be reconstructed from $\mathcal{
P}$. From that perspective we shall also write $\Sigma= \Sigma_{\mathcal{P}}$. 
\end{remark}

\begin{definition} 
A CW complex is called \textbf{regular} if its characteristic maps can 
be chosen to be embeddings.
\end{definition}

\begin{definition}{} 
	The 0-dimensional faces of $\mathcal{P}$ are called \textbf{vertices}. The 1-dimensional faces are called \textbf{edges} and \textbf{facets} are faces with co-dimension 1. The set
	$\mathcal{P}^{i} = \{ F: \: F \: \textit{is a face of} \: \mathcal{P} , \: \dim(F) \leq i \}$ is the i-skeleton of $\mathcal{P}$.
\end{definition}

\noindent \textbf{Regular CW structure of the dual polytope $\mathcal{P}$}. \\
Note that $\mathcal{P}$ is homeomorphic to an $n$-disc $\mathcal{D}^{n}$, 
where $\dim(\mathcal{P})=n$. The structure of the underlying fan induces a regular CW structure on $\mathcal{P}$ as follows: \\
The vertex $\{ 0 \} \in \Sigma$ is dual to the interior of $\mathcal{P}$, represented by an $n$-dimensional cell in the CW structure. Each $k$-dimensional cone in $\Sigma$ is dual to an $(n-k)$-dimensional face of $\mathcal{P}$, represented by an $(n-k)$-dimensional cell in the CW structure. Note also that the boundary maps are induced from the inclusion of faces in $\Sigma$ (or dually in $\mathcal{P}$). This gives us a CW complex on $\mathcal{P} \cong \mathcal{D}^{n}$, which we will use later.
Although we do not require the exact form of the boundary operators of $\mathcal{P}$ in our preceding discussion, we briefly describe the attaching maps. 
As usual, we start with a discrete set, $X^{0}$, whose points represent the vertices of $\mathcal{P}$. We glue each 1-dimensional cell to its topological boundary in $\mathcal{P}$, which consists of its neighboring 0-dimensional cells.
Inductively, we attach each $k$-cell, representing a $k$-dimensional face of $\mathcal{P}$, to its lower-dimensional neighboring faces. Note that $\vert \mathcal{P} \vert - \operatorname{int}(\mathcal{P}) \cong \mathcal{S}^{n-1}$. In the last step we attach $\operatorname{int}(\mathcal{P})$, represented by the $n$-cell, to its topological boundary $\vert \mathcal{P} \vert - \operatorname{int}(\mathcal{P})$. Note that due to the fact that each $k$-dimensional face of $\mathcal{P}$ is homeomorphic to $\mathcal{D}^{k}$, each characteristic map can be chosen to be an embedding. Thus, the above CW structure is a regular CW structure.

Let $T^{n}=$\scalebox{1.3}{ $\sfrac{\mathbb{R}^{n}}{\mathbb{Z}^{n}}$}$ 
\cong \overbrace{ \mathcal{S}^{1} \times \dots \times \mathcal{S}^{1}}^{n}$ 
be an $n$-torus. The projection map $\pi: \mathbb{R}^{n} \longrightarrow  
\text{\scalebox{1.3}{ $\sfrac{\mathbb{R}^{n}}{\mathbb{Z}^{n}} $}}$ maps a 
\textbf{rational} $k$-dimensional linear subspace $\mathbf{V} \subset \mathbb{R}^{n}$ 
to a compact subtorus $\pi(\mathbf{V}) \subset T^{n}$. The \textbf{rationality} here 
means that $\mathbf{V}$ has a basis in $\mathbb{Z}^{n}$. Now each affine $k$-plane 
parallel to $\mathbf{V}$  in $\mathbb{R}^{n}$ determines a subtorus ``parallel'' to 
$\pi(\mathbf{V})$. Collapsing each of these parallel subtori to a point will 
give us a compact subspace $\text{\scalebox{1.3}{$\sfrac{T^{n}}{\pi(\mathbf{V})}$} } \subset T^{n}$, which is obviously homeomorphic to $T^{n-k}$.\\

\noindent \textbf{Toric varieties}. \\
In this work, we mainly study compact toric varieties, from a topological perspective.   
Hence, we restrict ourselves to toric varieties associated with complete fans. 
This ensures that the resulting toric variety is compact. Compact toric varieties are 
even-dimensional topological pseudomanifolds with only even dimensional strata.
We recall the description of toric varieties as given by Yavin in \cite{yavin1990intersection}. \\ 
Topologically, we can construct the toric variety associated with a \textbf{complete rational cone complex} $\Sigma$ (\textbf{complete fan}) as follows: \\
Let $\mathcal{P}$ be the dual polytope of $\Sigma$ and $\overline{X}=\mathcal{P} \times T^{n}$. Each $\sigma_{D}=\delta(\sigma)$ , a $k$-dimensional face of $\mathcal{P}$, is dual to an $(n-k)$-dimensional cone $\sigma$ in $\Sigma$. Let $\mathbf{V}_{\sigma}$
be the linear span of $\sigma$ in $\mathbb{R}^{n}$. Now if $x \in \operatorname{int}(\sigma_{\operatorname{D}})$ we collapse $\{ x \} \times T^{n}$ to $\{ x \} \times \text{\scalebox{1.3}{$\sfrac{T^{n}}{\pi(\mathbf{V}_{\sigma})}$}}$. This procedure can be done for each face of $\mathcal{P}$ and the resulting space $X_{\mathcal{P}}$ or $(X_{\Sigma})$  is called the \textbf{toric variety} associated with $\Sigma$ (or $\mathcal{P}$). \\
Note that for each $x \in \mathcal{P}$ there is a unique $\sigma_{D}$ such that $x \in \operatorname{int}(\sigma_{D})$. Thus the above construction is well-defined.

\begin{example}[$n=2$]\label{2D}
	Let $n,m \in \mathbb{N}_{\geq 0}$, $(n,m) \neq 0$ and n and m be relatively prime.
	Consider the following fan:
		\begin{align*}
		\Sigma_{2} &= \big\{ \{ 0 \}, \tau_{1},\tau_{2},\tau_{3},\tau_{4},\sigma_{12},\sigma_{23},\sigma_{34},\sigma_{41} \big\} \\
		& \text{where} \\
		\tau_{1} &= \big\{ x(n,m) \; \vert \; x \in \mathbb{R}_{\geq 0} \big\}, \\
		\tau_{2} &= \big\{ x(-m,n) \; \vert \; x \in \mathbb{R}_{\geq 0} \big\}, \\
		\tau_{3} &= \big\{ x(-n,-m) \; \vert \; x \in \mathbb{R}_{\geq 0} \big\}, \\
		\tau_{4} &= \big\{ x(m,-n) \; \vert \; x \in \mathbb{R}_{\geq 0} \big\}, \\
		\sigma_{12} &= \big\{ x(n,m)+ y(-m,n) \;  \vert \; x,y \in \mathbb{R}_{\geq 0}\big\};
		\end{align*}
see Figure \ref{fig:Fan2D}.
	\begin{figure}[tbh]
		\centering
		\includegraphics[width=.55\linewidth]{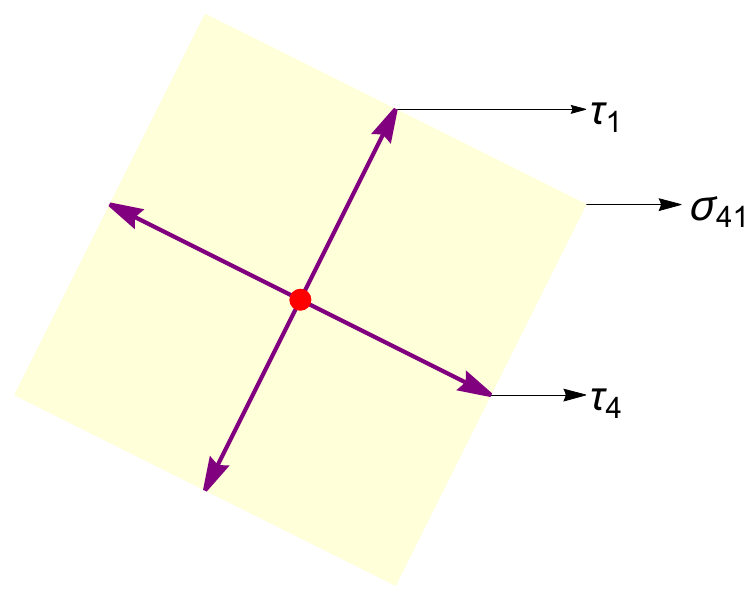}
		\caption{The 2-dimensional complete fan $\Sigma_{2}$. }
		\label{fig:Fan2D}
	\end{figure}

The cones $\sigma_{23}$, $\sigma_{34}$ and $\sigma_{41}$ are defined similarly. 
The dual polytope is homeomorphic to an 2-dimensional convex polygon with four 
1-dimensional
	faces or simply a square. Thus we can write 
	$\overline{X}= \mathcal{I} \times \mathcal{I} \times \mathcal{T}^{2}$. 
	The torus $\mathcal{T}^{2}$ can be written as 
	$\pi(\mathbf{V}_{\tau_{1}}) \times \pi(\mathbf{V}_{\tau_{2}})$. 
	Hence $\overline{X}$ can be rewritten as 
	$\big( \mathcal{I} \times \pi(\mathbf{V}_{\tau_{1}})  \big) \times  
	 \big( \mathcal{I} \times \pi(\mathbf{V}_{\tau_{2}})\big) $. 
	 Note that due to the structure of $\Sigma_{2}$ applying the topological 
	 construction of toric varieties yields that either $\pi(\mathbf{V}_{\tau_{1}})$ 
	 or $\pi(\mathbf{V}_{\tau_{2}})$ is subject to collapses on 1-dimensional faces 
	 of $\mathcal{P}(\Sigma_{2})$. Based on this observation it is easy to show 
	 that $X_{\Sigma_{2}} \cong X_{\Sigma_{1}} \times X_{\Sigma_{1}} \cong \mathbb{P}^{1} 
	 \times \mathbb{P}^{1}$, where $\Sigma_{1}= \big\{ \{0\}, \tau_{1}, \tau_{3} \big\}$;
	 see Figure \ref{fig:Toric2D}.
\begin{figure}[!h]
	\centering
	\includegraphics[width=.55\linewidth]{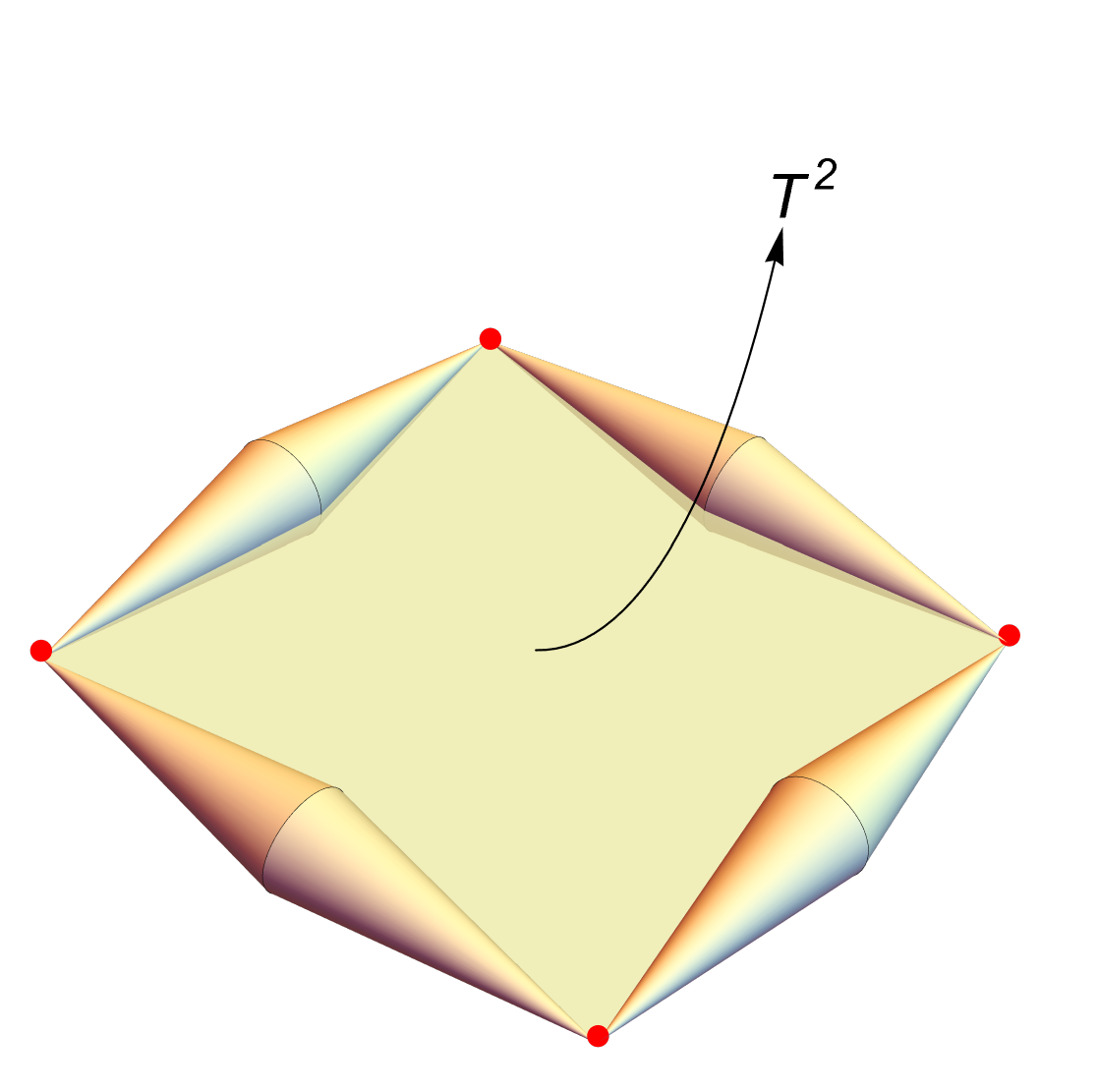}
	\caption{The toric variety associated to $\Sigma_{2}$ is homeomorphic 
	 to $\mathcal{S}^{2} \times \mathcal{S}^{2}$.}
	\label{fig:Toric2D}
\end{figure}
\end{example}

\begin{example}[$n=3$]
	Consider the following fan, $\Sigma_{3}$, where $\tau_{z}$ is generated by $\hat{i}_{z}=(0,0,1)$. The cones $\tau_{x}$ and $\tau_{y}$ are generated by $\hat{i}_{x}$ and $\hat{i}_{y}$, respectively. The fourth 1-dimensional cone of $\Sigma_{3}$ is generated by $i_{4}=(-1,-1,-1)$. The generators of the 2-dimensional cone $\sigma_{xz}$ are $\hat{i}_{x}$ and $\hat{i}_{z}$ and the five remaining 2-dimensional cones are constructed in a similar manner. Finally, $\omega_{xyz}$ is generated by $\hat{i}_{x}$, $\hat{i}_{y}$ and $\hat{i}_{z}$. 
	The three remaining 3-dimensional cones are constructed similarly;
	see Figure \ref{fig:Fan3D}. It is easy to show that the dual polytope is a tetrahedron. Applying the topological construction of toric varieties on $\Sigma_{3}$ yields a toric variety which is homeomorphic to $\mathbb{P}^{3}$ (Figure \ref{fig:Toric3D}).
\end{example}

\begin{figure}
	\centering
	\includegraphics[width=.65\linewidth]{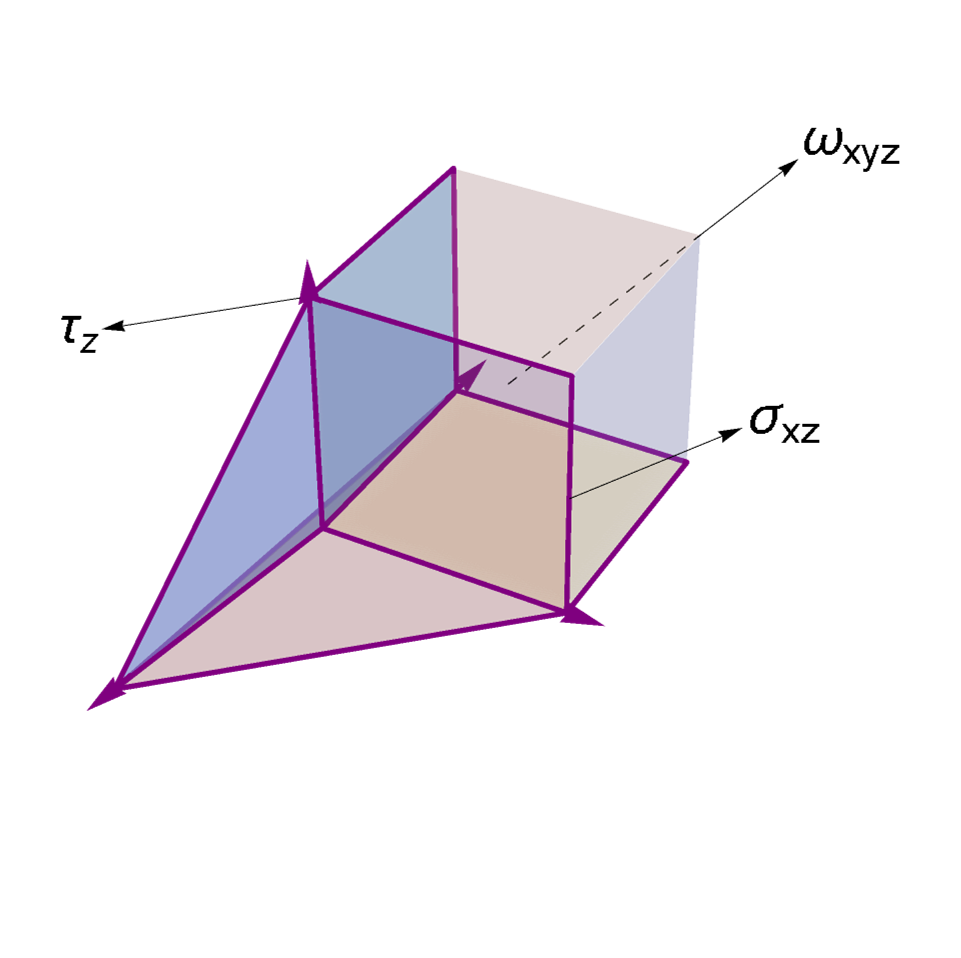}
	\caption{The complete fan $\Sigma_{3}$ in $\mathbb{R}^{3}$.}
	\label{fig:Fan3D}
\end{figure}

\begin{figure}
	\centering
	\includegraphics[width=.65\linewidth]{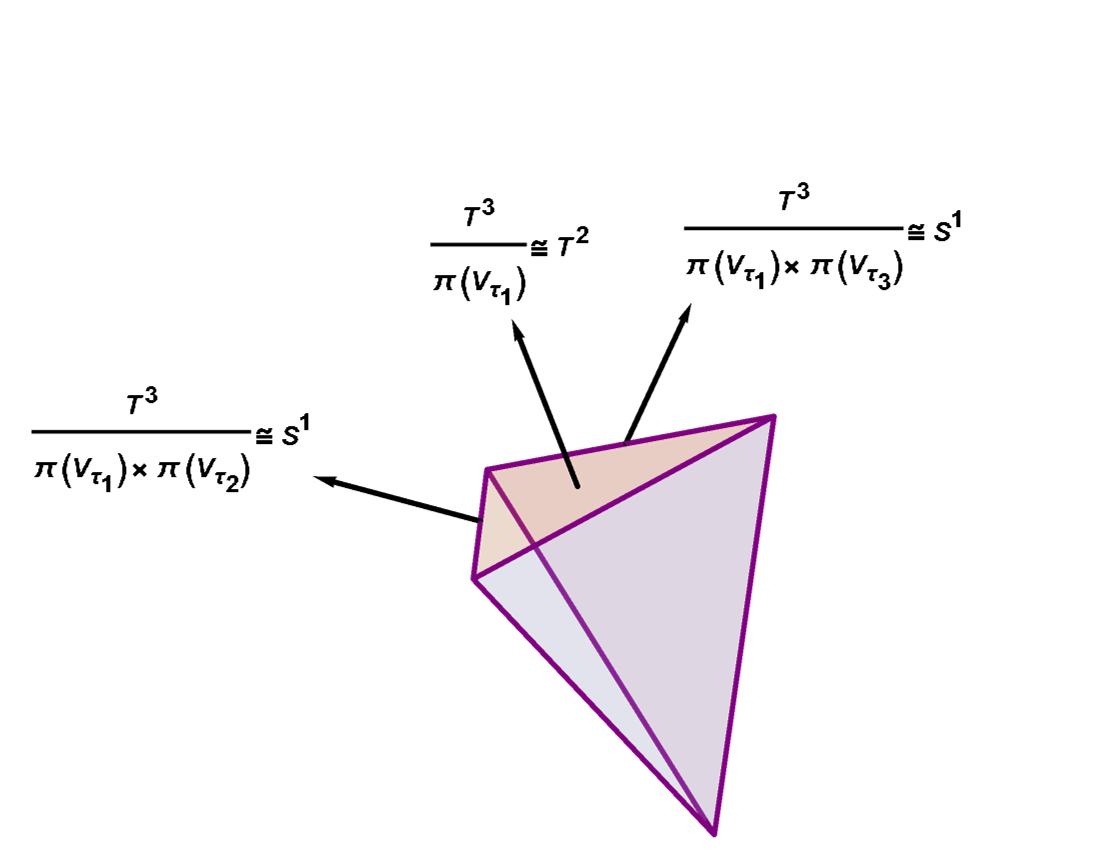}
	\caption{The toric variety $X_{\Sigma_{3}}$ is homeomorphic to $\mathbb{P}^{3}$. }
	\label{fig:Toric3D}
\end{figure}

\section{Topology of Toric Varieties}

In this section, we study toric varieties from the perspective of stratified pseudomanifolds. First of all, we will recall that toric varieties are indeed pseudomanifolds. In the next step, we will verify that the link bundles of toric varieties are topologically trivial. Toric varieties are normal pseudomanifolds. In other words, the link of the stratum with depth 1 is always homeomorphic to a circle $\mathcal{S}^1$.\\

\begin{definition}\label{pseudomanifolds}
	We define a \textbf{topologically stratified space} inductively on dimension. A 0-dimensional topologically stratified space $X$ is a countable set with the discrete topology. For $m > 0$ an \textbf{$m$-dimensional topologically stratified space} is a para-compact Hausdorff topological space $X$ equipped with a filtration
	\begin{align*}
		X=X_{m} \supseteq X_{m-1} \supseteq \dots \supseteq X_{1} \supseteq X_{0} \supseteq X_{-1}= \emptyset
	\end{align*}
	 by closed subsets $X_{j}$ such that if $x \in X_{j}-X_{j-1}$ there exists a neighborhood $\mathcal{N}_{x}$ of $x$ in $X$, a compact $(m-j-1)$-dimensional topologically stratified space $\mathcal{L}$ with filtration
	\begin{align*}
		\mathcal{L}=\mathcal{L}_{m-j-1} \supseteq \dots \supseteq \mathcal{L}_{1} \supseteq  \mathcal{L}_{0} \supseteq \mathcal{L}_{-1} = \emptyset,
	\end{align*}  
	and a homeomorphism $\phi : \mathcal{N}_{x} \longrightarrow \mathbb{R}^{j} \times \mathcal{C}(\mathcal{L}),$
	where $\mathcal{C}(\mathcal{L})$ is the open cone on $\mathcal{L}$, such that $\phi$ takes $\mathcal{N}_{x} \cap X_{j+i+1}$ homeomorphically onto
	\begin{align*}
		\mathbb{R}^{j} \times \mathcal{C}(\mathcal{L}_{i}) \subseteq \mathbb{R}^{j} \times \mathcal{C}(\mathcal{L})
	\end{align*}
	for $m-j-1 \geq i \geq 0$, and $\phi$ takes $\mathcal{N}_{x} \cap X_{j}$ homeomorphically onto
	\begin{align*}
		\mathbb{R}^{j} \times \{ \text{vertex of }\; \mathcal{C}(\mathcal{L}) \}.
	\end{align*}
\end{definition}

\begin{remark}\label{linkandstratum}
	It follows that $X_{j}-X_{j-1}$ is a $j-$dimensional topological manifold. (The empty set is a manifold of any dimension.) We call the connected components of these manifolds the \textbf{strata} of $X$. Any $\mathcal{L}$ that satisfies the above properties is referred to as a \textbf{link} of the stratum at $x$.   
\end{remark}
\begin{definition}\label{PSMFD}
	An \textbf{$m$-dimensional topological pseudomanifold} is a para-compact Hausdorff topological space $X$ which possesses a topological stratification such that $X_{m-1}=X_{m-2}$
	and $X-X_{m-1}$ is dense in X.
\end{definition}
\begin{definition}\label{DefSing}
	We call a stratum (homologically) \textbf{singular} if none of its links is a homology sphere. A stratum is (homologically) \textbf{rationally singular} if none of its links is a rational homology sphere.
\end{definition}
Now, we recall that there is a natural stratification of toric varieties. We also give an algebraic description of the situation. \\
Let $X$ be a toric variety over a polytope $\mathcal{
P}$. Consider the map $X \xrightarrow{\;\;\; p \;\;\;}\mathcal{P}$ introduced earlier. The preimage $X_{2i} = p^{-1}(\mathcal{P}^{i})$ is a $2i$-dimensional topological space. We claim that the filtration
\begin{align}\label{strtTV}
	X_{\mathcal{P}}=X_{2m} \supset X_{2(m-1)} \supset \dots \supset X_{2} \supset X_{0},
\end{align}
 is a stratification of the toric variety $X_{\mathcal{P}}$. The local homeomorphisms will be defined later. But first we need to give a topological description of links.\\

\noindent\textbf{Construction of links}\label{Triv} \phantom{.} \\ Let $\tau$ be a face of $\mathcal{P}$. Let $\mathcal{M}_{\tau}$ be an abstract polytope, geometrically realized as a subspace of $\mathcal{P}$, and defined as follows: \\
Let $\mathcal{S}_{\tau}=\big\{ \sigma \vert \sigma \in \mathcal{P}, \; \sigma \cap \tau \neq \emptyset \; \text{and} \; \dim(\sigma) > \dim(\tau) \big\}$ be the set of all higher dimensional neighboring faces of $\tau$ in $\mathcal{P}$. \\
We shall define an abstract polytope $\mathcal{M}_{\tau}$ with the following properties: 
\begin{enumerate}
	\item $\forall \sigma \in \mathcal{S}_{\tau}: \exists! \; \gamma_{\sigma} \in \mathcal{M}_{\tau} \; \text{such that} \; \operatorname{int}(\gamma_{\sigma}) \cap \operatorname{int}(\sigma) \neq \emptyset \; \text{with} \; \dim(\gamma_{\sigma})= \dim(\sigma)-(1+\dim(\tau))$ and $\operatorname{int}(\gamma_{\sigma}) \cap \operatorname{int}(\sigma^{\prime})= \emptyset$ if $\sigma^{\prime} \in \mathcal{P}$ and $\sigma^{\prime} \neq \sigma$.
	\item If $\omega, \sigma \in \mathcal{S}_{\tau}$ with $\omega \prec \sigma$, then we demand that $\gamma_{\omega} \prec \gamma_{\sigma}$ and $\gamma_{\sigma} \cap \omega = \gamma_{\omega}$.
\end{enumerate}
Note that $\operatorname{int}(\sigma) \cong \operatorname{int}(\mathcal{D}^{\dim(\sigma)})$. So, the first requirement can be satisfied by embedding $\operatorname{int}(\mathcal{D}^{\dim(\gamma_{\sigma})})$ in $\operatorname{int}(\mathcal{D}^{\dim(\sigma)})$. Yet the second condition describes how the boundary of $\mathcal{D}^{\dim(\gamma_{\sigma})}$ meets the boundary of $\mathcal{D}^{\dim(\sigma)}$. 
We refer to $\mathcal{M}_{\tau}$ as a \textbf{base} of the link of $x \in p^{-1}(\operatorname{int}(\tau)) \subset X_{\mathcal{P}}$. \\
\begin{proposition}\label{Existence2}
	Let $\Sigma$ be a complete fan in $\mathbb{R}^n$ and $\mathcal{
	P}$ its dual polytope. Given any face $\tau \in \mathcal{P}$ there exists an abstract polytope $\mathcal{M}_{\tau}$ satisfying the above conditions $(1)$ and $(2)$ and having a geometrical realization $\vert \mathcal{M}_{\tau} \vert$ as a subpolyhedron of a geometrical realization of $\mathcal{P}$. 
\end{proposition}
\begin{proof}
        First, we show the existence of $\mathcal{M}_{\tau}$.
		Let $\Sigma_{\mathcal{P}}$ be a complete fan and $\mathcal{P}$ be the dual polytope associated with $\Sigma_{\mathcal{P}}$. Let $\delta: \mathcal{P} \longrightarrow \Sigma_{\mathcal{P} }$ be the bijection defined in Remark \ref{Bijection}. Let $\Sigma_{\mathcal{S}_{\tau}}$ be the set of all lower dimensional neighboring cones of $\tau \in \mathcal{P}$ in $\Sigma_{\mathcal{P}}$. Dual to the previous construction, we define the dual fan associated with $\mathcal{M}_{\tau}$ to be a fan in $\mathbb{R}^{\dim(\delta(\tau))-1}$ with the following property:
		\begin{itemize}
			\item $\forall \delta(\sigma) \in \Sigma_{\mathcal{S}_{\tau}}: \: \exists! \delta(\gamma_{\sigma}) \in \Sigma_{\mathcal{M}_{\tau}}$ such that $\delta(\gamma_{\sigma}) \subset \delta(\tau)$ as a cone. In other words, $\delta(\gamma_{\sigma})$ lies in the topological boundary of the cone $ \delta(\tau)$.
		\end{itemize}
	One can verify easily that both sets of conditions are dually equivalent. In other words, there is an order-reversing bijection between $\Sigma_{\mathcal{S}_{\tau}}$ and $\mathcal{S}_{\tau}$. The bijection is the restriction of $\delta^{-1}$ to $\Sigma_{\mathcal{S}_{\tau}}$.
		The set $\Sigma_{\mathcal{M}_{\tau}}$ is a set of cones such that
		\begin{align*}
			0 \leq \dim(\delta(\gamma_{\sigma})) < \dim(\delta(\tau))
		\end{align*}
		Hence, $\Sigma_{\mathcal{S}_{\tau}}$ is an $(\dim(\delta(\tau))-1)$-dimensional cone-complex, embedded in $\mathbb{R}^{\dim(\delta(\tau))}$, a rational subspace of $\mathbb{R}^{\dim(\Sigma_{\mathcal{P}})}$. By cone-complex, we mean that $\Sigma_{\mathcal{M}_{\tau}}$ satisfies the conditions (1) and (2) of Definition \ref{ConeComplex}. However, instead of the third axiom, we have
		\begin{align*}
			\bigcup_{\delta(\gamma_{\sigma}) \in \Sigma_{\mathcal{M}_{\tau}}} \delta(\gamma_{\sigma}) \cong \mathbb{R}^{\dim(\delta(\tau))-1}.
		\end{align*}
		But the crucial point to bear in mind is the following. Reversing the inclusions and grading in $\Sigma_{\mathcal{M}_{\tau}}$ will still result in an abstract polytope, which we can geometrically realize in $\mathbb{R}^{\dim(\delta(\tau)-1)}$. Existence of the topological boundary of a cone yields the existence of $\mathcal{M}_{\tau}$.\\
		Now, we embed $\mathcal{C}(\mathcal{M}_{\tau})$ in $\mathcal{P}$, geometrically.
		For a given point $x \in \operatorname{int}(\tau)$, where $\dim(\tau) \geq 1$, let $\mathcal{V} \cong \mathbb{R}^{\dim(\vert \mathcal{M}_{\tau} \vert)+1}$ be an affine subspace of $\mathbb{R}^{\dim(\vert \mathcal{P} \vert)}$ which is orthogonal to $\tau$ in $\mathbb{R}^{\dim(\vert \mathcal{P} \vert)}$ and $x \in \mathcal{V}$. Note that $\dim( \vert \mathcal{M}_{\tau} \vert)=\dim(\vert \mathcal{P} \vert)-(1+\dim(\tau))$, hence, such $\mathcal{V}$ can always be found. Thus, we have $\sigma_{\gamma} \cap \mathcal{V} \neq \emptyset$ for each $\gamma \in \mathcal{M}_{\tau}$, where $\sigma_{\gamma} \in \mathcal{S}_{\tau}$ with $\operatorname{int}(\gamma) \cap \operatorname{int}( \sigma_{\gamma}) \neq \emptyset$. Now, we choose a geometrical realization of $\gamma$ such that $\gamma \subset \sigma_{\gamma} \cap \mathcal{V}$. Note that $\dim( \sigma_{\gamma} \cap \mathcal{V} )= \dim(\gamma)+1$. Thus, we can find such a geometrical realization consistently. Now, let $\mathcal{C}( \vert \mathcal{M}_{\tau} \vert)=\sfrac{\big( \vert \mathcal{M}_{\tau} \vert \times \mathcal{I} \big)}{\vert \mathcal{M}_{\tau} \vert \times \{1\}}$ be the cone of $\vert \mathcal{M}_{\tau} \vert$. We embed $\mathcal{C}( \vert \mathcal{M}_{\tau} \vert)$ into $\vert \mathcal{P} \vert$ as follows: \\
		Let $v \in \mathcal{C}( \vert \mathcal{M}_{\tau} \vert)$ be the vertex of $\mathcal{C}( \vert \mathcal{M}_{\tau} \vert)$ and $\theta: \mathcal{C}( \vert \mathcal{M}_{\tau} \vert) \longrightarrow \vert \mathcal{P} \vert$ a map such that for each $ \gamma \in \mathcal{M}_{\tau},$
		\begin{align*}
			&\theta \big( \operatorname{int}( \gamma) \times [0,1) \big) \subset \operatorname{int}(\sigma_{\gamma}) \cap \mathcal{V},  \\ 
			& \theta \big( \gamma \times [0,1) \big) \subset \Big( \operatorname{int}(\sigma_{\gamma}) \cup \Big( 
			\bigcup_{\substack{\omega_{\gamma} \in \mathcal{S}_{\tau}\\ \omega_\gamma \prec \sigma_{\gamma}}} \operatorname{int}(\omega_{\gamma}) \Big) \Big) \cap \mathcal{V} \\
			&\theta \big( \gamma \times \{0\} \big) \cong \vert \mathcal{M}_{\tau} \vert \cap \operatorname{int}(\sigma_{\gamma}) \cong \operatorname{int}(\gamma).
		\end{align*}
		Note that for each $\eta, \gamma \in \mathcal{M}_{\tau}$ if $\gamma \prec \eta$ then we have $\sigma_{\gamma} \prec \sigma_{\eta}$ in $\mathcal{S}_{\tau}$. We require that 
		\begin{align}\label{thetareq}
			\theta \big( \eta \times [0,1) \big) \cap \sigma_{\gamma}=\theta \big( \gamma \times [0,1) \big). 
		\end{align}
		As the last requirement, we want $\theta\big( \operatorname{int}(\gamma) \times [0,1) \big) \hookrightarrow \operatorname{int}(\sigma_{\gamma})$ to be a topological embedding. Note that this can always be fulfilled because $\gamma \cong \mathcal{D}^{\dim(\gamma)}$ and $\sigma_{\gamma} \cong \mathcal{D}^{\dim(\sigma_{\gamma})}$. At last, one should bear in mind that Equation (\ref{thetareq}) ensures that $\theta$ is also an embedding on the topological boundary of $\sigma_{\gamma} \cap \mathcal{V}$ for each $\sigma_{\gamma} \in \mathcal{S}_{\tau}$. We set $\theta(v)=x$. Note that because of $\gamma \subset \sigma_{\gamma} \cap \mathcal{V}$ and the fact that we can choose $\theta\big( \gamma \times [0,1) \big) \hookrightarrow \operatorname{int}(\sigma_{\gamma})$ as an embedding, it is possible to choose $\theta$ continuous on $\mathcal{C}(\vert \mathcal{M}_{\tau} \vert)$ and thus an embedding of $\mathcal{C}(\vert \mathcal{M}_{\tau} \vert)$ into $\vert \mathcal{P} \vert$. Now let $\mathcal{V}^{\perp} \subset \mathbb{R}^{\dim(\vert \mathcal{P} \vert)}$ be an orthogonal affine subspace to $\mathcal{V}$ such that $\tau \subset \mathcal{V}^{\perp}$. Choose $\mathcal{U}$ a neighborhood of $x$ in $\vert \mathcal{P} \vert$ such that $\mathcal{U} \cap \mathcal{V}^{\perp} \cong \operatorname{int}(\mathcal{D}^{\dim(\tau)}) \cong \mathbb{R}^{\dim(\tau)}\cong \operatorname{int}(\tau)$ and $\mathcal{U} \cap \mathcal{V} \cong \mathcal{C}(\vert \mathcal{M}_{\tau} \vert)$. Thus, we have 
		\begin{align} \label{decom}
			\mathcal{U} \cong \mathbb{R}^{\dim(\tau)} \times \mathcal{C}(\vert \mathcal{M}_{\tau} \vert).
		\end{align}
	
\end{proof}
\begin{remark}
	Consider $\mathcal{M}_{\tau}$ and $\mathcal{M}^{\prime}_{\tau}$ such that they both satisfy the above conditions. Then for each $\sigma \in \mathcal{S}_{\tau}$ and $\gamma_{\sigma} \in \mathcal{M}_{\tau}$ there exists a $\gamma^{\prime}_{\sigma} \in \mathcal{M}^{\prime}_{\tau}$ with $\dim(\gamma_{\sigma})=\dim(\gamma^{\prime}_{\sigma})$. Note that the construction of $\mathcal{S}_{\tau}$ and the uniqueness in the first condition implies that $\operatorname{card}(\mathcal{M}^{i}_{\tau}-\mathcal{M}^{i-1}_{\tau})$ i.e. the number of $i$-dimensional faces of $\mathcal{M}_{\tau}$ is equal to $\operatorname{card}(\mathcal{M^{\prime}}^{i}_{\tau}-\mathcal{M^{\prime}}^{i-1}_{\tau})$. Thus, there is a bijection between $\mathcal{M}_{\tau}$ and $\mathcal{M^{\prime}}_{\tau}$.
\end{remark}
Now, let $\vert \mathcal{M}_{\tau} \vert \subset \vert \mathcal{P} \vert$ be a geometrical realization of $\mathcal{M}_{\tau}$ and $n=\dim(\vert \mathcal{P} \vert)$. Recall the bijection $\delta: \mathcal{P} \longrightarrow \Sigma_{\mathcal{P}}$ that we introduced earlier. For each $\gamma \in \mathcal{M}_{\tau}$ choose $\sigma_{\gamma} \in \mathcal{S}_{\tau}$ such that $\gamma \cap \sigma_{\gamma} \neq \emptyset$. Note $\tau \prec \sigma_{\gamma}$ and thus $\delta (\sigma_{\gamma}) \prec \delta(\tau)$ in $\Sigma_{\mathcal{P}}$. This implies that $\pi(\delta (\sigma_{\gamma})) \subset \pi(\delta(\tau))$ in $T^{n}$. We construct $\mathcal{L}_{\tau}$, the link of a point $x \in p^{-1}(\operatorname{int}(\tau))$, by means of the following map: \\
\begin{align}
	\mathcal{M}_{\tau} \times T^{n} &\longrightarrow \mathcal{L}_{\tau} \nonumber \\
	\{y\} \times T^{n}  &\longmapsto \{y\} \times \scalebox{1.3}{ $\sfrac{T^{n}}{\Big(\pi (\delta(\sigma_{\gamma}))\times \big( \sfrac{T^{n}}{\pi (\delta(\tau))} \big) \Big)}$},
\end{align}\label{Linkmap}
where $y \in \operatorname{int}(\gamma)$. Note that similar to the topological construction of toric varieties, 
\[ T^{n}/\Big(\pi (\delta(\sigma_{\gamma}))\times \big( \sfrac{T^{n}}{\pi (\delta(\tau))} \big) \Big)\] is defined by collapsing $\pi (\delta(\sigma_{\gamma}))$ and each parallel torus to $\pi (\delta(\sigma_{\gamma}))$ in $T^{n}$ to a point and then collapsing each parallel torus to $\sfrac{T^{n}}{\pi (\delta(\tau))}$ in $\scalebox{1}{ $\sfrac{T^{n}}{\pi (\delta(\sigma_{\gamma})) }$}$ to a point. Due to the previous consideration, $\scalebox{1.3}{ $\sfrac{T^{n}}{\Big(\pi (\delta(\sigma_{\gamma}))\times \big( \sfrac{T^{n}}{\pi (\delta(\tau))} \big) \Big)}$}$ is well-defined.
\begin{remark}
	Note that for each $y \in \operatorname{int}(\gamma)$, $\scalebox{1.3}{ $\sfrac{T^{n}}{\Big(\pi (\delta(\sigma_{\gamma}))\times \big( \sfrac{T^{n}}{\pi (\delta(\tau))} \big) \Big)}$} \subseteq \sfrac{T^{n}}{\pi (\delta(\sigma_{\gamma}))} $ and hence $\mathcal{L}_{\tau} \subset X$, where $n=\dim(\vert \mathcal{P} \vert)$.
\end{remark}

\begin{remark}
The dimension relation 
$\dim(\vert \mathcal{M}_{\tau} \vert)=\dim(\vert \mathcal{P} \vert)-(1+\dim(\tau))$ holds. 
This comes from the fact that $\operatorname{int}(\vert \mathcal{P} \vert)$ considered as a face in $\mathcal{P}$, which is dual to $\{0\}\in\Sigma_{\mathcal{P}}$, is a higher dimensional neighboring face of each $\tau \in \mathcal{P}$. So $\dim(\mathcal{L}_{\tau})=\dim(\vert \mathcal{P} \vert)-(1+ \dim(\tau))+\dim(\pi(\delta(\tau))$. Note that $dim(\tau)=\dim(\vert \mathcal{P} \vert) - \dim(\delta(\tau))$. Thus, we have
	\begin{align*}
		\dim(\mathcal{L}_{\tau})= \dim(X_{\mathcal{P}})-(1+2\dim(\tau)).
	\end{align*}
\end{remark}

\begin{remark}\label{StratProof}
	There is a natural projection $\mathcal{L}_{\tau} \xrightarrow{\:\:\: p_{\mathcal{L}} \:\:\:} \mathcal{M}_{\tau}$. Similarly, we define $\big( \mathcal{L}_{\tau} \big)_{2i+1}=p_{\mathcal{L}}^{-1}(\mathcal{M}^{i}_{\tau})$.
\end{remark}
At this point, we want to recall that there is a natural stratification of toric varieties. 
We will also give an algebraic description of the strata. 
We claim that the filtration given in Expression (\ref{strtTV}), 
is a stratification in the sense of Definition \ref{pseudomanifolds}. \\
\begin{proposition}\label{trivlink}
	Let $\Sigma$ be a complete fan and $\mathcal{P}$ the associated dual polytope. Then $X_{\mathcal{P}}$, the associated toric varieties with $\mathcal{P}$, is a topological pseudomanifold with a trivial link bundle. 	
\end{proposition}
\begin{proof}
	We can construct $\mathcal{U}$ in the proof of Proposition \ref{Existence2} as follows. We embed $\mathcal{C}(\vert \mathcal{M}_{\tau} \vert)$ in $\mathcal{V}$ as in that proof. Now, we embed $\mathring{\mathcal{D}}^{\dim(\tau)} \times \mathcal{C}(\vert \mathcal{M}_{\tau} \vert)$ in $\mathcal{V} \times \mathcal{V}^{\perp} \cong \mathbb{R}^{\dim(\vert \mathcal{P} \vert)}$ such that
	$
		\big( \mathring{\mathcal{D}}^{\dim(\tau)} \times \mathcal{C}(\vert \mathcal{M}_{\tau} \vert) \big) \cap \mathcal{V}^{\perp} = \mathring{\mathcal{D}}^{\dim(\tau)}.
	$
	Finally, we set
	$
		\mathcal{U} \cong \mathring{\mathcal{D}}^{\dim(\tau)} \times \mathcal{C}(\vert \mathcal{M}_{\tau} \vert) \cong \mathbb{R}^{\dim(\tau)} \times \mathcal{C}(\vert \mathcal{M}_{\tau} \vert). 
	$

For each $\sigma_{\gamma} \in \mathcal{S}_{\tau}$, we have $\mathcal{U} \cap \operatorname{int}(\sigma_{\gamma}) \cong \big(\mathbb{R}^{\dim(\tau)} \cap \operatorname{int}(\sigma_{\gamma}) \big) \times \big(\mathcal{C}(\vert \mathcal{M}_{\tau}\vert) \cap \operatorname{int}(\sigma_{\gamma}) \big)$. Keep also in mind that $\big( \bigcup_{\substack{\gamma \in \mathcal{M}_{\tau}}} \operatorname{int}(\sigma_{\gamma}) \big) \cup \operatorname{int}(\tau)$ is an open cover of $\mathcal{U}$. Thus
\begin{align*}
	p^{-1}(\mathcal{U})=\bigcup_{\gamma \in \mathcal{M}_{\tau}} \Big( (\operatorname{int}(\sigma_{\gamma}) \cap \mathcal{U})\times \sfrac{T^{n}}{\pi(\delta(\sigma_{\gamma}))} \Big) \cup \Big( \big( \operatorname{int}(\tau) \cap \mathcal{U} \big) \times \sfrac{T^{n}}{\pi(\delta(\tau))} \Big).
\end{align*}
Recall that $\tau \prec \sigma_{\gamma}$ in $\mathcal{P}$, hence $\delta(\sigma_{\gamma}) \prec \delta(\tau)$ which implies $\pi(\delta(\sigma_{\gamma})) \subset \pi(\delta(\tau))$ and $\sfrac{T^{n}}{\pi(\delta(\tau))} \subset \sfrac{T^{n}}{\pi(\delta(\sigma_{\gamma}))}$. Now, define $T_{\tau} = \sfrac{ T^{n} }{ \pi(\delta(\tau)) }$. Then, we have 
\begin{align*}
	p^{-1}(\mathcal{U})= T_{\tau} \times \Bigg( \bigcup_{\gamma \in \mathcal{M}_{\tau}} \Big( (\operatorname{int}(\sigma_{\gamma}) \cap \mathcal{U})\times \sfrac{T^{n}}{\big( \pi(\delta(\sigma_{\gamma})) \times T_{\tau} \big)} \Big) \cup  \big( \operatorname{int}(\tau) \cap \mathcal{U} \big) \Bigg) .
\end{align*} 
Using Homeomorphism (\ref{decom}), we arrive at
\begin{align*}
	p^{-1}(\mathcal{U})= \big( T_{\tau} \times \mathbb{R}^{\dim(\tau)} \big) \times \Bigg( \bigcup_{\gamma \in \mathcal{M}_{\tau}} \Big( \theta\big(\operatorname{int}(\gamma) \times [0,1) \big)\times \sfrac{T^{n}}{\big( \pi(\delta(\sigma_{\gamma})) \times T_{\tau} \big)} \Big) \cup  v \Bigg) .
\end{align*}
Using the fact that $\theta$ is a topological embedding gives us
\begin{align*}
	p^{-1}(\mathcal{U}) \cong \big( T_{\tau} \times \mathbb{R}^{\dim(\tau)} \big) \times \Bigg(  [0,1) \times \bigcup_{\gamma \in \mathcal{M}_{\tau}} \Big( \operatorname{int}(\gamma) \times \sfrac{T^{n}}{\big( \pi(\delta(\sigma_{\gamma})) \times T_{\tau} \big)} \Big) \cup  v \Bigg) .
\end{align*}
Thus, we have
\begin{align}
	p^{-1}(\mathcal{U}) \cong \big( T_{\tau} \times \mathbb{R}^{\dim(\tau)} \big) \times  \mathcal{C} \Big(  \bigcup_{\gamma \in \mathcal{M}_{\tau}}  \operatorname{int}(\gamma) \times \sfrac{T^{n}}{\big( \pi(\delta(\sigma_{\gamma})) \times T_{\tau} \big)} \Big)  .
	\label{triva}
\end{align}
This homeomorphism gives us more than just the required \emph{local} triviality. 
Recall that $X_{2j}-X_{2(j-1)}=p^{-1}(\mathcal{P}_{j})-p^{-1}( \mathcal{P}_{j-1})$ is simply the disjoint union of preimages of the interior of all $j-$dimensional faces of $\mathcal{P}$. Hence, we have
\begin{align*}
	X_{2j}-X_{2(j-1)}&= \bigsqcup_{\substack{\tau \in \mathcal{P} \\ \dim(\tau)=j}} p^{-1}(\operatorname{int}(\tau)) 
	  =\bigsqcup_{\substack{\tau \in \mathcal{P} \\ 
	\dim(\tau)=j}} \operatorname{int}(\tau) \times \sfrac{T^{n}}{\pi(\delta(\tau))} \\
	&\cong\bigsqcup_{\substack{\tau \in \mathcal{P} \\ \dim(\tau)=j}} \mathbb{R}^{\dim(\tau)} \times T_{\tau}.
\end{align*}\label{orbit}
This means that $\mathbb{R}^{\dim(\tau)} \times T_{\tau}$ is a connected component of $X_{2j}-X_{2(j-1)}$. Thus, if we show that the given filtration endows $X_{\mathcal{P}}$ with a stratification then the link bundle is trivial. Now, consider the filtration of $\mathcal{L}_{\tau}$ that we introduced earlier
\begin{align*}
	\mathcal{L}=\mathcal{L}_{2m+1} \supset \mathcal{L}_{2(m-1)+1} \supset \dots \supset \mathcal{L}_{1}.
\end{align*}
Let $x \in X_{j}-X_{j-1}$ and $x \in \operatorname{int}(\tau)$. Thus, we can write $j=2\dim(\tau)$. Choose $\mathcal{U}$ as described above. We want to investigate the intersection of $p^{-1}(\mathcal{U})$ with $X_{2\dim(\tau)+i+1}$. Note that in the above filtration of toric varieties we have only even-dimensional\footnote{Note that $X_{2j}$ can be identified with $X_{2j+1}$. This means that $X_{2j+1}-X_{2j}=\emptyset$ and thus the required conditions in Definition \ref{pseudomanifolds} are trivially fulfilled.} topological spaces $X_{j}$. Thus, we can write $i+1=2l$ with $l \in \mathbb{N}_{>0}$. Hence, we have
\begin{align*}
	&p^{-1}(\mathcal{U}) \cap X_{2(\dim(\tau)+l))}= \\
	&\Big(\bigcup_{\substack{\sigma \in \mathcal{P} \\ \dim(\sigma) \leq \dim(\tau)+l}} \big( \operatorname{int}(\sigma) \times \sfrac{T^{n}}{\pi(\delta(\sigma))} \big) \Big) \bigcap \\
	& \Big[ \big(\mathbb{R}^{\dim(\tau)} \times T_{\tau}\big) \times \Big( \big(  [0,1) \times \bigcup_{\substack{\gamma \in \mathcal{M}_{\tau}}} \operatorname{int}(\gamma) \times \sfrac{T^{n}}{\big(\pi(\delta(\sigma_{\gamma})) \times T_{\tau}\big)} \; \big) \cup v \Big) \Big] = \\
	& \big(\mathbb{R}^{\dim(\tau)} \times T_{\tau} \big) \times \Big[ [0,1) \times \Big( \big( \bigcup_{\substack{\sigma \in \mathcal{P} \\ \dim(\sigma) \leq \dim(\tau)+l}} \operatorname{int}(\sigma) \times \sfrac{T^{n}}{\pi(\delta(\sigma))}  \; \big) \\
	& \bigcap \big( \bigcup_{\substack{\gamma \in \mathcal{M}_{\tau} }} \operatorname{int}(\gamma) \times \sfrac{T^{n}}{\big( \pi(\delta(\sigma_{\gamma})) \times T_{\tau} \big)} \; \big) \bigcup v \Big) \Big] = \\
\end{align*}
\begin{align*}
	& \big(\mathbb{R}^{\dim(\tau)} \times T_{\tau} \big) \times \Big[ [0,1) \times \Big( \bigcup_{\substack{\gamma \in \mathcal{M}_{\tau} \\
			\dim(\gamma) \leq l-1 }} \operatorname{int}(\gamma) \times \sfrac{T^{n}}{\big( \pi(\delta(\sigma_{\gamma})) \times T_{\tau} \big)} \; \big) \bigcup v  \Big) \Big]= \\
	& \big(\mathbb{R}^{\dim(\tau)} \times T_{\tau} \big) \times \mathcal{C}\big( (\mathcal{L}_{\tau})_{\underbrace{2l-1}_{=i} } \big).
\end{align*}
Note that if $l=0$ then we have
\begin{align*}
	p^{-1}(\mathcal{U}) \cap X_{2\dim(\tau)}= \big(\mathbb{R}^{\dim(\tau)} \times T_{\tau} \big) \times v.  
\end{align*}
Consider that $\vert \mathcal{P} \vert -\vert \mathcal{P}^{n-1} \vert$ is the interior of $\vert \mathcal{P} \vert$ and it is dense in $\vert \mathcal{P} \vert$. This implies that $X_{\mathcal{P}}-X_{2(n-1)} \cong T^{n} \times \operatorname{int}(\mathcal{P})$ is dense in $X_{\mathcal{P}}$. Now, if $\dim(\tau)=0$ then $\mathcal{V} \cong \mathbb{R}^{\dim(\vert \mathcal{P} \vert)}$ and $\mathcal{U}$ can be chosen as $\mathcal{C}(\vert \mathcal{M}_{\tau} \vert)$. This yields
\begin{align*}
	p^{-1}(\mathcal{U}) \cong  \mathcal{C} \Big(  \bigcup_{\gamma \in \mathcal{M}_{\tau}}  \operatorname{int}(\gamma) \times \sfrac{T^{n}}{ \pi(\delta(\sigma_{\gamma})) } \Big)  .
\end{align*}
\end{proof}
\begin{remark}
	Note that $\mathcal{S}^{1} \times \mathbb{R} \cong \mathbb{C}^{\ast}$. So $p^{-1}(\operatorname{int}(\tau))=(\mathbb{C}^{\ast})^{\dim(\tau)}$ and specially $p^{-1}(\operatorname{int}(\mathcal{P}))=(\mathbb{C}^{\ast})^{\dim(\vert \mathcal{P} \vert)}$. 
There is an algebraic action of $(\mathbb{C}^{\ast})^{n}$ on $X_{\mathcal{P}}$ with finitely many orbits. The preimage $p^{-1}(\operatorname{int}(\mathcal{P}))$ is dense in $X_{\mathcal{P}}$, as mentioned earlier. One should also note that orbits are in one-to-one correspondence with the faces of $\mathcal{P}$ and each stratum can be written as a disjoint union of finitely many orbits.
\end{remark}
\begin{remark}
	Given $\tau , \eta \in \mathcal{P}$ with $\dim(\tau)=\dim(\eta)$, 
	$\mathcal{L}_{\tau}$ and $\mathcal{L}_{\eta}$ are not necessarily even homotopy equivalent.
\end{remark}
\begin{example}[$n=2$]\label{n=2}
	Consider the complete fan introduced in Example \ref{2D}. Let $\tau=\delta^{-1}(\tau_{1})$ be the 1-dimensional dual face to $\tau_{1}$ in $\mathcal{P}$. Consequently, $\mathcal{M}_{\tau}=\big\{ \{0 \}  \big\}$ and $\vert \mathcal{M}_{\tau} \vert \cong \ast$. Note that $\vert \mathcal{M}_{\tau} \vert \subset \operatorname{int}(\mathcal{P})$. Using the map defined in Expression \ref{Linkmap} gives us
	\begin{align*}
	\mathcal{M}_{\tau} \times T^{2} &\longrightarrow \mathcal{L}_{\tau} \\
	\ast \times T^{2}  &\longmapsto \ast \times \scalebox{1.3}{ $\sfrac{T^{2}}{\mathcal{S}^{1}}$}.
	\end{align*}
	Hence, we get $\mathcal{L}_{\tau} \cong \mathcal{S}^{1}$.
	Now, let $\nu_{12} = \delta^{-1}(\sigma_{12})$ be the dual 0-dimensional face to $\sigma_{12}$ in $\mathcal{P}$. The set $\mathcal{M}_{\nu_{12}}$ has 3 elements and $\vert \mathcal{M}_{\nu_{12}} \vert \cong \mathcal{I}$ with $\operatorname{int}(\mathcal{M}_{\nu_{12}}) \subset \operatorname{int}(\mathcal{P})$. We embed $\{0\} \in \vert \mathcal{M}_{\tau} \vert $ in $\vert \mathcal{P} \vert$ such that $\{0\} \subset \operatorname{int}(\delta^{-1}(\tau_{1}))$ and similarly $\{1\} \subset \operatorname{int}(\delta^{-1}(\tau_{2}))$. Accordingly, we can describe $\mathcal{L}_{\nu_{12}}$ as below
	\begin{align*}
	\mathcal{L}_{\nu_{12}} \cong \big( \operatorname{int}(\mathcal{I})\times T^{2} \big) \bigcup \big(\{0\} \times \sfrac{T^{2}}{\pi(\tau_{1})} \big) \bigcup \big( \{1\} \times \sfrac{T^{2}}{\pi(\tau_{2})} \big).
	\end{align*}  
	
	\begin{figure}[H]
		\centering
		\includegraphics[width=.75\linewidth]{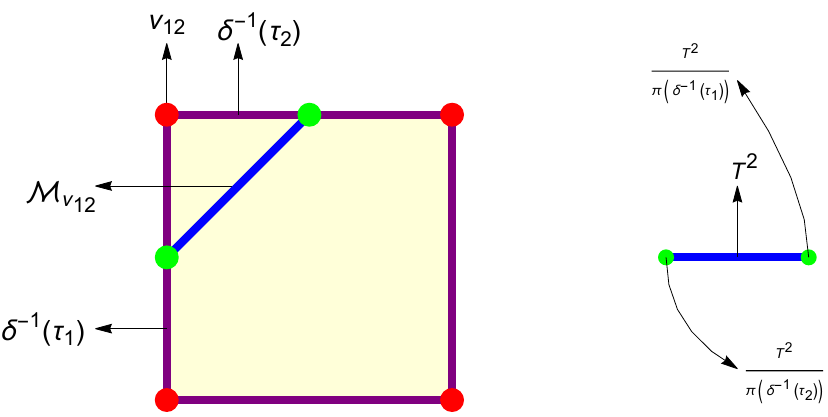}
		\caption{ Link of $\nu_{12}$.  }
		\label{2DLink}
	\end{figure}
\end{example}

\begin{remark}\label{noS1S2}
	Note that $\mathcal{L}_{\nu_{12}} \not\simeq \mathcal{S}^{1} \times \mathcal{S}^{2}$. This can be easily deduced from the structure of $\sigma_{12}$ or simply by comparing the homology groups computed later in this section. However, we can generalize our observation in the following form:
	
Let $\nu$ be a vertex of a 2-dimensional rational convex 
polytope $\mathcal{P}$ with dual fan $\Sigma_{\mathcal{P}}$. 
Then $\nu$ has only 3 higher dimensional neighboring faces, 
namely two 1-dimensional faces $\tau$ and $\eta$ and the 2-dimensional face, 
$\operatorname{int}(\mathcal{P})$. Thus, we have  
\[ \mathcal{L}_{\nu_{12}} \cong \big( \operatorname{int}(\mathcal{I})\times \mathcal{T}^{2} \big) \bigcup \big(\{0\} \times \sfrac{\mathcal{T}^{2}}{\pi(\delta(\eta))} \big) \bigcup \big( \{1\} \times \sfrac{\mathcal{T}^{2}}{\pi(\delta(\tau))} \big).\] 
		and 	$\mathcal{L}_{\nu_{12}} \not\simeq \mathcal{S}^{1} \times \mathcal{S}^{2}$, which comes from the fact that we consider only complete proper fans.
\end{remark}

\begin{remark}\label{LinkTopStra}
Let $\mathcal{P}$ be an $n$-dimensional rational polytope and $\tau$ an $(n-1)$-dimensional face of $\mathcal{P}$. So $\mathcal{M}_{\tau}= \big\{ \{0\} \big\}$, because $\operatorname{int}(\mathcal{P})$ is the only higher dimensional neighboring face of $\tau$. This implies that
$\mathcal{L}_{\tau} \cong \sfrac{T^{n}}{ T^{n-1} } \cong \mathcal{S}^{1}.$
Later, we will use this observation and conclude that: 

		Let $X_{\mathcal{P}}$ be the toric variety associated to $\mathcal{P}$. Then $X_{\mathcal{P}}$ can not have a singular stratum with co-dimension 2.
\end{remark}

\begin{example}[$n=3$]\label{3DLink}
	Consider the complete fan shown in Figure \ref{3DFanEx}. We define the 1-dimensional cones as follows:
	\begin{align*}
	\tau_{1} &= \big\{ x(1,0,1) \; \vert \; x \in \mathbb{R}_{\geq 0} \big\}, \\
	\tau_{2} &= \big\{ x(0,1,1) \; \vert \; x \in \mathbb{R}_{\geq 0} \big\}, \\
	\tau_{3} &= \big\{ x(-1,0,1) \; \vert \; x \in \mathbb{R}_{\geq 0} \big\}, \\
	\tau_{4} &= \big\{ x(0,-1,1) \; \vert \; x \in \mathbb{R}_{\geq 0} \big\}, \\
	\tau_{5} &= \big\{ x(0,0,-1) \; \vert \; x \in \mathbb{R}_{\geq 0} \big\}. \\
	\end{align*}
	\begin{figure}[H]
		\centering
		\includegraphics[width=.50\linewidth]{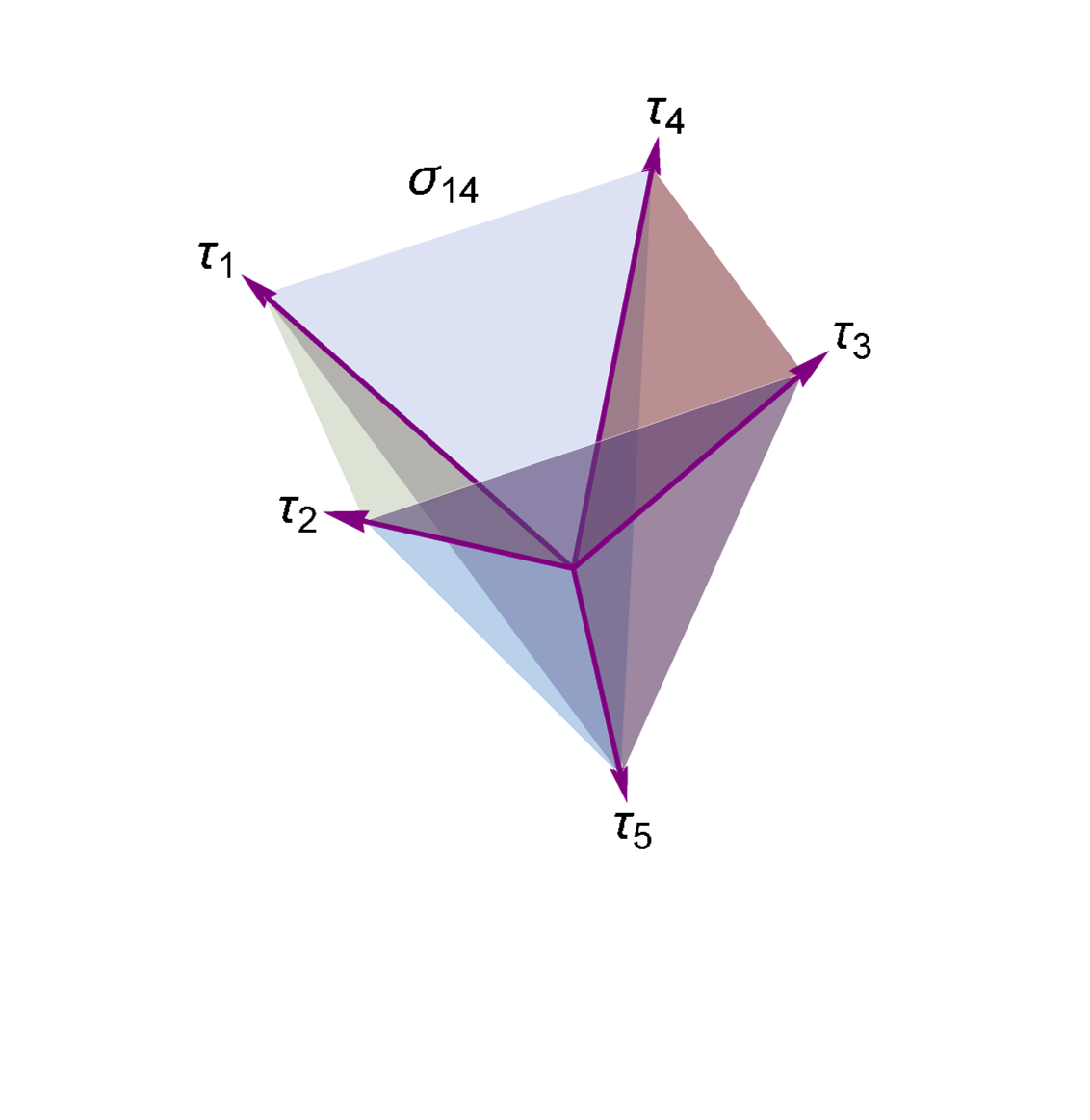}
		\caption{A complete fan $\Sigma$ shown in $\mathbb{R}^3$.    }
		\label{3DFanEx}
	\end{figure}
	Let $\sigma_{14}$, illustrated in the above figure, be the 2-dimensional cone generated by the generators of $\tau_{1}$ and $\tau_{4}$. We define the rest of the 2-dimensional cones similarly. At last, let $\omega_{1234}$ be the 3-dimensional cone which is generated by the generators of $\tau_{1}$, $\tau_{2}$, $\tau_{3}$, and $\tau_{4}$. The cones $\omega_{125}$, $\omega_{235}$, $\omega_{345}$ and $\omega_{145}$ are defined similarly. Consequently, we have
	
	\begin{align*}
	\Sigma=\big\{ \{0\},\tau_{1},\tau_{2},\tau_{3},\tau_{4},\tau_{5},\sigma_{12},\sigma_{23},\sigma_{34},\sigma_{14},\omega_{125},\omega_{235},\omega_{345},\omega_{145},\omega_{1234} \big\}.
	\end{align*} 
The polytope dual to $\Sigma$ is then the pyramid shown in Figure \ref{Poly3DEx}.
	\begin{figure}[H]
		\centering
		\includegraphics[width=.35\linewidth]{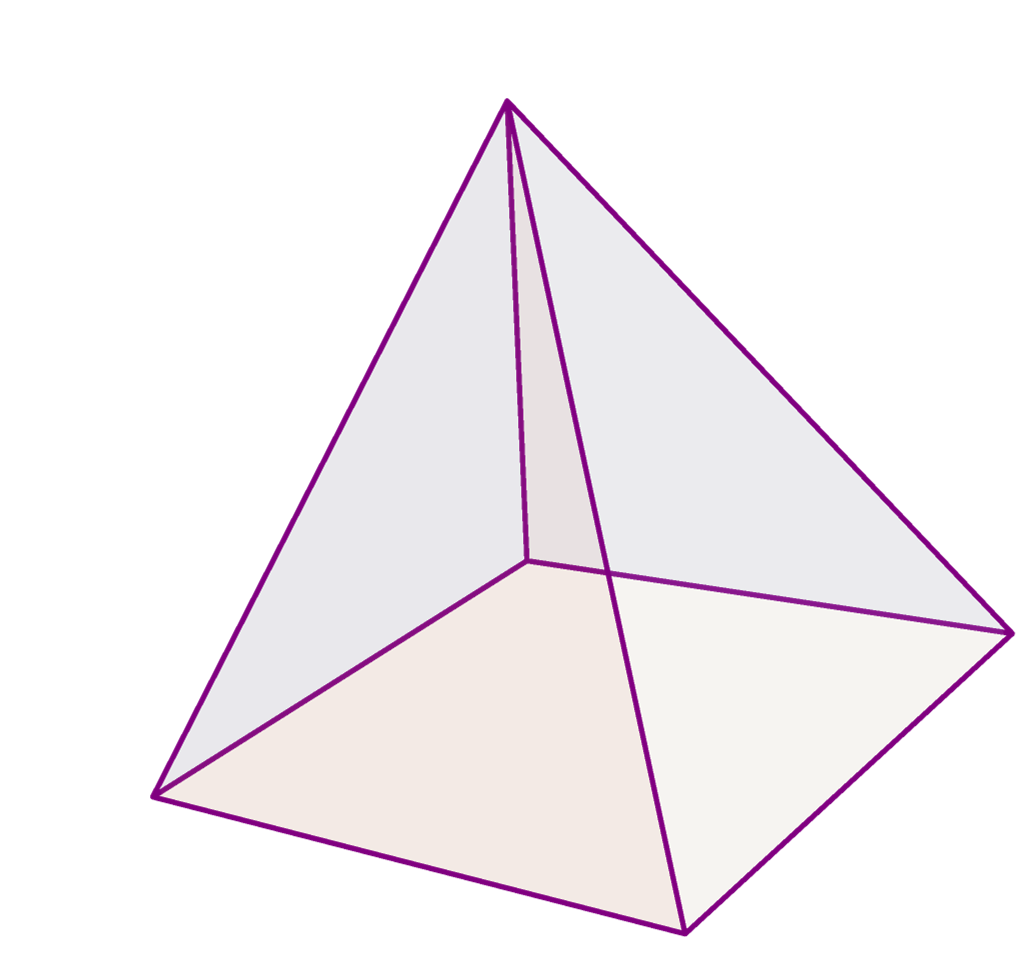}
		\caption{The dual polytope of $\Sigma$ illustrated in Figure \ref{3DFanEx}    .}
		\label{Poly3DEx}
	\end{figure}
	As we discussed earlier, the link of a 2-dimensional face of $\mathcal{P}_{\Sigma}$ is homeomorphic to $\mathcal{S}^{1}$. Now, let $\eta_{12}=\delta^{-1}(\sigma_{12})$ be the 1-dimensional face of $\mathcal{P}$ which is dual to $\sigma_{12}$. It is easy to see that $\mathcal{M}_{\eta_{12}}=\{ \gamma_{\operatorname{int}(\mathcal{P})},\gamma_{\delta^{-1}(\tau_{1})},\gamma_{\delta^{-1}(\tau_{2})} \}$. As in the 2-dimensional case we have $\vert \mathcal{M}_{\eta_{12}} \vert \cong \mathcal{I}$, where $\operatorname{int}( \mathcal{M}_{\eta_{12}} ) \subset \operatorname{int}(\mathcal{P})$, and $\{0\} \subset \operatorname{int}(\delta^{-1}(\tau_{1}))$ considered as a 0-dimensional face of $\mathcal{I}$. Similarly, we have $\{1\} \subset \operatorname{int}(\delta^{-1}(\tau_{2}))$. This gives us 
	\begin{align*}
	\mathcal{L}_{\eta_{12}} \cong \big( \operatorname{int}(\mathcal{I})\times \sfrac{T^{3}}{ \big( \sfrac{T^{3}}{ \pi(\sigma_{12})  }\big)} \big) \cup \big(\{0\} \times 
	\sfrac{T^{3}}{ \big( \sfrac{T^{3}}{ \pi(\sigma_{12})  } \times \pi(\tau_{1}) \big)}
	\cup \big( \{1\} \times \sfrac{T^{3}}{ \big( \sfrac{T^{3}}{ \pi(\sigma_{12})  } \times \pi(\tau_{2}) \big)} \big).
	\end{align*} 
	Note that $\sfrac{T^{3}}{ \big( \sfrac{T^{3}}{ \pi(\sigma_{12}) } \big)} \cong T^{2}$. Hence, we have 
	\begin{align*}
	\mathcal{L}_{\eta_{12}} \cong \big( \operatorname{int}(\mathcal{I})\times T^{2} \big) \cup \big(\{0\} \times \sfrac{T^{2}}{\pi(\tau_{1})} \big) \cup \big( \{1\} \times \sfrac{T^{2}}{\pi(\tau_{2})} \big).
	\end{align*}
	With the same argument as in the 2-dimensional case, we can show that $\mathcal{L}_{\eta_{12}} \not\simeq \mathcal{S}^{1} \times \mathcal{S}^{2}$. \\
	Forthwith, we want to describe a link of the point $v$ at the apex of the pyramid. The set $M_{v}$ is a 2-dimensional abstract polytope with four 1-dimensional faces and hence four 0-dimensional faces. Thus, we have 
	\begin{align*}
	\mathcal{L}_{v} \cong 
	 \big( \operatorname{int}(\mathcal{M}_{v})\times T^{3} \big) \cup \bigcup_{\substack{\gamma_{\tau_{i}} \in \mathcal{M}_{v} \\ \dim(\gamma_{\tau_{i}})=1}} \big( \operatorname{int}(\gamma_{\tau_{i}}) \times 
	 \sfrac{T^{3}}{\pi(\delta(\tau_{i}))} \big) \cup
	 \bigcup_{\substack{\gamma_{\sigma_{i}} \in \mathcal{M}_{v} \\ \dim(\gamma_{\sigma_{i}})=0}} \big( \operatorname{int}(\gamma_{\sigma_{i}}) \times \sfrac{T^{3}}{\pi(\delta(\sigma_{i}))} \big).
	\end{align*}
	Recall that $\sfrac{T^{3}}{\pi(\delta(\sigma_{i}))} \cong \mathcal{S}^{1}$ if $\dim(\sigma_{i})=1$. However, it is easy to show that we can not factor out any $\mathcal{S}^{1}$ in $\mathcal{L}_{v}$. The situation is displayed in Figure \ref{Link3DEx}:
	\begin{figure}[H]
		\centering
		\includegraphics[width=.65\linewidth]{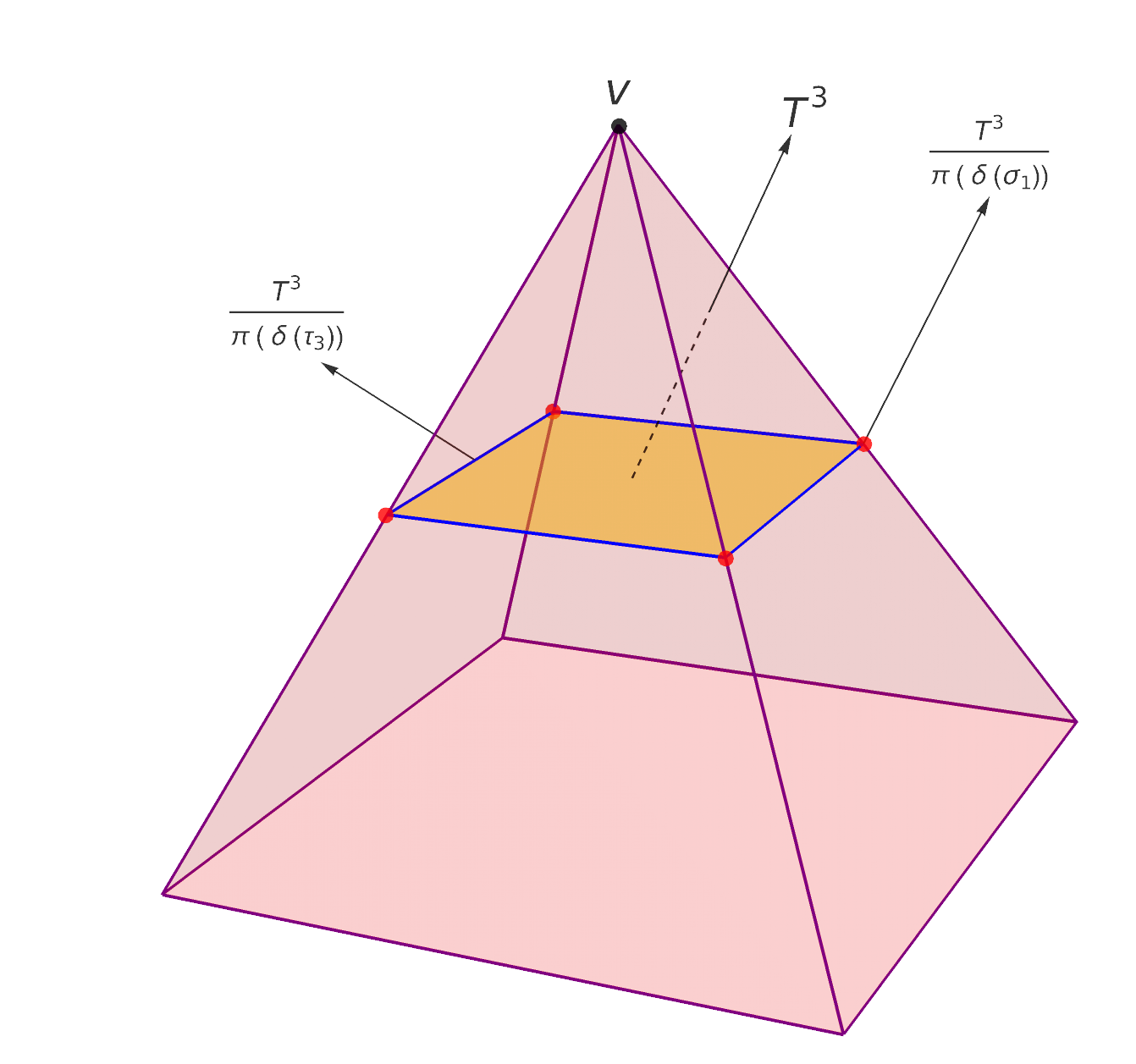}
		\caption{ Link of the rationally singular point $\nu$ in $X_{\Sigma}$.  }
		\label{Link3DEx}
	\end{figure}
\end{example}

\section{CW Structures on Toric Varieties}\label{CW}

In this section, we endow toric varieties with CW structures. In his dissertation \cite{fischli1992toric}, Fischli has described a procedure that yields a 
CW structure of toric varieties. He carried out his method only for real 4-dimensional 
toric varieties. Our primary focus is on real 6-dimensional toric varieties. 
The path that we follow is slightly different from Fischli's, 
and one can use it also for real 6-dimensional compact toric varieties. 
The crucial idea is to ensure that each collapse 
$\sfrac{T^{n} }{ \pi(\tau) } \longrightarrow \sfrac{T^{n} }{ \pi(\sigma) }$ 
for $\tau \prec \sigma$ with $ \sigma, \tau \in \Sigma$, 
an $n$-dimensional complete rational fan, is cellular. 
We will see that in fact, the collapses are automatically cellular for 4-dimensional 
toric varieties. However, a slight modification is needed for 6-dimensional toric varieties.
We will in addition give an idea of how this procedure can inductively be used for arbitrary dimensions. \\
Again, we think of $T^{n}$ as $\sfrac{\mathbb{R}^{n}}{\mathbb{Z}^{n}}$ and
we continue to use the natural projection $\pi: \mathbb{R}^{n} \to \sfrac{\mathbb{R}^{n}}{\mathbb{Z}^{n}}$. Let $\sigma \in \Sigma$ be a $k$-dimensional (rational) cone. 
There is, in addition, the orthogonal projection 
$\psi: \mathbb{R}^{n} \to \sigma^{\perp}$, 
where $\sigma^{\perp}= \{ x \in \mathbb{R}^{d} \vert x.y=0 \; \forall y \in \sigma \}$. Hence, we can write $\sfrac{T^{n}}{\pi(\sigma)} \cong \sfrac{\sigma^{\perp}}{\psi(\mathbb{Z}^{n})}$. Thus, choosing a CW structure on $\sigma^{\perp}$ which is periodic with respect to 
$\psi(\mathbb{Z}^{n})$ induces a finite CW structure on $\sfrac{T^{n}}{\pi(\sigma)}$. Closed $n$-cells homeomorphic to $\mathcal{D}^n$ will be denoted by $e^{n}$.

\subsection{Real 4-dimensional Toric Varieties}\label{RL4D}

Let us start with real 4-dimensional toric varieties. We endow $T^{2}$ with the minimal CW structure with one $0$-cell, two $1$-cells and one $2$-cell:
\begin{align*}
T^{2} = e^{0} \cup \big( e^{1}_{T^{2}_{x}} \cup e^{1}_{T^{2}_{y}} \big) 
   \cup_{f=xyx^{-1}y^{-1}} e^{2}_{T^{2}}
\end{align*}
This structure is induced by the following
CW structure on $\mathbb{R}^{2}$:\\
On $\mathbb{R},$ each interval $[n,n+1 ]$ is considered as a 1-cell and each point $(n) \in \mathbb{R}$ is viewed as a 0-cell, where $n \in \mathbb{Z}$. 
The space $\mathbb{R}^{2}$ is then equipped with the product CW structure on 
$\mathbb{R} \times \mathbb{R}$. Now, let $\tau$ be a 1-dimensional cone in 
$\Sigma$ which is generated by 
$\begin{pmatrix} 
n \\ 
m 
\end{pmatrix}$
where $n$ and $m$ are relatively prime. 
Note that $\mathbb{R}^{2}=\tau \oplus  \tau^{\perp}$. 
Thus, we have then the following decomposition of $\begin{pmatrix} 
1 \\ 
0 
\end{pmatrix}$:
\begin{align*}
\begin{pmatrix} 
1 \\ 
0 
\end{pmatrix}
= (-m)\begin{pmatrix} 
\frac{-m}{m^{2} +n^{2}  } \\ 
\frac{n}{m^{2} +n^{2}  } 
\end{pmatrix} 
+
(n)
\begin{pmatrix} 
\frac{n}{m^{2} +n^{2}  } \\ 
\frac{m}{m^{2} +n^{2}  } 
\end{pmatrix}.
\end{align*}
Since $\begin{pmatrix} 
	1 \\ 
	0 
\end{pmatrix}$ corresponds to the $1$-cell $e^{1}_{T^{2}_{y}}$ of $T^{2}$, the first coefficient, $-m$, of this decomposition will tell us the attaching degree of $\partial (e^{1}_{T^{2}_{y}} \times e^{2}_{\mathcal{
P}} )$ to cells $e^{1}_{\tau} \times e^{1}_{\mathcal{
P}}$. Now, consider the following CW structure on $\tau^{\perp}$:\\
A 1-cell starts at the point $i \cdot 
\begin{pmatrix} 
\frac{-m}{m^{2} +n^{2}  } \\ 
\frac{n}{m^{2} +n^{2}  } 
\end{pmatrix} $ and ends at $(i+1) \cdot
\begin{pmatrix} 
\frac{-m}{m^{2} +n^{2}  } \\ 
\frac{n}{m^{2} +n^{2}  } 
\end{pmatrix} $ for $i \in \mathbb{Z}$. Each point $i \cdot
\begin{pmatrix} 
\frac{-m}{m^{2} +n^{2}  } \\ 
\frac{n}{m^{2} +n^{2}  } 
\end{pmatrix} $ is considered to be a 0-cell. This CW structure of $\tau^{\perp}$ induces a CW structure on $\sfrac{T^{2}}{\pi(\tau)}$. For each $\tau \in \Sigma$ equip $\sfrac{T^{2}}{\pi(\tau)}$ with the above CW structure, which has the form 
$\sfrac{T^{2}}{\pi(\tau)}=e^{0}_{\tau} \cup e^{1}_{\tau}$. 
The CW structure on the polytope $\mathcal{P}$ has been described
in Section \ref{sec.preliminaries}. 
(It is regular and consists of the faces of $\mathcal{P}$.)
Consequently, we have the following cellular chain groups for the toric variety $X_{\Sigma}$:
\begin{align*}
\mathcal{C}_{4}(X) &= \mathbb{Q} \braket{ e^{2}_{T^{2}} \times e^{2}_{\mathcal{P}}}  \\
\mathcal{C}_{3}(X) &= \mathbb{Q} \braket{e^{1}_{T^{2}_{x}} \times e^{2}_{\mathcal{P}} } \oplus  \mathbb{Q} \braket{e^{1}_{T^{2}_{y}} \times e^{2}_{\mathcal{P}} } \\
\mathcal{C}_{2}(X) &=  \mathbb{Q} \braket{e^{0}_{T^{2}} \times e^{2}_{\mathcal{P}} } 
  \oplus
\bigoplus_{\substack{\tau_{i} \in \Sigma \\ \dim(\tau_{i})=1}} \mathbb{Q} \braket{e^{1}_{\tau_{i}} \times e^{1}_{\mathcal{P}_{i}}} \\
\mathcal{C}_{1}(X) &= \bigoplus_{\substack{\tau_{i} \in \Sigma \\ \dim(\tau_{i})=1}} \mathbb{Q} \braket{e^{0}_{\tau_{i}} \times e^{1}_{\mathcal{P}_{i}}}\\ 
\mathcal{C}_{0}(X) &=  \bigoplus_{\substack{\sigma_{i} \in \Sigma \\ \dim(\sigma_{i})=2}} \mathbb{Q} \braket{e^{0}_{\sigma_{i}} \times e^{0}_{\mathcal{P}_{i}}}.
\end{align*} 
We keep in mind that $p^{-1}(\operatorname{int}(\delta^{-1}(\tau_{i}))) \cong \sfrac{T^{2}}{\pi(\tau_{i})} \times \operatorname{int}(\delta^{-1}(\tau_{i}))$ and $\operatorname{int}(\tau_{i})$ represents the interior of a 1-cell, which we will denoted by $e^{1}_{\mathcal{P}_{i}}$ in the above CW structure. \\
It remains to determine the boundary operators in the above CW structure. We can write $ \begin{pmatrix} 
0 \\ 
1 
\end{pmatrix} $ in the following form
\begin{align*}
\begin{pmatrix} 
0 \\ 
1 
\end{pmatrix} = 
(n) \begin{pmatrix} 
\frac{-m}{m^{2} + n^{2}} \\ 
\frac{n}{m^{2} + n^{2}}
\end{pmatrix}
+(m) \begin{pmatrix} 
\frac{n}{m^{2}+n^{2}  } \\ 
\frac{m}{m^{2}+n^{2}}
\end{pmatrix}.
\end{align*}
Since $\begin{pmatrix} 
	0 \\ 
	1 
\end{pmatrix}$ corresponds to the $1$-cell $e^{1}_{T^{2}_{x}}$ of $T^{2}$, the first coefficient, $n$, of this decomposition will tell us the attaching degree of $\partial (e^{1}_{T^{2}_{x}} \times e^{2}_{\mathcal{
			P}} )$ to cells $e^{1}_{\tau} \times e^{1}_{\mathcal{P}}$.
Let $\tau_{i}$ be an 1-dimensional cone in $\Sigma$ with $\begin{pmatrix}
m_{i} \\
n_{i}
\end{pmatrix}$ as the generator. In consequence, we have
\begin{align*}
\partial(e^{1}_{T^{2}_{x}} \times e^{2}_{\mathcal{P}}) &= \sum_{i} n_{i} (e^{1}_{\tau_{i}} \times e^{1}_{\mathcal{P}_{i}}) \\
\partial(e^{1}_{T^{2}_{y}} \times e^{2}_{\mathcal{P}}) &= \sum_{i} -m_{i} (e^{1}_{\tau_{i}} \times e^{1}_{\mathcal{P}_{i}}).
\end{align*}

At this point, we can compute the homology groups of the link of the point $x \in (X_{\Sigma})_{0}$. Let $\sigma_{x}$ be the dual cone to $x$ in $\Sigma$. Let $\tau_{1}$ and $\tau_{2}$ be the 1-dimensional cones with $\tau_{1},\tau_{2} \prec \sigma$. Then, we have the following chain group for the link of $x$, $\mathcal{L}_{x}$.
\begin{align*}
\mathcal{C}_{3}(\mathcal{L}_{x}) &=  \mathbb{Q} \braket{e^{1}_{\mathcal{I}} \times e^{2}_{T^{2}}}, \\
\mathcal{C}_{2}(\mathcal{L}_{x}) &=  \mathbb{Q} \braket{e^{1}_{\mathcal{I}} \times e^{1}_{T^{2}_{x}}} \bigoplus \mathbb{Q} \braket{e^{1}_{\mathcal{I}} \times e^{1}_{T^{2}_{y}}}, \\ 
\mathcal{C}_{1}(\mathcal{L}_{x}) &=  \mathbb{Q} \braket{e^{1}_{\mathcal{I}} \times e^{0}_{T^{2}}} \bigoplus \mathbb{Q} \braket{e^{0}_{\mathcal{I}} \times e^{1}_{\mathcal{S}^{1}_{0}}} \bigoplus \mathbb{Q} \braket{e^{0}_{\mathcal{I}} \times e^{1}_{\mathcal{S}^{1}_{1}}}, \\
\mathcal{C}_{0}(\mathcal{L}_{x}) &=  \mathbb{Q} \braket{e^{0}_{\mathcal{I}} \times e^{0}_{\mathcal{S}^{1}_{0}}} \bigoplus \mathbb{Q} \braket{e^{0}_{\mathcal{I}} \times e^{0}_{\mathcal{S}^{1}_{1}}}.
\end{align*} 
(Here the interval $\mathcal{I}$ is the one defined in the Example \ref{n=2}.) We get the following boundary operators 
\begin{align}
\partial_{3}=\begin{pmatrix}
0 \\
0
\end{pmatrix}, \; \partial_{2}=\begin{pmatrix}
0 & 0 \\
-m_{1} & n_{1} \\
-m_{2} & n_{2}  \\
\end{pmatrix}, \: \partial_{1}=\begin{pmatrix}
+1 & 0 & 0 \\
-1 & 0 & 0 \\
\end{pmatrix},\label{rand2L}
\end{align}
where $\begin{pmatrix}
m_{1} \\
n_{1}
\end{pmatrix}$ and $\begin{pmatrix}
m_{2} \\
n_{2}
\end{pmatrix}$ are the generators of $\tau_{1}$ and $\tau_{2}$, respectively. Hence, we have
\begin{align*}
\operatorname{rk}(H_{3}(\mathcal{L}_{x};\mathbb{Q})) &=1, \\
\operatorname{rk}(H_{2}(\mathcal{L}_{x};\mathbb{Q})) &=0, \\
\operatorname{rk}(H_{1}(\mathcal{L}_{x};\mathbb{Q})) &=0, \\
\operatorname{rk}(H_{0}(\mathcal{L}_{x};\mathbb{Q})) &=1. \\
\end{align*}
As expected, $\mathcal{L}_{x}$ is a rational homology sphere.  But at this point, it is also worthwhile to study the homology groups of $\mathcal{L}_{x}$ with integral coefficients. 
It is easy to see that $H_{1}(\mathcal{L}_{x};\mathbb{Z})=0$ if and only if $\det\begin{pmatrix}
-m_{1} & n_{1} \\
-m_{2} & n_{2} \\
\end{pmatrix}= \pm 1.$ Thus, $\mathcal{L}_{x}$ is an (integral) homology sphere if the previous condition holds. This means that $X_{\sigma}$ is smooth if the generators of all $\sigma \in \Sigma$ with $\dim(\sigma)=2$ satisfy the previous condition. However, there is also an algebraic description of the singularities of toric varieties. Cox, Little, and Schenck in \cite[Theorem~1.3.12]{cox2011toric} show that a toric variety $X_{\sigma}$ is 
smooth as a variety if and only if the minimal generators of $\sigma \subset \mathcal{N} 
\otimes \mathbb{R}$ form a part of an $\mathbb{Z}$-basis of the lattice $\mathcal{N}$ for each $\sigma \in \Sigma$. Note that for real 4-dimensional toric varieties, this translates to the same condition that we obtained from our topological approach.\\

\subsection{Real 6-dimensional Toric Varieties}

In this section, we give a CW structure on real 6-dimensional toric varieties. 
As mentioned earlier, we need a slight refinement of the minimal CW structure of each $T^{2}$ here. The goal is to ensure that each collapse $\sfrac{T^{3}}{\pi(\tau)} \longrightarrow \sfrac{T^{3}}{\pi(\sigma)}$, where $\tau \prec \sigma$ with $\sigma, \tau \in \Sigma$, is cellular. Here again, $\Sigma$ is a complete fan. \\
At first, we consider the case of a 1-dimensional cone $\tau$ in $\Sigma$ which is generated by $\begin{pmatrix}
n \\
m\\
l
\end{pmatrix}$, where $n,m$ and $l$ are all nonzero and $\operatorname{gcd}(n,m,l)=1$. Thus, we have $\tau = \operatorname{span}\begin{pmatrix}
n \\
m \\
l
\end{pmatrix}$ and $\tau^{\perp}=\operatorname{span}\begin{pmatrix}
\begin{pmatrix}
-m \\
n \\
0
\end{pmatrix} , \begin{pmatrix}
-l \\
0 \\
n
\end{pmatrix} \\
\end{pmatrix}$. This gives us the following decomposition of $\begin{pmatrix}
1 \\
0 \\
0
\end{pmatrix}$
\begin{align*}
\begin{pmatrix}
1 \\
0 \\
0
\end{pmatrix} = (-m) \begin{pmatrix}
\frac{-m}{\Delta} \\
\frac{n}{\Delta} \\
0
\end{pmatrix} +
(-l) \begin{pmatrix}
\frac{-l}{\Delta} \\
0 \\
\frac{n}{\Delta}
\end{pmatrix} + (n)
\begin{pmatrix}
\frac{n}{\Delta} \\
\frac{m}{\Delta} \\
\frac{l}{\Delta}
\end{pmatrix},
\end{align*}
where $\Delta=l^2+n^2+m^2$. We can write the above decomposition as
\begin{align*}
\begin{pmatrix}
1 \\
0 \\
0
\end{pmatrix} = (lm) \begin{pmatrix}
\frac{m^2+l^2}{lm} \frac{1}{\Delta}\\
-\frac{mn}{lm} \frac{1}{\Delta} \\
-\frac{nl}{ml} \frac{1}{\Delta}
\end{pmatrix} + (n) \begin{pmatrix}
\frac{n}{\Delta} \\
\frac{m}{\Delta} \\
\frac{l}{\Delta}
\end{pmatrix}.
\end{align*} 
Similarly, we have
\begin{align*}
\begin{pmatrix}
0 \\
1 \\
0
\end{pmatrix} &= (nl) \begin{pmatrix}
-\frac{nm}{nl} \frac{1}{\Delta}\\
\frac{n^2+l^2}{nl} \frac{1}{\Delta} \\
-\frac{ml}{nl} \frac{1}{\Delta}
\end{pmatrix} + (m) \begin{pmatrix}
\frac{n}{\Delta} \\
\frac{m}{\Delta} \\
\frac{l}{\Delta}
\end{pmatrix} \\
\begin{pmatrix}
0 \\
0 \\
1
\end{pmatrix} &= (nm) \begin{pmatrix}
-\frac{nl}{nm} \frac{1}{\Delta}\\
-\frac{ml}{nm} \frac{1}{\Delta} \\
\frac{n^2+m^2}{nm} \frac{1}{\Delta}
\end{pmatrix} + (l) \begin{pmatrix}
\frac{n}{\Delta} \\
\frac{m}{\Delta} \\
\frac{l}{\Delta}
\end{pmatrix}.
\end{align*}
With these in hand, we endow $\tau^{\perp}$ with the following CW structure. The points $i \begin{pmatrix}
\frac{m^2+l^2}{lm} \frac{1}{\Delta}\\
-\frac{mn}{lm} \frac{1}{\Delta} \\
-\frac{nl}{ml} \frac{1}{\Delta}
\end{pmatrix},$ $i \begin{pmatrix}
-\frac{nm}{nl} \frac{1}{\Delta}\\
\frac{n^2+l^2}{nl} \frac{1}{\Delta} \\
-\frac{ml}{nl} \frac{1}{\Delta}
\end{pmatrix} $, and $i \begin{pmatrix}
-\frac{nl}{nm} \frac{1}{\Delta}\\
-\frac{ml}{nm} \frac{1}{\Delta} \\
\frac{n^2+m^2}{nm} \frac{1}{\Delta}
\end{pmatrix}$ are considered to be 0-cells where $i \in \mathbb{Z}.$ A 1-cell starts at $i \begin{pmatrix}
\frac{m^2+l^2}{lm} \frac{1}{\Delta}\\
-\frac{mn}{lm} \frac{1}{\Delta} \\
-\frac{nl}{ml} \frac{1}{\Delta}
\end{pmatrix}$ and ends at  $(i+1) \begin{pmatrix}
\frac{m^2+l^2}{lm} \frac{1}{\Delta}\\
-\frac{mn}{lm} \frac{1}{\Delta} \\
-\frac{nl}{ml} \frac{1}{\Delta}
\end{pmatrix}$  (similarly for $\begin{pmatrix}
-\frac{nm}{nl} \frac{1}{\Delta}\\
\frac{n^2+l^2}{nl} \frac{1}{\Delta} \\
-\frac{ml}{nl} \frac{1}{\Delta}
\end{pmatrix}$ and $\begin{pmatrix}
-\frac{nl}{nm} \frac{1}{\Delta}\\
-\frac{ml}{nm} \frac{1}{\Delta} \\
\frac{n^2+m^2}{nm} \frac{1}{\Delta}
\end{pmatrix}$).
Note that 
\begin{align*}
\begin{pmatrix}
\frac{m^2+l^2}{lm} \frac{1}{\Delta}\\
-\frac{mn}{lm} \frac{1}{\Delta} \\
-\frac{nl}{ml} \frac{1}{\Delta}
\end{pmatrix} + \begin{pmatrix}
-\frac{nm}{nl} \frac{1}{\Delta}\\
\frac{n^2+l^2}{nl} \frac{1}{\Delta} \\
-\frac{ml}{nl} \frac{1}{\Delta}
\end{pmatrix} + \begin{pmatrix}
-\frac{nl}{nm} \frac{1}{\Delta}\\
-\frac{ml}{nm} \frac{1}{\Delta} \\
\frac{n^2+m^2}{nm} \frac{1}{\Delta}
\end{pmatrix} =0.
\end{align*}
The above relation ensures that we can form two $2$-simplices with the above vectors. 
We consider these $2$-simplices as 2-cells of $\tau^{\perp}$. 
From the construction, it is clear that the previous CW structure 
is periodic with respect to $\psi(\mathbb{Z}^{3})$, 
where $\psi$ denotes the natural projection from 
$\mathbb{R}^{3}$ onto $\tau^{\perp}$. 
The induced CW structure on $\sfrac
{T^{3}}{\pi (\tau)} \cong \sfrac{ \tau^{\perp} }{ \psi(\mathbb{Z}^{3})  } \cong T^{2}$ 
can be described as follows.
\begin{align}\label{T2w2Cell}
T^{2} =(e^{2}_{T^{2}_{1} } \cup e^{2}_{T^{2}_{2}} ) \cup 
    (e^{1}_{T^{2}_{x} } \cup e^{1}_{T^{2}_{y}} \cup e^{1}_{T^{2}_{z}}) 
    \cup  e^{0}_{T^{2}}.
\end{align}
With an appropriate orientation we have the following boundary operators.
\begin{align*}
\partial_{2}^{T^{2}}=\begin{pmatrix}
-1 & 1 \\
-1 & 1 \\
-1 & 1 
\end{pmatrix} , \:\: \partial^{T^{2}}_{1} = \begin{pmatrix}
0 & 0 & 0 \\
\end{pmatrix}.
\end{align*}
Schematically, we have
\begin{figure}[H]
	\centering
	\includegraphics[width=.55\linewidth]{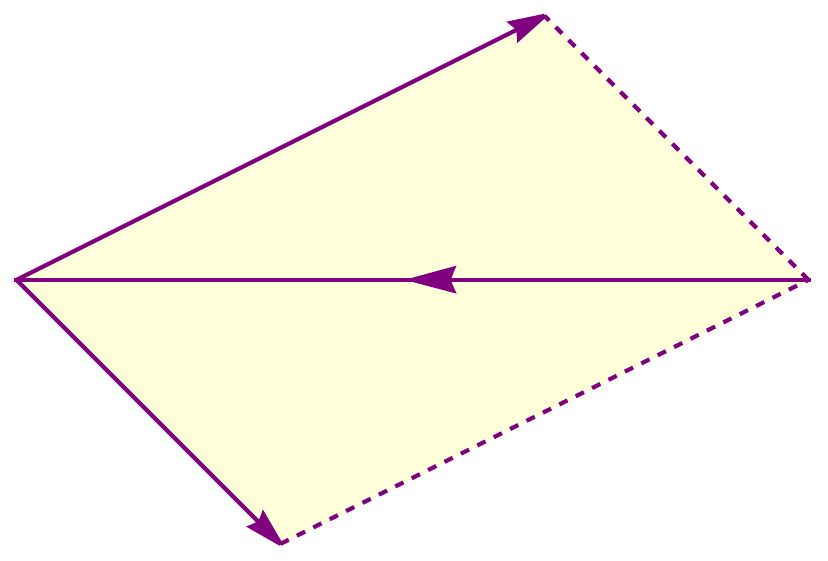}
	\caption{CW structure on $T^{2}$ with two 2-Cells. Note that we identify the opposite boundary edges as in the minimal CW structure on $T^{2}$. }
	\label{T2CW}
\end{figure}
There is yet another case that we need to consider and study, namely when one or two of $n,m,l$ are zero. Without loss of generality assume that $\tau=\operatorname{span}\begin{pmatrix}
n \\
m \\
0
\end{pmatrix}$. This yields 
\begin{align*}
\tau^{\perp} = \operatorname{span}\begin{pmatrix}
\begin{pmatrix}
-m \\
n\\
0
\end{pmatrix},
\begin{pmatrix}
0 \\
0\\
1
\end{pmatrix}
\end{pmatrix}.
\end{align*}
Applying the previous construction induces the following CW structure on 
$\sfrac{T^{3}}{\pi(\tau)} \cong T^{2}$.
\[
T^{2}=e^{2}_{T^{2}} \cup (e^{1}_{T^{2}_{z} } \cup e^{1}_{T^{2}_{y}} ) \cup e^{0}_{T^{2}},
\]
which is the minimal CW structure on $T^{2}$ with the well-known vanishing boundary operators. From the previous section, we know that collapses of the form $T^{2} \longrightarrow \mathcal{S}^1$ are automatically cellular. Thus to ensure cellularity, it is enough to equip each $\sfrac{T^{3}}{\pi(\tau)} \cong T^{2} $ with the appropriate CW structure. To determine the boundary operators of $X_{\Sigma}$ associated with this CW structure, one could also go further and investigate the collapses $\sfrac{T^{3}}{\pi(\tau)} \longrightarrow \sfrac{T^{3}}{\pi(\sigma)}$, such that $\tau \prec \sigma$, $\dim(\tau)=1$, $\dim(\sigma)=2$ and $\tau, \sigma \in \Sigma$. Since we are primarily interested in rational homology, it will turn out that we do not need these collapsing data explicitly for our purposes. \\

For the sake of simplicity, we introduce the following notation that we will use for some of our matrix representations. Let $\mathbf{A}$ be an $l  p \times k  h$ matrix where $l,p,k,h \in \mathbb{N}$. For our purposes, it is sometimes practical to study only a specific part of a matrix. Let us now consider $\mathbf{A}$ as an $p \times h$ matrix where each element (or block) of $\mathbf{A}$ is an $l \times k$ matrix. From the context it will be clear how one can obtain the rest of the matrix from an arbitrary block. Then we represent $\mathbf{A}$ as
\begin{align*}
\mathbf{A} =
\begin{pmatrix}
l \overbrace{ \begin{Bmatrix}
	\phantom{1} & \phantom{1}  & \phantom{1}\\
	\phantom{1} &  \phantom{1}  & \phantom{1}\\
	\phantom{1} &  \phantom{1}  & \phantom{1}\\
	\end{Bmatrix}}^{k}
\end{pmatrix}_{p \times h}.
\end{align*}
\\
At last, we aim to compute the homology groups of an arbitrary real 6-dimensional toric variety. \\

\subsubsection{The Homology of real 6-dimensional Toric Varieties}\label{6DHoGr}

Let $\Sigma$ be a complete fan and $\mathcal{P}$ the associated dual polytope to $\Sigma$. Let $f_1$, $f_2$ and $f_{3}$ denote the number of 1-dimensional, 2-dimensional and 3-dimensional cones of $\Sigma$, respectively. With the above considerations, we can endow  $X_{\mathcal{P}}$, the toric variety associated to $\mathcal{P}$, with the following cellular chain groups:
	
	\begin{align*}
	\mathcal{C}_{6} (X_\mathcal{P}) &= \mathbb{Q} \braket{e^{3}_{T^{3}} \times e^{3}_{\mathcal{P} } } \\
	\mathcal{C}_{5} (X_\mathcal{P}) &= \bigoplus_{i=1}^{3} \mathbb{Q} \braket{e^{2}_{T^{3}_{i} } \times e^{3}_{\mathcal{P} } } \\
	\mathcal{C}_{4} (X_\mathcal{P}) &= \bigoplus_{i=1}^{3} \mathbb{Q} \braket{e^{1}_{T^{3}_{i} } \times e^{3}_{\mathcal{P} } } \bigoplus_{i=1}^{\gamma} \big( \mathbb{Q} \braket{ e^{2}_{(T^{2}_{\gamma_{i}})_{1} } \times e^{2}_{\mathcal{P}_{\gamma_i} }  } \oplus \mathbb{Q} \braket{ e^{2}_{(T^{2}_{\gamma_{i}})_{2} } \times e^{2}_{\mathcal{P}_{\gamma_i} }  } \big) \\&\phantom{=\;\:} \bigoplus_{j=1}^{\omega} \mathbb{Q} \braket{ e^{2}_{T^{2}_{\omega_{j}} } \times e^{2}_{\mathcal{P}_{\omega_j} }  } \\
	\mathcal{C}_{3} (X_\mathcal{P}) &= \mathbb{Q} \braket{e^{0}_{T^{3} } \times e^{3}_{\mathcal{P} } } \bigoplus_{i=1}^{\gamma} \bigoplus_{l=1}^{3}  \mathbb{Q} \braket{ e^{1}_{(T^{2}_{\gamma_{i}})_{l} } \times e^{2}_{\mathcal{P}_{\gamma_{i} } }  }  \bigoplus_{j=1}^{\omega} \bigoplus^{2}_{l=1} \mathbb{Q} \braket{ e^{1}_{(T^{2}_{\omega_{j}})_{l}   } \times e^{2}_{\mathcal{P}_{\omega_{j}  } }  } \\
	\mathcal{C}_{2} (X_\mathcal{P}) &= \bigoplus_{i=1}^{\gamma}   \mathbb{Q} \braket{ e^{0}_{(T^{2}_{\gamma_{i}}) } \times e^{2}_{\mathcal{P}_{\gamma_{i}   }}  }  \bigoplus_{j=1}^{\omega}   \mathbb{Q} \braket{ e^{0}_{(T^{2}_{\omega_{j}}) } \times e^{1}_{\mathcal{P}_{ \omega_{j} } }  } \bigoplus_{l=1}^{f_{2}} \mathbb{Q} \braket{ e^{1}_{\mathcal{S}_{l}^{1}} \times e^{1}_{\mathcal{P}_{l}} } \\
	\mathcal{C}_{1} (X_{\mathcal{P}}) &= \bigoplus_{i=0}^{f_{1}} \mathbb{Q} \braket{e^{0}_{\mathcal{S}^{1}_{i}} \times e^{1}_{\mathcal{P}_{i}    }   } \\
	\mathcal{C}_{0} (X_{\mathcal{P}}) &= \bigoplus_{i=0}^{f_{3}} \mathbb{Q} \braket{e^{0}_{\mathcal{P}_{i}    }   }, 
	\end{align*} 
	where $\gamma$ and $\omega$ are the number of 2-dimensional tori with three and two 1-cells, respectively. The cell $e^{2}_{\gamma_{i}}$ is the 2-cell of $\mathcal{P}$, endowed with the regular CW structure, which is attached to $T^{2}_{\gamma_{i}}$, where $i=1,...,\gamma$, and similarly for $e^{2}_{\mathcal{P}_{\omega_{j}}}$. Hence, we have $\gamma + \omega=f_{1}$.\\
	
	\begin{remark}\label{GammaOmega}
		In the following discussion, without loss of generality and for the sake of simplicity, we only consider 
		$\begin{pmatrix}
			n_{\omega_{j}} \\
			m_{\omega_{j}} \\
			0 \\
		\end{pmatrix}$ as generators of 1-dimensional faces in $\Sigma$ with at least one zero-entry.
		Considering 	$\begin{pmatrix}
			0 \\
			m_{\omega_{j}} \\
			l_{\omega_{j}} \\
		\end{pmatrix}$, 	$\begin{pmatrix}
			n_{\omega_{j}} \\
			0 \\
			l_{\omega_{j}} \\
		\end{pmatrix}$, 	$\begin{pmatrix}
			n_{\omega_{j}} \\
			0 \\
			0 \\
		\end{pmatrix}$, 	$\begin{pmatrix}
			0 \\
			m_{\omega_{j}} \\
			0 \\
		\end{pmatrix}$, and 	$\begin{pmatrix}
			0 \\
			0 \\
			l_{\omega_{j}} \\
		\end{pmatrix}$ will merely change the position of zero elements in the rows labeled by $\omega$ in the boundary operators $\partial_{5}$ and $\partial_{4}$. The computation goes along the same lines.
	\end{remark}
	Let $\begin{pmatrix}
	n_{\gamma_{i}} \\
	m_{\gamma_{i}} \\
	l_{\gamma_{i}} \\
	\end{pmatrix} $ and $\begin{pmatrix}
	n_{\omega_{i}} \\
	m_{\omega_{i}} \\
	0 \\
	\end{pmatrix} $, with $i=1,..., \gamma$ and $j=1,...,\omega$, be the generator of 1-dimensional cones in $\Sigma$ dual to $e^{2}_{\mathcal{P}_{\gamma_{i}}}$ and $e^{2}_{\mathcal{P}_{\omega_{j}}}$, respectively, where we consider the 2-cells as 2-dimensional faces of $\mathcal{P}$. Consider the 2-cell, $e^{2}_{T^{3}_{3}}$, in $T^3 \cong \sfrac{\mathbb{R}^{3}}{\mathbb{Z}^{3}}$ with $e^1_{T^3_1}$ and $e^1_{T^3_2}$ in its topological boundary, where $e^1_{T^3_1}$ is the 1-cell of $T^3$ induced by $\left( \begin{array}{c}
    	1 \\ 0 \\ 0
    \end{array} \right)$ under the map $\mathbb{R}^{3} \longrightarrow \sfrac{\mathbb{R}^{3}}{\mathbb{Z}^{3}}$ and $e^1_{T^3_2}$ induced by $\left( \begin{array}{c}
    	0 \\ 1 \\ 0
    \end{array} \right)$. Under the projection $\mathbb{R}^{3} \longrightarrow \operatorname{span}( \left( \begin{array}{c}
    n_{\gamma_{i}} \frac{1 }{\Delta} \\ m_{\gamma_{i}} \frac{1 }{\Delta}\\ l_{\gamma_{i}} \frac{1 }{\Delta}
\end{array} \right)^{\bot})$, we map the vectors $\left( \begin{array}{c}
1 \\ 0 \\ 0
\end{array} \right)$ and $\left( \begin{array}{c}
0 \\ 1 \\ 0
\end{array} \right)$ to $\left( \begin{array}{c}
(m^{2}_{\gamma_{i}}+l^{2}_{\gamma_{i}}) \frac{1}{\Delta}\\ -m_{i}n_{i} \frac{1}{\Delta} \\ -n_{i}l_{i} \frac{1}{\Delta}
\end{array} \right)$ and $\left( \begin{array}{c}
-n_{\gamma_{i}}m_{\gamma_{i}} \frac{1}{\Delta} \\ (n_{\gamma_{i}}^{2}+l_{\gamma_{i}}^{2})\frac{1}{\Delta} \\ -m_{\gamma_{i}}l_{\gamma_{i}}
\frac{1}{\Delta} \end{array} \right)$, respectively. Hence, the area enclosed by the vectors $\left( \begin{array}{c}
1 \\ 0 \\ 0
\end{array} \right)$ and $\left( \begin{array}{c}
0 \\ 1 \\ 0
\end{array} \right)$ maps to the 2-cell enclosed by $\left( \begin{array}{c}
(m^{2}_{\gamma_{i}}+l^{2}_{\gamma_{i}})\frac{1}{\Delta} \\ -m_{\gamma_{i}}n_{\gamma_{i}}\frac{1}{\Delta} \\ -n_{\gamma_{i}}l_{\gamma_{i}}\frac{1}{\Delta}
\end{array} \right)$ and $\left( \begin{array}{c}
-n_{\gamma_{i}}m_{\gamma_{i}}\frac{1}{\Delta} \\ (n_{\gamma_{i}}^{2}+l_{\gamma_{i}}^{2})\frac{1}{\Delta} \\ -m_{\gamma_{i}}l_{\gamma_{i}}\frac{1}{\Delta}
\end{array} \right)$ in $\operatorname{span}(\left( \begin{array}{c}
 n_{\gamma_{i}}\frac{1 }{\Delta} \\  m_{\gamma_{i}}\frac{1 }{\Delta} \\  l_{\gamma_{i}}\frac{1 }{\Delta}
\end{array} \right)^{\bot})$. We have 
$ \| \left( \begin{array}{c}
	(m^{2}_{\gamma_{i}}+l^{2}_{\gamma_{i}})\frac{1}{\Delta} \\ -m_{\gamma_{i}}n_{\gamma_{i}}\frac{1}{\Delta} \\ -n_{\gamma_{i}}l_{\gamma_{i}}\frac{1}{\Delta}
\end{array} \right) \times \left( \begin{array}{c}
	-n_{\gamma_{i}}m_{\gamma_{i}}\frac{1}{\Delta} \\ (n_{\gamma_{i}}^{2}+l_{\gamma_{i}}^{2})\frac{1}{\Delta} \\ -m_{\gamma_{i}}l_{\gamma_{i}}\frac{1}{\Delta}
\end{array} \right) \|  = \frac{l_{\gamma_{i}}}{\sqrt{\Delta}}$. Note that the area enclosed by the two generators in $\operatorname{span}(\left( \begin{array}{c}
1 \\ 0 \\ 0
\end{array} \right),\left( \begin{array}{c}
0 \\ 1 \\ 0
\end{array} \right))$ is $\| \left( \begin{array}{c}
1 \\ 0 \\ 0
\end{array} \right) \times \left( \begin{array}{c}
0 \\ 1 \\ 0
\end{array} \right) \| =1$. However, the area enclosed by the generators of a 2-cell in $\operatorname{span}(\left( \begin{array}{c}
n_{\gamma_{i}}  \frac{1}{\Delta} \\ m_{\gamma_{i}}  \frac{1}{\Delta} \\ l_{\gamma_{i}}    \frac{1}{\Delta}
\end{array} \right)^{\bot}) $, since $n_{\gamma_{i}}$, $m_{\gamma_{i}}$, and $l_{\gamma_{i}}$ are coprime, is $\| \left( \begin{array}{c}
n_{\gamma_{i}} \frac{1}{\Delta} \\ m_{\gamma_{i}} \frac{1}{\Delta} \\  l_{\gamma_{i}} \frac{1}{\Delta}
\end{array} \right) \| = \frac{1}{\sqrt{\Delta}}$. It follows that the degree of the map from the boundary of $e^{2}_{T^{3}_{3}} \times e^{3}_{\mathcal{P}}$ to the boundaries of $	e^{2}_{(T^{2}_{\gamma_{i}})_{1} } \times e^{2}_{\mathcal{P}_{\gamma_{i} } }$ and $	e^{2}_{(T^{2}_{\gamma_{i}})_{2} } \times e^{2}_{\mathcal{P}_{\gamma_{i} } }$ is $\frac{\sfrac{l_{\gamma_{i}}}{\sqrt{\Delta}}}{\sfrac{1}{\sqrt{\Delta}}}=l_{\gamma_{i}}$. Note that $\omega+\gamma \geq 4$. Hence, we get the following boundary operators. If the generator has at least one zero entry, the treatment goes also along the same lines.
	\begin{align*}
	\partial_{6}&=\begin{pmatrix}
	0 \\
	0 \\
	0 \\
	\end{pmatrix} \\
	\partial_{5} &= \kbordermatrix{
		& 
		e^{2}_{T^{3}_{1}} \times e^{3}_{\mathcal{P}} & 	e^{2}_{T^{3}_{2}} \times e^{3}_{\mathcal{P}} & 	e^{2}_{T^{3}_{3}} \times e^{2}_{\mathcal{P}}	 \\
		e^{1}_{T^{3}_{1}} \times e^{3}_{\mathcal{P}} &   0  &  0  &  0 \\
		e^{1}_{T^{3}_{2}} \times e^{3}_{\mathcal{P}} &   0  &  0  &  0 \\
		e^{1}_{T^{3}_{3}} \times e^{3}_{\mathcal{P}} &   0  &  0  & 0 \\
		e^{2}_{(T^{2}_{\gamma_{i}})_{1} } \times e^{2}_{\mathcal{P}_{\gamma_{i} } }  \gamma \Big\{ & 
		n_{\gamma_{i}}  & m_{\gamma_{i}}   &   l_{\gamma_{i}}   \\
		e^{2}_{(T^{2}_{\gamma_{i}})_{2} } \times e^{2}_{\mathcal{P}_{\gamma_{i } } } \gamma \Big\{ &  
		n_{\gamma_{i}} &m_{\gamma_{i}} &l_{\gamma_{i}}  \\
		e^{2}_{T^{2}_{\omega_{j}} } \times e^{2}_{\mathcal{P}_{\omega_{j } } } \; \omega\Big\{ &   n_{\omega_{j}}  &   m_{\omega_{j}} 	  &    	0  \\
	}
	\\
	\partial_{4} &= \kbordermatrix{
		& 
		e^{1}_{T^{3}_{1}} \times e^{3}_{\mathcal{P}} & 		e^{1}_{T^{3}_{2}} \times e^{3}_{\mathcal{P}} & 		e^{1}_{T^{3}_{3}} \times e^{3}_{\mathcal{P}} & e^{2}_{(T^{2}_{\gamma_{i}})_{1} } \times e^{2}_{\mathcal{P}_{\gamma_{i } } } & 
		e^{2}_{(T^{2}_{\gamma_{i}})_{2} } \times e^{2}_{\mathcal{P}_{\gamma_{i } } } &
		e^{2}_{T^{2}_{\omega_{j}} } \times e^{2}_{\mathcal{P}_{\omega_{j } } }
		\\
		e^{0}_{T^{3}} \times e^{3}_{\mathcal{P}} &   0  &  0  &  0 &  \overbrace{0 }^{ \gamma } & \overbrace{0 }^{ \gamma } & \overbrace{0}^{\omega} \\
		e^{1}_{(T^{2}_{\gamma_{i}})_{1} } \times e^{2}_{\mathcal{P}_{\gamma_{i } } } \gamma \Big\{ &m_{\gamma_{i}} l_{\gamma_{i}} & 0 & 0 & 1 & -1 & 0  \\	
		e^{1}_{(T^{2}_{\gamma_{i}})_{2} } \times e^{2}_{\mathcal{P}_{\gamma_{i } } } \gamma \Big\{ &0 & n_{\gamma_{i}} l_{\gamma_{i}} & 0 & 1 & -1 & 0  \\
		e^{1}_{(T^{2}_{\gamma_{i}})_{3} } \times e^{2}_{\mathcal{P}_{\gamma_{i } } } \gamma \Big\{ &0 & 0 & n_{\gamma_{i}} m_{\gamma_{i}} & 1 & -1 & 0  \\
		e^{1}_{(T^{2}_{\omega_{j}})_{1} } \times e^{2}_{\mathcal{P}_{\omega_{j } } } \omega \Big\{ &-m_{\omega_{j}} & 0 & 0 & 0 & 0 & 0  \\
		e^{1}_{(T^{2}_{\omega_{j}})_{2} } \times e^{2}_{\mathcal{P}_{\omega_{j } } } \omega \Big\{ &0 & n_{\omega_{j}} & 0 & 0 & 0 & 0  \\
	}
	\end{align*} 
	\begin{small}
		\begin{align*}
		\partial_{3} = \kbordermatrix{
			& 
			e^{0}_{T^{3}_{1}} \times e^{3}_{\mathcal{P}} & 		 e^{1}_{(T^{2}_{\gamma_{i}})_{1} } \times e^{2}_{\mathcal{P}_{\gamma_{i } } } & 	
			e^{1}_{(T^{2}_{\gamma_{i}})_{2} } \times e^{2}_{\mathcal{P}_{\gamma_{i } } } & 
			e^{1}_{(T^{2}_{\gamma_{i}})_{3} } \times e^{2}_{\mathcal{P}_{\gamma_{i } } }& 
			e^{1}_{(T^{2}_{\omega_{j}})_{1} } \times e^{2}_{\mathcal{P}_{\omega_{j } } } &
			e^{1}_{(T^{2}_{\omega_{j}})_{2} } \times e^{2}_{\mathcal{P}_{\omega_{j } } }
			\\
			e^{0}_{T^{2}_{\gamma_{i}} } \times e^{2}_{\mathcal{P}_{\gamma_{i } } } \gamma \Big\{  &   1  &  \overbrace{0 }^{ \gamma }  &  \overbrace{0 }^{ \gamma } &  \overbrace{0 }^{ \gamma } & \overbrace{0 }^{ \omega } & \overbrace{0}^{\omega} \\
			e^{0}_{T^{2}_{\omega_{j}} } \times e^{2}_{\mathcal{P}_{\omega_{j} } } \omega \Big\{ &1 & 0 & 0 & 0 & 0 & 0  \\	
			e^{1}_{\mathcal{S}^{1}_{i} } \times e^{1}_{\mathcal{P}_{ l} } \;\;\; f_{2} \Big\{ &0 &  \mathsf{C}_{1} &  \mathsf{C}_{2} &  \mathsf{C}_{3} &  \mathsf{C}_{4}  &  \mathsf{C}_{5} \\
		}
		\end{align*}
	\end{small}
	
	\begin{align*}
	\partial_{2} = \kbordermatrix{
		& e^{0}_{T^{2}_{\gamma_{i}}} \times e^{2}_{\mathcal{P}_{\gamma_{i}}} & 		e^{0}_{T^{2}_{\omega_{j}}} \times e^{2}_{\mathcal{P}_{\omega_{j}}} & 	
		e^{1}_{\mathcal{S}^{1}_{l}} \times e^{1}_{\mathcal{P}_{l}}  \\
		e^{0}_{\mathcal{S}^{1}_{l}} \times e^{1}_{\mathcal{P}_{l} }  \: f_{2} \Big\{  &   \overbrace{  \mathsf{P}_{1} }^{ \gamma }  &   \overbrace{ \mathsf{P}_{2} }^{  \omega }  &  \overbrace{0}^{f_{2}}   \\
	},
	\end{align*}
	where $ \mathsf{C}_{i}, \: i=1,...,5$ are collapsing data that we do not need explicitly. The values in the blocks $ \mathsf{P}_{i}$ for $i=1,2$ are purely determined by the regular CW structure on $\mathcal{P}$.  
	Then, one can easily deduce that
	\begin{align*}
	\operatorname{rk}(H_{6}(X_\mathcal{P};\mathbb{Q})) &=1, \\
	\operatorname{rk}(H_{5}(X_\mathcal{P};\mathbb{Q})) &=0. \\
	\end{align*}
	Moreover, we have
	\begin{align*}
	\operatorname{rk}(\operatorname{Im}(\partial_{5})) &=3,\\
	\operatorname{rk}(\operatorname{ker}(\partial_{4})) &=\gamma+\omega, \\
	\operatorname{rk}(\operatorname{Im}(\partial_{4})) &=\gamma+3, \\
	\operatorname{rk}(\operatorname{ker}(\partial_{2})) &=\operatorname{rk}(\operatorname{ker}(\partial_{2}^{\mathcal{P}}))+f_{2}, \\
	\operatorname{rk}(\operatorname{Im}(\partial_{2}))&=\operatorname{rk}(\operatorname{Im}(\partial_{2}^{\mathcal{P}}))\\ 
	\partial_{1}&=\partial_{1}^{\mathcal{P}}
	\end{align*}
	where $\partial_{i}^{\mathcal{P}}$ denote the boundary operators of $\mathcal{P} \cong \mathcal{D}^{3}$. Thus, we have
	\begin{align*}
	\operatorname{rk}(H_{4}(X_{\mathcal{P}};\mathbb{Q}) )=f_{1}-3.
	\end{align*}
	Define $b^{\prime}$ by the equation

		\begin{align}\label{b1}
		\operatorname{rk}(\operatorname{Im}(\partial_{3}))=f_{2}-b^{\prime}+\operatorname{rk}(\operatorname{Im}(\partial^{\mathcal{P}}_{3}))
	\end{align}
	and $b$ by the equation
		\begin{align}\label{b2}
		b^{\prime}=f_{1}-3-b.
	\end{align}
	In other words, $b^{\prime}$ is the number of linearly dependent rows of $\partial_{3}$.
Straightforward calculation then shows that
	\begin{align*}
	\operatorname{rk}(\operatorname{Im}(\partial_{3}))&=-(f_{1}-f_{2}-b-3)+\operatorname{rk}(\operatorname{Im}(\partial_{3}^{\mathcal{P}})), \\
	\operatorname{rk}(\operatorname{ker}(\partial_{3}))&= (f_{1}-f_{2}-b-3)+(3\gamma+2\omega).
	\end{align*}
	This implies that
	\begin{align*}
	\operatorname{rk}(H_{3}(X_{\mathcal{P}};\mathbb{Q}) )=3f_{1}-f_{2}-b-6.
	\end{align*}
and
	\begin{align*}
	\operatorname{rk}(H_{2}(X_{\mathcal{P}};\mathbb{Q}) )=f_{1}-3-b.
	\end{align*}
		 If $X_{\mathcal{P}}$ is non-singular then Poincar\'e duality implies $b=0$ and hence $b^{\prime}=f_{1}-3$.
The previous considerations lead to the following proposition.
\begin{proposition}\label{OrdiHom}
	Let $X_{\mathcal{P}}$ be a 6-dimensional toric variety associated to a complete rational fan $\Sigma$, which is dual to the polytope $\mathcal{P}$. Let $f_{2}$ and $f_{1}$ be the number of 2- and 1-dimensional cones in $\Sigma$, respectively. Then, we have
	\begin{align*}
	\operatorname{rk}(H_{6}(X_{\mathcal{P}};\mathbb{Q}))&=1,\\
	\operatorname{rk}(H_{5}(X_{\mathcal{P}};\mathbb{Q}))&=0,\\
	\operatorname{rk}(H_{4}(X_{\mathcal{P}};\mathbb{Q}))&=f_{1}-3,\\ \operatorname{rk}(H_{3}(X_{\mathcal{P}};\mathbb{Q}))&=3f_{1}-f_{2}-b-6,\\
	\operatorname{rk}(H_{2}(X_{\mathcal{P}};\mathbb{Q}) )&=f_{1}-3-b,\\
	\operatorname{rk}(H_{1}(X_{\mathcal{P}};\mathbb{Q}) )&=0,\\
	\operatorname{rk}(H_{0}(X_{\mathcal{P}};\mathbb{Q}) )&=1,\\
	\end{align*}
	where the parameter $b$ is determined by Equations \ref{b1} and \ref{b2}.
\end{proposition}

\begin{remark}
	In the next section, we give a geometric description of $b$, at least for some special cases. There is yet another approach by McConnell in \cite{mcconnell1989rational}, which computes the homology groups of 6-dimensional toric varieties, using spectral sequences.
\end{remark}
\begin{remark}
	As observed by McConnell in \cite{mcconnell1989rational}, $b$ is not combinatorially invariant, i.e. $b$ is not merely determined by the number of cones.
\end{remark}

\subsection{Singularities of 6-dimensional Toric Varieties}\label{Sing6D}

Let $\Sigma$ be a complete fan in $\mathbb{R}^{3}$ and $\mathcal{P}$ be the associated dual polytope. As mentioned in Remark \ref{LinkTopStra}, the link of a point $x \in X_{4}-X_{2}$ is simply homeomorphic to $\mathcal{S}^{1}$. For $x \in X_{2}-X_{0}$, we employ the construction of links introduced earlier. Let $\tau \in \Sigma$ such that $p(x) \in \operatorname{int}(\delta^{-1}(\tau))$, where $p$ is the natural projection $p:X_{\Sigma} \longrightarrow \vert \mathcal{P} \vert$. Then, we have $\mathcal{S}_{\tau}=\{\sigma_{1}, \sigma_{2}, \operatorname{int}(\vert \mathcal{P} \vert) \}$, where $\sigma_{1}$ and $\sigma_{2}$ are the two 2-dimensional neighboring faces of $\tau$. Hence as in Example \ref{n=2}, we have $\mathcal{M}_{\tau} \cong \mathcal{I}$. However, in contrast to Example \ref{n=2}, here we have $\sfrac{T^{3}}{\pi(\delta(\tau))} \cong \mathcal{S}^{1}$. But in the end, $\mathcal{L}_{\tau}$, the link of $x$, has the same CW structure as in Example \ref{n=2},
\begin{align*}
\mathcal{L}_{\tau} \cong \big( \operatorname{int}(\mathcal{I})\times \sfrac{T^{3}}{(\sfrac{T^{3}}{\pi(\delta(\tau))})} \big) \bigcup \big(\{0\} \times \sfrac{(\sfrac{T^{3}}{\delta(\pi(\sigma_{1}))})}{(\sfrac{T^{3}}{\pi(\delta(\tau))})} \big) \bigcup \big(\{1\} \times \sfrac{(\sfrac{T^{3}}{\delta(\pi(\sigma_{2}))})}{(\sfrac{T^{3}}{\pi(\delta(\tau))})} \big).
\end{align*}
Note that with the same argument as in Remark \ref{noS1S2}, we conclude that $\mathcal{L} \not\simeq \mathcal{S}^{1} \times \mathcal{S}^{2}$. The computational method goes along the same line as in Section \ref{RL4D} and yields

\begin{align}\label{LinkHomRnk}
\begin{split}
\operatorname{rk}(H_{3}(\mathcal{L}_{\tau};\mathbb{Q})) &=1, \\
\operatorname{rk}(H_{2}(\mathcal{L}_{\tau};\mathbb{Q})) &=0, \\
\operatorname{rk}(H_{1}(\mathcal{L}_{\tau};\mathbb{Q})) &=0, \\
\operatorname{rk}(H_{0}(\mathcal{L}_{\tau};\mathbb{Q})) &=1. 
\end{split}
\end{align}

\begin{remark}
	Studying the homology groups of $\mathcal{L}_{\tau}$ with integral coefficients yields the same result for the smoothness of a 2-dimensional stratum in $X_{\Sigma}$ as in the previous case. This means that if the two generators of $\delta(\tau)$ in $\Sigma = \mathcal{N} \otimes \mathbb{R}$ do not form a part of a $\mathbb{Z}$-basis of the lattice $\mathcal{N}$ then the stratum is singular in the sense of Definition \ref{DefSing}. 
\end{remark}
In Example \ref{3DLink}, we constructed the link of a connected component of $X_{0}$ (a point) in a specific real 6-dimensional toric variety. Here, we generalize this Example to an arbitrary real 6-dimensional compact toric variety. \\
Let $x \in X_{0}$. We consider $x$ as a 0-dimensional face of $\mathcal{P}$. The set $\mathcal{S}_{x}$ consists of $\operatorname{int}(\mathcal{P})$ and all 1- and 2-dimensional faces of $\mathcal{P}$, $\tau$, with $x \prec \tau$. Following the introduced construction of links, $\vert \mathcal{M}_{x} \vert$ is a 2-dimensional polytope with $f_{x_{1}}$ vertices and $f_{x_{2}}$ 1-dimensional faces, where $f_{x_{1}}$ and $f_{x_{2}}$ denote the number of 1- and 2-dimensional neighboring faces of $x$ in $\Sigma$, respectively. Since $\vert \mathcal{M}_{x} \vert \cong \mathcal{D}^2$, its boundary is homeomorphic to a circle. As the Euler characteristic of a circle vanishes, $f_{x_{1}} = f_{x_{2}}$. Hence, we have the following relations.
\begin{align*}
p^{-1}_{\mathcal{L}_{x}}(\operatorname{int}(\mathcal{M}_{x}))&=T^{3}\times \operatorname{int}(\mathcal{M}_{x}) \\
p^{-1}_{\mathcal{L}_{x}}(\operatorname{int}(\tau)) &=T^{2} \times \operatorname{int}(\tau) \:\:\:\:\:\:\:\: \text{for} \: \tau \in \mathcal{M}_{x} \: \text{with}\: \dim(\tau)=1 \\ 
p^{-1}_{\mathcal{L}_{x}}(\nu) &=\mathcal{S}^{1} \times \nu \:\:\:\:\:\:\:\:\:\:\:\:\:\:\:\:\:\:\:\, \text{for} \: \nu \in \mathcal{M}_{x} \: \text{with}\: \dim(\nu)=0.
\end{align*}     
At this point, we can endow $\mathcal{L}_{x}$ with a CW structure and compute the corresponding homology groups. As before, we equip $T^{3}$ with the minimal CW structure and each $T^{2}$ with an appropriate CW structure such that each collapse map becomes cellular as we discussed earlier. Each 1-dimensional face of $\mathcal{M}_{x}$, $\tau_{i}$, is associated to a 2-dimensional face of $\mathcal{P}$, which is dual to a 1-dimensional cone in $\Sigma$ such that the dual cones lie in $\mathcal{S}_{x}$. Let $a$ and $b$ be the numbers of such 1-dimensional cones in $\Sigma$, whose generators have no zero entry and at least one zero-entry, respectively. Take into consideration that $a+b \geq 3$. Let $\begin{pmatrix}
n_{a_{i}} \\
m_{a_{i}} \\
l_{a_{i}} \\
\end{pmatrix}$ and $\begin{pmatrix}
n_{b_{j}} \\
m_{b_{j}} \\
0 \\
\end{pmatrix}$ with $i=1,\dots,a$ and $j=1, \dots , b$ be the generators of these 1-dimensional cones in $\Sigma$ where $n_{a_{i}}, m_{a_{i}}$ and $l_{a_{i}}$ are non-zero and at least one of the entries $n_{b_{j}}$ or $m_{b_{j}}$ is non-zero.

Accordingly, we get the following chain groups for $\mathcal{L}_{x}$, where we use the regular cellular chain groups on $\mathcal{M}_{x}$:
\begin{align*}
\mathcal{C}_{5} (\mathcal{L}_{x}) &= \mathbb{Q} \braket{e^{3}_{T^{3}} \times e^{2}_{\mathcal{M}_{x} } } \\
\mathcal{C}_{4} (\mathcal{L}_{x}) &= \bigoplus_{i=1}^{3} \mathbb{Q} \braket{e^{2}_{T^{3}_{i} } \times e^{2}_{\mathcal{M}_{x} } } \\
\mathcal{C}_{3} (\mathcal{L}_{x}) &= \bigoplus_{i=1}^{3} \mathbb{Q} \braket{e^{1}_{T^{3}_{i} } \times e^{2}_{\mathcal{M}_{x} } } \bigoplus_{i=1}^{a} \big( \mathbb{Q} \braket{ e^{2}_{(T^{2}_{a_{i}})_{1} } \times e^{1}_{(\mathcal{M}_{x_{a } })_{i} }  } \oplus \mathbb{Q} \braket{ e^{2}_{(T^{2}_{a_{i}})_{2} } \times e^{1}_{(\mathcal{M}_{x_{a } })_{i} }  } \big) \\&\phantom{=\;\:} \bigoplus_{j=1}^{b} \mathbb{Q} \braket{ e^{2}_{(T^{2}_{b_{j}}) } \times e^{1}_{(\mathcal{M}_{x_{b } })_{i} }  } \\
\mathcal{C}_{2} (\mathcal{L}_{x}) &= \mathbb{Q} \braket{e^{0}_{T^{3} } \times e^{2}_{\mathcal{M}_{x} } } \bigoplus_{i=1}^{a} \bigoplus_{l=1}^{3}  \mathbb{Q} \braket{ e^{1}_{(T^{2}_{a_{i}})_{l} } \times e^{1}_{(\mathcal{M}_{x_{a } })_{i} }  }  \bigoplus_{j=1}^{b} \bigoplus^{2}_{l=1} \mathbb{Q} \braket{ e^{1}_{(T^{2}_{b_{j}})_{l}   } \times e^{1}_{(\mathcal{M}_{x_{b } })_{j} }  } \\
\mathcal{C}_{1} (\mathcal{L}_{x}) &= \bigoplus_{i=1}^{a}   \mathbb{Q} \braket{ e^{0}_{(T^{2}_{a_{i}}) } \times e^{1}_{(\mathcal{M}_{x_{a } })_{i} }  }  \bigoplus_{j=1}^{b}   \mathbb{Q} \braket{ e^{0}_{(T^{2}_{b_{j}}) } \times e^{1}_{(\mathcal{M}_{x_{b } })_{j} }  } \\
\mathcal{C}_{0} (\mathcal{L}_{x}) &= \bigoplus_{l=0}^{f_{x_{1}}} \mathbb{Q} \braket{e^{0}_{\mathcal{M}_{x_{l}}    }   }. 
\end{align*}
\begin{remark}
	In the above chain complex, $e^{1}_{(\mathcal{M}_{x_{a}})_{i}  }$ and $e^{1}_{(\mathcal{M}_{x_{b}})_{j}  }$ are those 1-cells of $\mathcal{M}_{x}$, whose associated $T^{2}$ in $\mathcal{L}_{x}$ has three and two 1-cells, respectively. Similarly, we have labeled each $T^{2}$ with either $a_{i}$ or $b_{j}$ according to the generator of the collapsed 1-dimensional cone in $\Sigma$.	
\end{remark}
\noindent Consequently, we obtain the following boundary operators for the above CW structure on $\mathcal{L}_{x}$.
\begin{align*}
\partial_{5}&=\begin{pmatrix}
0 \\
0 \\
0 \\
\end{pmatrix} \\
\end{align*}

\begin{align}
\partial_{4} &= \kbordermatrix{
	& 
	e^{2}_{T^{3}_{1}} \times e^{2}_{\mathcal{M}_{x}} & 	e^{2}_{T^{3}_{2}} \times e^{2}_{\mathcal{M}_{x}} & 	e^{2}_{T^{3}_{3}} \times e^{2}_{\mathcal{M}_{x}}	 \\
	e^{1}_{T^{3}_{1}} \times e^{2}_{\mathcal{M}_{x}} &   0  &  0  &  0 \\
	e^{1}_{T^{3}_{2}} \times e^{2}_{\mathcal{M}_{x}} &   0  &  0  &  0 \\
	e^{1}_{T^{3}_{3}} \times e^{2}_{\mathcal{M}_{x}} &   0  &  0  & 0 \\
	e^{2}_{(T^{2}_{a_{i}})_{1} } \times e^{1}_{(\mathcal{M}_{x_{a } })_{i} }  a \Big\{ & 
		n_{a_{i}} 	 &   m_{a_{i}}  &   l_{a_{i}}  \\
	e^{2}_{(T^{2}_{a_{i}})_{2} } \times e^{1}_{(\mathcal{M}_{x_{a } })_{i} } a \Big\{ &  
		n_{a_{i}} 	 &  m_{a_{i}}  &   l_{a_{i}}  \\
	e^{2}_{(T^{2}_{b_{j}}) } \times e^{1}_{(\mathcal{M}_{x_{b } })_{j} } \; b\Big\{ &   n_{b_{j}}   &   m_{b_{j}} 	  &    	0 \\
}
\end{align}
\begin{footnotesize}
	\begin{align}\label{Link6D}
	\partial_{3} &= \kbordermatrix{
		& 
		e^{1}_{T^{3}_{1}} \times e^{2}_{\mathcal{M}_{x}} & 		e^{1}_{T^{3}_{2}} \times e^{2}_{\mathcal{M}_{x}} & 		e^{1}_{T^{3}_{3}} \times e^{2}_{\mathcal{M}_{x}} & e^{2}_{(T^{2}_{a_{i}})_{1} } \times e^{1}_{(\mathcal{M}_{x_{a } })_{i} } & 
		e^{2}_{(T^{2}_{a_{i}})_{2} } \times e^{1}_{(\mathcal{M}_{x_{a } })_{i} } &
		e^{2}_{(T^{2}_{b_{j}}) } \times e^{1}_{(\mathcal{M}_{x_{b } })_{j} }
		\\
		e^{0}_{T^{3}} \times e^{2}_{\mathcal{M}_{x}} &   0  &  0  &  0 &  \overbrace{0 }^{ a } & \overbrace{0 }^{ a } & \overbrace{0}^{b} \\
		e^{1}_{(T^{2}_{a_{i}})_{1} } \times e^{1}_{(\mathcal{M}_{x_{a } })_{i} } a \Big\{ &m_{a_{i}} l_{a_{i}} & 0 & 0 & 1 & -1 & 0  \\	
		e^{1}_{(T^{2}_{a_{i}})_{2} } \times e^{1}_{(\mathcal{M}_{x_{a } })_{i} } a \Big\{ &0 & n_{a_{i}} l_{a_{i}} & 0 & 1 & -1 & 0  \\
		e^{1}_{(T^{2}_{a_{i}})_{3} } \times e^{1}_{(\mathcal{M}_{x_{a } })_{i} } a \Big\{ &0 & 0 & n_{a_{i}} m_{a_{i}} & 1 & -1 & 0  \\
		e^{1}_{(T^{2}_{b_{j}})_{1} } \times e^{1}_{(\mathcal{M}_{x_{b } })_{j} } b \Big\{ &-m_{b_{j}} & 0 & 0 & 0 & 0 & 0  \\
		e^{1}_{(T^{2}_{b_{j}})_{2} } \times e^{1}_{(\mathcal{M}_{x_{b } })_{j} } b \Big\{ &0 & n_{b_{j}} & 0 & 0 & 0 & 0  \\
	}
	\end{align}
\end{footnotesize}
In order to avoid any ambiguity on $\partial_{3}$, we describe the explicit form of $\partial_{3}$ in more details. Let $l\in \{1,2,3\}$. The rows associated to $ e^{1}_{(T^{2}_{a_{i}})_{l} } \times e^{1}_{(\mathcal{M}_{x_{a } })_{i} }$ have only non-zero entries on the columns labeled by $e^{1}_{T^{3}_{l}} \times e^{2}_{\mathcal{M}_{x}}$, $e^{2}_{(T^{2}_{a_{i}})_{1} } \times e^{1}_{(\mathcal{M}_{x_{a } })_{i} }$ and $e^{2}_{(T^{2}_{a_{i}})_{2} } \times e^{1}_{(\mathcal{M}_{x_{a } })_{i} }$. Similarly, the columns labeled by $e^{1}_{T^{3}_{l}} \times e^{2}_{\mathcal{M}_{x}}$ have only non-zero entires on the rows labeled by $ e^{1}_{(T^{2}_{a_{i}})_{l} } \times e^{1}_{(\mathcal{M}_{x_{a } })_{i} }$ and possibly on  $e^{1}_{(T^{2}_{b_{j}})_{1} } \times e^{1}_{(\mathcal{M}_{x_{b } })_{j} }$ or $ e^{1}_{(T^{2}_{b_{j}})_{2} } \times e^{1}_{(\mathcal{M}_{x_{b } })_{j} }$. Lastly, the columns associated to $e^{2}_{(T^{2}_{a_{i}})_{1} } \times e^{1}_{(\mathcal{M}_{x_{a } })_{i} }$ and $e^{2}_{(T^{2}_{a_{i}})_{1} } \times e^{1}_{(\mathcal{M}_{x_{a } })_{i} }$ have only non-zero entries on the rows labeled by $   e^{1}_{(T^{2}_{a_{i}})_{1} } \times e^{1}_{(\mathcal{M}_{x_{a } })_{i} }$, $   e^{1}_{(T^{2}_{a_{i}})_{2} } \times e^{1}_{(\mathcal{M}_{x_{a } })_{i} }$ and $   e^{1}_{(T^{2}_{a_{i}})_{3} } \times e^{1}_{(\mathcal{M}_{x_{a } })_{i} }$. For instance, let us start with $i=1$. Consider the following part of $\partial_{3}$:
\begin{align*}
\begin{pmatrix}
m_{a_{1}} l_{a_{1}} & 0 & 0 &1&-1 &0 \\
0& n_{a_{1}} l_{a_{1}} & 0 & 1&-1&0 \\
0&0& n_{a_{1}} m_{a_{1}}&1&-1&0\\
\end{pmatrix}.
\end{align*}
Adding $e^{2}_{(T^{2}_{a_{i}})_{1} } \times e^{1}_{(\mathcal{M}_{x_{a } })_{i} }$ and $e^{2}_{(T^{2}_{a_{i}})_{2} } \times e^{1}_{(\mathcal{M}_{x_{a } })_{i} }$ to columns, and $   e^{1}_{(T^{2}_{a_{i}})_{1} } \times e^{1}_{(\mathcal{M}_{x_{a } })_{i} }$, $   e^{1}_{(T^{2}_{a_{i}})_{2} } \times e^{1}_{(\mathcal{M}_{x_{a } })_{i} }$ and $   e^{1}_{(T^{2}_{a_{i}})_{3} } \times e^{1}_{(\mathcal{M}_{x_{a } })_{i} }$ to rows gives us
\begin{align*}
\begin{pmatrix}
m_{a_{1}} l_{a_{1}} & 0 & 0 &1&-1&0&0 &0 \\
0& n_{a_{1}} l_{a_{1}} & 0 & 1&-1&0&0&0 \\
0&0& n_{a_{1}} m_{a_{1}}&1&-1&0&0&0\\
m_{a_{2}} l_{a_{2}} & 0 & 0 &0&0&1&-1 &0 \\
0& n_{a_{2}} l_{a_{2}} & 0 & 0&0&1&-1&0 \\
0&0& n_{a_{2}} m_{a_{2}}&0&0&1&-1&0\\
\end{pmatrix}.
\end{align*}  
Adding more columns and rows goes along the same lines. Computing the ker and Im of the above boundary operators is straightforward. We find
\begin{align*}
\operatorname{rk(ker(}\partial_{5}))&=1, \; \operatorname{rk(Im(}\partial_{5}))=0, \\ 
\operatorname{rk(ker(}\partial_{4}) )&=0, \; \operatorname{rk(Im(}\partial_{4}) )=3, \\
\operatorname{rk(ker(}\partial_{3}) )&=a+b=f_{x_{1} },  
\end{align*}
which yield
\begin{align*}
\operatorname{rk}(H_{5}(\mathcal{L}_{x}; \mathbb{Q}))&=1 \\
\operatorname{rk}(H_{4}(\mathcal{L}_{x}; \mathbb{Q}))&=0 \\
\operatorname{rk}(H_{3}(\mathcal{L}_{x}; \mathbb{Q}))&=f_{x_{1} }-3.
\end{align*}

The link $\mathcal{L}_{x}$ is equipped with a stratification 
$\mathcal{L}_{x}=(\mathcal{L}_{x})_{5} \supset (\mathcal{L}_{x})_{3} \supset (\mathcal{L}_{x})_{1}$ and a link of $(\mathcal{L}_{x)_{1}}$ in $\mathcal{L}_{x}$ is a link of $(X_{\Sigma})_{2}$ in $X_{\Sigma}$. As those links are rational homology spheres \ref{LinkHomRnk}, $\mathcal{L}_{x}$ is a rational homology manifold. In particular, the homology groups below the middle degree of $\mathcal{
L}_{x}$ can be obtained using rational Poincaré duality. 

\begin{remark}
	There is yet another way for computing the homology groups below the middle degree. We can use our previous method of orthogonal decomposition and compute $\partial_{2}$. The boundary operator $\partial_{2}$ will have the following form.
	\begin{tiny}
		\begin{align*}
		\partial_{2} = \kbordermatrix{
			& 
			e^{0}_{T^{3}_{1}} \times e^{2}_{\mathcal{M}_{x}} & 		 e^{1}_{(T^{2}_{a_{i}})_{1} } \times e^{1}_{(\mathcal{M}_{x_{a } })_{i} } & 	
			e^{1}_{(T^{2}_{a_{i}})_{2} } \times e^{1}_{(\mathcal{M}_{x_{a } })_{i} } & 
			e^{1}_{(T^{2}_{a_{i}})_{3} } \times e^{1}_{(\mathcal{M}_{x_{a } })_{i} }& 
			e^{1}_{(T^{2}_{b_{j}})_{1} } \times e^{1}_{(\mathcal{M}_{x_{b } })_{j} } &
			e^{1}_{(T^{2}_{b_{j}})_{2} } \times e^{1}_{(\mathcal{M}_{x_{b } })_{j} }
			\\
			e^{0}_{T^{2}_{a_{i}} } \times e^{1}_{(\mathcal{M}_{x_{a } })_{i} } a \Big\{  &   1  &  \overbrace{0 }^{ a }  &  \overbrace{0 }^{ a } &  \overbrace{0 }^{ a } & \overbrace{0 }^{ b } & \overbrace{0}^{b} \\
			e^{0}_{T^{2}_{b_{j}} } \times e^{1}_{(\mathcal{M}_{x_{b} })_{j} } b \Big\{ &1 & 0 & 0 & 0 & 0 & 0  \\	
			e^{1}_{\mathcal{S}^{1}_{i} } \times e^{0}_{(\mathcal{M}_{ x})_{i} } a+b \Big\{ &0 & \Big[ \phantom{---}  &  &  &  & \phantom{---} \Big]  \\
		}
		\end{align*}
	\end{tiny}
	It can be shown that rows associated to the cells of the form $ e^{1}_{\mathcal{S}^{1}_{i} } \times e^{0}_{(\mathcal{M}_{ x})_{i} }$ are linearly independent. This implies
	\begin{align*}
	\operatorname{rk(ker(}\partial_{2}))&=(3a+2b)-(a+b)=2a+b \\
	\operatorname{rk(Im(}\partial_{2}))&=a+b+1.
	\end{align*}
	Keep in mind that the upper left part of $\partial_{2}$ is merely determined by the orientation of $\vert \mathcal{M}_{x} \vert$. Similarly, for $\partial_{1}$ we have 
	\begin{align*}
	\partial_{1} = \kbordermatrix{
		& 
		e^{0}_{T^{2}_{a_{i}} } \times e^{1}_{(\mathcal{M}_{x_{a } })_{i} } & 	
		e^{0}_{T^{2}_{b_{j}} } \times e^{1}_{(\mathcal{M}_{x_{b} })_{j} } & 	
		e^{1}_{\mathcal{S}^{1}_{i} } \times e^{0}_{(\mathcal{M}_{ x})_{i} } & 
		\\
		e^{0}_{\mathcal{S}^{1}_{i} } \times e^{0}_{(\mathcal{M}_{ x})_{i} } a+b \Big\{  &   \overbrace{z_{i}}^{a}  & \overbrace{y_{j}}^{b}   &  \overbrace{0 }^{ a+b }  \\
	},
	\end{align*}
	where again the $y_{j}$ and the $z_{i}$ are solely determined by the orientation of $\vert \mathcal{M}_{x} \vert$.
	Using $\vert \mathcal{M}_{x} \vert \cong \mathcal{D}^{2}$ 
	yields
	\begin{align*}
	\operatorname{rk(ker(}\partial_{1}))&=a+b+1,\\
	\operatorname{rk(Im(}\partial_{1}))&=f_{x_{1}}-1, \\
	\operatorname{rk(ker(}\partial_{0}) )&=f_{x_{1}}.
	\end{align*}
\end{remark}
In any case, we have shown the following proposition. 
\begin{proposition}
	\label{Link3D}
	Let $\Sigma$ be a complete 3-dimensional fan, $X_{\Sigma}$ the associated toric variety and $\mathcal{L}_{x}$ be the link of $x \in (X_{\Sigma})_{0}$. Let $f_{x_{1}}$ be the number of 1-dimensional faces of the dual polyhedron $\mathcal{P}$ to $\Sigma$ with $x$ as a proper face. Then the Betti numbers of $\mathcal{L}_{x}$ are given by
	\begin{align*}
	\operatorname{rk(}H_{5}(\mathcal{L}_{x};\mathbb{Q}))&=1,\\
	\operatorname{rk(}H_{4}(\mathcal{L}_{x};\mathbb{Q}))&=0,\\
	\operatorname{rk(}H_{3}(\mathcal{L}_{x};\mathbb{Q}) )&=f_{x_{1}}-3,\\
	\operatorname{rk(}H_{2}(\mathcal{L}_{x};\mathbb{Q}))&=f_{x_{1}}-3,\\
	\operatorname{rk(}H_{1}(\mathcal{L}_{x};\mathbb{Q}))&=0,\\
	\operatorname{rk(}H_{0}(\mathcal{L}_{x};\mathbb{Q}))&=1.\\
	\end{align*} 
\end{proposition}

\begin{corollary}
	A point $x \in (X_{\Sigma})_{0}$ is rationally singular if $x$ in $\mathcal{P}$ has more than three 1-dimensional neighboring faces.
\end{corollary}
\begin{proposition}\label{SmoothLink}
	Let $X_{\Sigma}$ be a 4-dimensional toric variety associated to a complete fan $\Sigma$ with the stratification $X_{\Sigma}=X_{4} \supset X_{0}$. Let $\nu \in X_{0}$ be a point whose link is an integral homology sphere. Then the link of $\nu$, $\mathcal{L}_{\nu}$, is in fact homeomorphic to $\mathcal{S}^{3}$.
\end{proposition}
\begin{proof}
	Let $\mathcal{P}$ be the dual polytope to $\Sigma$. Let $\sigma \in \Sigma$ be the cone dual to $\nu$. The assumption on the link of $\nu$ implies that $H_{1}(\mathcal{L}_{\nu}; \mathbb{Z}) = 0$. Hence, the determinant of the two generators of $\sigma$ vanishes, and they form a basis for $\mathbb{Z}^{2}$. It follows that we can consider the images of the generators of $\sigma$ under the projection map $\pi : \mathbb{R}^{2} \longrightarrow \sfrac{\mathbb{R}^{2}}{\mathbb{Z}^{2}}$ to be 1-cells of the minimal CW-structure of $T^{2}$, which is attached to the interior of $\mathcal{I} \cong \vert \mathcal{M}_{\nu} \vert$. Let $\mathcal{S}^1_1$ and $\mathcal{S}^1_2$ be the image of the first and second generators of $\sigma$, respectively. Since the generators form a basis, we can set $T^{2} \cong \mathcal{S}^1_1 \times \mathcal{S}^1_2$ to be the torus attached to $\operatorname{int}(\mathcal{I})$. Let $p_{\mathcal{L}_{\nu}}: \mathcal{L}_{\nu} \longrightarrow \mathcal{M}_{\nu}$ be the projection defined earlier. From the construction of the link, it follows that we have either $p_{\mathcal{L}_{\nu}}^{-1}(\{0\})=\mathcal{S}^{1}_{1}$
	and $p_{\mathcal{L}_{\nu}}^{-1}(\{1\})=\mathcal{S}^{1}_{2}$ or $p_{\mathcal{L}_{\nu}}^{-1}(\{0\})=\mathcal{S}^{1}_{2}$
	and $p_{\mathcal{L}_{\nu}}^{-1}(\{1\})=\mathcal{S}^{1}_{1}$. Hence, the link $\mathcal{L}_{\nu}$ is homeomorphic to the join $ \mathcal{S}^{1}_{1} \ast \mathcal{S}^{1}_{2} \cong S^{3}$, which proves our claim.
\end{proof}

\begin{proposition}
	Let $X_{\Sigma}$ be a 6-dimensional toric variety associated with a complete fan $\Sigma$ with the stratification $X_{\Sigma}=X_{6} \supset X_{2} \supset X_{0}$. If $\nu \in X_{2}- X_{0}$ is a point whose link is an integral homology sphere, then the link is in fact homeomorphic to $S^3$. If $\nu \in X_{0}$ is a point whose link is an integral homology sphere, then the link of $\nu$, $\mathcal{L}_{\nu}$, is in fact homeomorphic to $S^5$.
\end{proposition}
\begin{proof}
	If $\nu \in X_{2}-X_{0}$, the statement follows from the proof of Proposition \ref{SmoothLink}. Now, let $\nu \in X_{0}$ be a point whose link is an integral homology sphere. From the assumption, it follows that the base of the link, $\vert \mathcal{M}_{\nu} \vert$, is a 2-simplex. Let $\left( \begin{array}{c}
		n_{1} \\ m_{1} \\ l_{1}
	\end{array} \right)$, $\left( \begin{array}{c}
		n_{2} \\ m_{2} \\ l_{2}
	\end{array} \right)$, and $\left( \begin{array}{c}
		n_{3} \\ m_{3} \\ l_{3}
	\end{array} \right)$ be the generators of the dual cone to $\nu$. For the moment, we assume that all entries are non-zero. Since $H_{3}(\mathcal{L}_{\nu},\mathbb{Z})=0$, the image of the boundary operator 
	\begin{align*}
		\partial_{4}= \left( \begin{array}{ccc}
			0 & 0 & 0  \\
			0 & 0 & 0  \\
			0 & 0 & 0  \\
			n_{1} & m_{1} & l_{1} \\
			n_{1} & m_{1} & l_{1} \\
			n_{2} & m_{2} & l_{2} \\
			n_{2} & m_{2} & l_{2} \\	
			n_{3} & m_{3} & l_{3} \\
			n_{3} & m_{3} & l_{3} \\
		\end{array} \right)
	\end{align*}
	is isomorphic to $\mathbb{Z}^3$. It is the case if and only if 
	\begin{align*}
		\operatorname{det} \left( \begin{array}{ccc}
			n_{1} & m_{1} & l_{1} \\
			n_{2} & m_{2} & l_{2} \\
			n_{3} & m_{3} & l_{3} \\
		\end{array}  \right)=1.
	\end{align*}
	This means that the generators of the dual cone form a basis of $\mathbb{Z}^{3}$. We attach the torus $T^3 \cong S_{1}^{1} \times S^{1}_{2} \times S_{3}^{1} $ to the interior of $\vert \mathcal{M}_{\nu} \vert$, where $S^{1}_{i}$ is the image of $\left( \begin{array}{c}
		n_{i} \\ m_{i} \\ l_{i} \\
	\end{array} \right)$ under the map $\mathbb{R}^{3} \longrightarrow \sfrac{\mathbb{R}^{3}}{\mathbb{Z}^{3}}$ for $i=1,2,3$. On each 1-dimensional face of $\vert \mathcal{M}_{\nu} \vert$, we collapse one of the $S^1_{i}$ in $T^3$ to a point. To each vertex of $\vert \mathcal{M}_{\nu} \vert$, we attach one of the $S^1_{i}$. From the construction, it follows that $\mathcal{L}_{\nu} \cong S^{1} \ast S^{3} \cong S^{5}$. We illustrate the above considerations in Figure \ref{LinkSphere}.
	\begin{figure}[H]
		\centering
		\includegraphics[width=.55\linewidth]{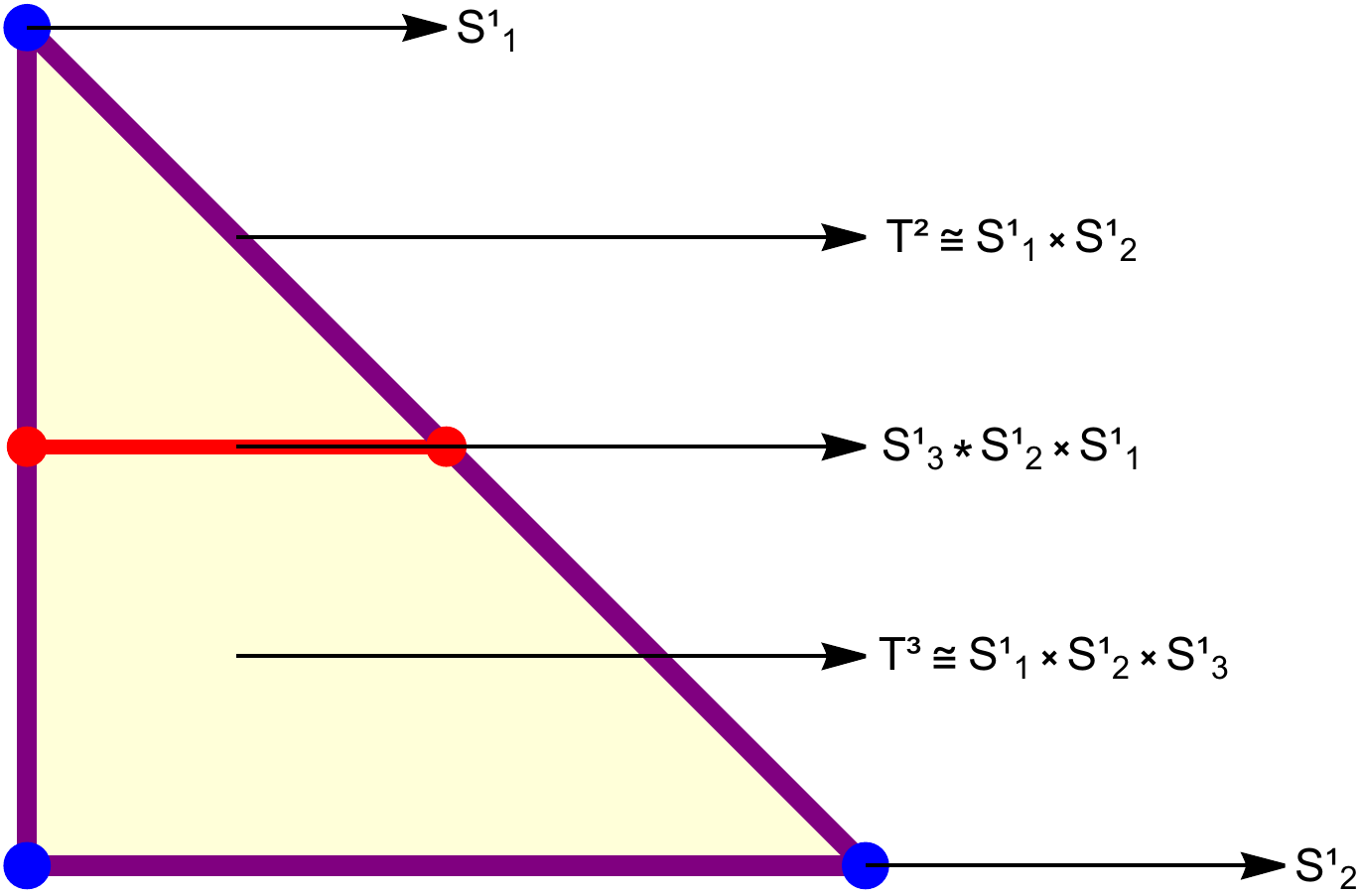}
		\caption{The link of a smooth point $\nu \in X_{0}$ is homeomorphic to $S^{5}$}
		\label{LinkSphere}
	\end{figure}
	Now, let one of the generators have at least one zero entry. The boundary operator $\partial_{5}$ has five rows with at least one non-zero entry. In this case, the rest of the proof goes along the same lines as above.
\end{proof}

\section{\texorpdfstring{$\mathbb{Q}$}--Homology Pseudomanifolds, Lefschetz Duality and Intersection Homology}\label{Qpsmfd}

We have seen in previous sections that the link of a connected component 
of a  4-co-dimensional stratum in a toric variety is a rational homology 3-sphere. 
For instance, consider a 4-dimensional toric variety. 
As described earlier, the link of each isolated singularity is a 
rational homology 3-sphere. Such a toric variety satisfies Poincaré duality rationally. 
This suggests to consider strata with rational homology spherical links
as homologically regular.
We will thus introduce the notion of $\mathbb{Q}$-homology pseudomanifolds 
(or $\mathbb{Q}$-pseudomanifold for short), which is related to the 
concept of homology stratification considered by 
Rourke and Sanderson in \cite{rourkesanderson}. 

\begin{definition}[$\mathbb{Q}$-pseudomanifold]
	We call a topological space $X$ with filtration
	\begin{align}\label{Qstra}
	X \supset X_{i} \supset  X_{i-1} \supset \dots \supset X_{0} \supset X_{-1}= \emptyset
	\end{align}
	a $\mathbb{Q}$-\textbf{pseudomanifold} if 
	(\ref{Qstra}) can be augmented to a filtration of the form
	\begin{align*}
	X \supset X_{i+k}  \supset \dots \supset X_{i+1} \supset  X_{i} \supset \dots \supset X_{0} \supset X_{-1}= \emptyset,
	\end{align*} 
such that $X$ is a pseudomanifold with respect to it and the 
link of each connected component of $X_{i+(j+1)}-X_{i+j}$ for 
$j=0, \dots, k-1$ in $X$ is a rational homology sphere. 
We call the filtration (\ref{Qstra}) a 
\textbf{$\mathbb{Q}$-homology stratification} 
(or \textbf{$\mathbb{Q}$-stratification} for short) 
of $X$.
\end{definition}

\begin{example}
A 6-dimensional toric variety $X$ with $X=X_{6} \supset X_{0}$
is a $\mathbb{Q}$-pseudomanifold, but generally not a 
stratified pseudomanifold with the respect to this stratification, since $X_6 - X_0$ is usually not a manifold,
only a $\mathbb{Q}$-homology manifold.
\end{example}

Definition \ref{PSMFD} covers pseudomanifolds without boundary. 
We need to extend that definition to incorporate boundaries.
The following definition is based on the corresponding PL-theoretic
definition as formulated by Goresky and MacPherson in \cite[Section~5.2, Page~155]{goresky1980intersection}.
\begin{definition}[Pseudomanifold with boundaries]\label{PSMwB}
	An \textbf{$n$-dimensional stratified pseudomanifold with boundary} 
	is a pair 
	$(X,\mathcal{B})$ of topological spaces such that $\mathcal{B}$ is closed in $X$,
	together with a filtration $\{ X_i \}$ on $X$ such that:
	\begin{enumerate}
		\item The space $X-\mathcal{B}$ with the induced filtration $(X-\mathcal{B})_i = (X-\mathcal{B})\cap X_i$ is an $n$-dimensional stratified pseudomanifold.
		\item The space $\mathcal{B}$ with the induced filtration 
		$\mathcal{B}_{i-1}=\mathcal{B} \cap X_i$ is an $(n-1)$-dimensional stratified pseudomanifold.
		\item The topological space $\mathcal{B}$ has an \textit{open filtered collar neighborhood} in $X$, i.e. there exists an open neighborhood 
		$\mathcal{N}$ of $\mathcal{B}$ and a filtered homeomorphism $\mathcal{N} \longrightarrow [0,1) \times \mathcal{B}$ (where $[0,1)$ is given the trivial filtration) that takes $\mathcal{B}$ to $\{0\} \times \mathcal{B}$ by the identity map.
	\end{enumerate}
	The topological space $\mathcal{B}$ is called the 
	\textbf{boundary} of $X$ and is also denoted by $\partial X$.
\end{definition}
The question arises to what extent the boundary $\mathcal{B}$ is intrinsic 
to the topology of the space. The following proposition from 
Friedman-McClure \cite{friedman2013cup} ensures that when there are 
no one-co-dimensional strata in $X$, the boundary $\mathcal{B}$ depends only on the underlying space $X$ and not the choice of a specific filtration 
(among those without one co-dimensional stratum).
\begin{proposition}
	Let $(X, \partial X)$ and $(X^\prime, \partial X^\prime)$ be 
	equidimensional stratified pseudomanifolds with boundaries and 
	no one-co-dimensional strata, 
	and let $h: X \longrightarrow X^\prime$ be a homeomorphism 
	(which is not required to be filtration preserving). 
	Then $h$ takes $\partial X$ onto $\partial X^\prime$.
\end{proposition}
A \textbf{closed} pseudomanifold is a compact pseudomanifold with empty boundary.\\

In a 6-dimensional toric variety $X$, let $\mathcal{C}_{x}$ be conical neighborhoods 
of $x \in X_{0}$, such that 
$\mathcal{C}_{x} \cap \mathcal{C}_{x^\prime} = \varnothing$ for all 
$x,x^\prime \in X_{0},$ $x\not= x'$. 
Removing somewhat smaller open cones in all $\mathcal{C}_{x}$ from $X$ gives us a pseudomanifold with boundary 
$(\mathcal{M},\partial\mathcal{M})$, where now all links are rational homology 
spheres. 

\begin{definition}\label{Cout}
We call this process of removing a disjoint union of
cone-like neighborhoods of $X_{0}$ \textbf{cutting out} the 0-dimensional stratum of $X$.
\end{definition}

We shall prove sheaf-theoretically that compact oriented pseudomanifolds
with boundary whose interior is a rational homology manifold satisfy
Lefschetz duality for ordinary rational cohomology.
We begin with some basic facts on direct and inverse limits.
Let $(I,\leq)$ be an (upwards) directed set, indexing direct systems
$\{ G_i \}_{i\in I},$ $\{ H_i \}_{i\in I}$
of abelian groups.
Let $\{ f_i \}: \{ G_i \} \to \{ H_i \},$ 
$f_i: G_i \to H_i,$ be a map of direct systems.
If every $f_i$ is an isomorphism, then
\[ \underset{\longrightarrow}{\lim}_I f_i:
   \underset{\longrightarrow}{\lim}_I G_i 
   \longrightarrow
   \underset{\longrightarrow}{\lim}_I H_i  \]
is an isomorphism. (This is a consequence of 
the direct limit being a functor.)
If $J\subset I$ is a cofinal subset, then
the inclusion induces an isomorphism
\[ \underset{\longrightarrow}{\lim}_J G_j
   \stackrel{\simeq}{\longrightarrow}
   \underset{\longrightarrow}{\lim}_I G_i \]
of direct limits. The diagram
\[ \xymatrix{
\underset{\longrightarrow}{\lim}_J G_j 
  \ar[d]_{\underset{\longrightarrow}{\lim}_J f_j} 
   \ar[r]^\simeq & \underset{\longrightarrow}{\lim}_I G_i 
   \ar[d]^{\underset{\longrightarrow}{\lim}_I f_i} \\
\underset{\longrightarrow}{\lim}_J H_j 
  \ar[r]^\simeq & \underset{\longrightarrow}{\lim}_I H_i  
} \]
commutes. This has the following consequence.
If every $f_j,$ $j\in J,$ is an isomorphism, then 
their direct limit is an isomorphism, and, by commutativity of
the diagram, the direct limit of all $f_i$, $i\in I,$ is an
isomorphism as well. We make use of this principle at various
points in the ensuing arguments.

Similar remarks apply to inverse limits.
Let $(I,\leq)$ be an (upwards) directed set, indexing inverse systems
$\{ G_i \}_{i\in I},$ $\{ H_i \}_{i\in I}$
of abelian groups.
Let $\{ f_i \}: \{ G_i \} \to \{ H_i \},$ 
$f_i: G_i \to H_i,$ be a map of inverse systems.
If every $f_i$ is an isomorphism, then
\[ \underset{\longleftarrow}{\lim}_I f_i:
   \underset{\longleftarrow}{\lim}_I G_i 
   \longrightarrow
   \underset{\longleftarrow}{\lim}_I H_i  \]
is an isomorphism. (This is a consequence of 
the inverse limit being a functor.)
If $J\subset I$ is a cofinal subset, then
projection to components indexed by $J$ induces an isomorphism
\[ \underset{\longleftarrow}{\lim}_I G_i
   \stackrel{\simeq}{\longrightarrow}
   \underset{\longleftarrow}{\lim}_J G_j \]
of inverse limits.
The diagram
\[ \xymatrix{
\underset{\longleftarrow}{\lim}_I G_i 
  \ar[d]_{\underset{\longleftarrow}{\lim}_I f_i} 
   \ar[r]^\simeq & \underset{\longleftarrow}{\lim}_J G_j 
   \ar[d]^{\underset{\longleftarrow}{\lim}_J f_j} \\
\underset{\longleftarrow}{\lim}_I H_i 
  \ar[r]^\simeq & \underset{\longleftarrow}{\lim}_J H_j  
} \]
commutes. Again, we deduce that
if every $f_j,$ $j\in J,$ is an isomorphism, then 
their inverse limit is an isomorphism, and, by commutativity of
the diagram, the inverse limit of all $f_i$, $i\in I,$ is an
isomorphism as well.

The \emph{cohomological dimension} of a topological space $X$ over
a commutative Noetherian ring $R$ is the smallest
$n\in \mathbb{N} \cup \{ \infty \}$ such that
$H^i_c (U;A)=0$ for all open $U\subset X$, all sheaves $A$ 
of $R$-modules on $X$, and
all $i>n$. Topologically stratified pseudomanifolds of dimension $n$
have cohomological dimension $n$ (Borel \cite[p. 60]{borel}).
Recall that a topological space is called \emph{first-countable}
if each point has a countable neighborhood basis.
Pseudomanifolds with boundary are first-countable, as can be
seen by shrinking distinguished neighborhoods appropriately.

For the purposes of the present paper, we shall adopt the following
concept of rational homology manifold with boundary. 

\begin{definition}
A \textbf{rational homology $n$-manifold with boundary} is
an $n$-dimensional topologically stratified pseudomanifold 
$X$ with boundary $\partial X$ such that
$X-\partial X$ is a rational homology $n$-manifold. We will also refer to the pair $(X, \partial X)$ as a $\mathbb{Q}$-\textbf{manifold with boundary}.
\end{definition}

In particular, rational homology manifolds $(X,\partial X)$ with boundary
in the sense of this definition are paracompact Hausdorff spaces
which are locally compact and (strongly) locally contractible.
(A space is \emph{strongly locally contractible} if every point possesses
a neighborhood basis consisting of contractible sets.)
Since they are topologically stratified pseudomanifolds,
rational homology manifolds with boundary have finite cohomological 
dimension over $\rat$, and their cohomological dimension agrees with
their pseudomanifold dimension.
Note that in the above definition, we do not explicitly require $\partial X$
to be a rational homology $(n-1)$-manifold. This turns out to be true,
as we shall see in Proposition \ref{prop.bndryofrathomolmfdisrathomolmfd}.

The cohomology sheaf in degree $k\in \mathbb{Z}$ of 
a differential graded sheaf $S$ on a space $X$ will be denoted by $\mathbf{H}^k (S)$,
the hypercohomology by $\mathcal{H}^k (X;S)$.
If $S$ is concentrated in degree $0$, then 
$\mathcal{H}^k (X;S)$ is the sheaf cohomology group $H^k (X;S)$.
The constant sheaf with stalk $\rat$ on a space $X$ will be written
as $\rat_X$ and will frequently be considered as a differential graded
sheaf concentrated in degree $0$.

\begin{lemma} \label{lem.ratxadjiso}
Let $(X,\partial X)$ be a pseudomanifold with boundary and let
$i$ be the open inclusion $i: X-\partial X \hookrightarrow X$.
Then the canonical adjunction map
\[ \rat_X \longrightarrow Ri_* i^* \rat_X   \]
is a quasi-isomorphism.
\end{lemma}
\begin{proof}
We recall that the property of being a quasi-isomorphism can
be checked point by point on cohomology stalks.
Let $S$ be any differential sheaf on $X$.
The adjunction $S\to Ri_* i^* S$ restricts to the identity
$i^* S \to i^* S$ on $X-\partial X$ and is therefore a quasi-isomorphism at every 
point $x\in X-\partial X$. It remains to check that it is a quasi-isomorphism
at points $x$ in $\partial X$.

Let $U$ be an open neighborhood of $x\in\partial X$ in $X$.
By definition of a pseudomanifold with boundary,
$\partial X$ is collared in $X$. Thus there exists an open
neighborhood $N$ of $\partial X$ and a (filtered) homeomorphism
$\phi: N \to [0,1) \times \partial X$ which restricts to the identity 
$\partial X \rightarrow \{ 0 \} \times \partial X$.
Note that $\phi$ restricts to a homeomorphism
$\phi: N-\partial X \to (0,1) \times \partial X$.
The intersection $U \cap N$ is an open neighborhood of $x$ in $X$.
Thus $\phi (U\cap N)$ is an open neighborhood of $\phi (x) = (0,x)$
in $[0,1) \times \partial X$. By definition of the product topology,
there exists an open neighborhood $V\subset \partial X$ of $x$ in $\partial X$
and $\epsilon >0$ such that
$[0,\epsilon) \times V \subset \phi(U\cap N)$.
Set
\[ U' := \phi^{-1} ([0,\epsilon) \times V). \]
Then $U' \subset X$ is an open neighborhood of $\phi^{-1} (0,x)=x$ in $X$
and $U' \subset U$, for
\[ U' = \phi^{-1} ([0,\epsilon) \times V)
    \subset \phi^{-1} (\phi (U\cap N)) = U\cap N \subset U. \]
We have thus shown that every point $x\in \partial X$ has a neighborhood
basis consisting of open sets of the form $U'$.
We note that $U' -\partial X \subset U-\partial X$.

The sheafification $\mathbf{A}$ of a presheaf $U\mapsto A(U)$
has stalks
\[  \mathbf{A}_x = \underset{\longrightarrow}{\lim}~ A(U), \]
where the direct limit ranges over all open neighborhoods $U$ of $x$.
Since, by definition, $\mathbf{H}^k (S)$ is the sheafification of the
presheaf $U \mapsto H^k \Gamma (U;S)$ 
(Iversen \cite[p. 89]{iversen}),
the cohomology sheaf has stalks
\[ \mathbf{H}^k (S)_x = \underset{\longrightarrow}{\lim}~ H^k \Gamma (U;S). \]
We note that $\mathbf{H}^k (S)_x = H^k (S_x)$, as restriction to
stalks is an exact functor.
Thus for the constant sheaf $S=\rat_X$ on $X$,
\[ \mathbf{H}^0 (\rat_X)_x = H^0 (\rat_{X,x}) = \rat \]
and
\[ \mathbf{H}^k (\rat_X)_x = H^k (\rat_{X,x}) =0,~ k\not= 0. \]

An injective resolution $S\to I$ of $S$ is a quasi-isomorphism
and thus induces an isomorphism
$\mathbf{H}^k (S) \cong \mathbf{H}^k (I)$.
Therefore, the cohomology stalks may also be computed as
\[ \mathbf{H}^k (S)_x = \underset{\longrightarrow}{\lim}~ H^k \Gamma (U;I). \]
Now the groups $H^k \Gamma (U;I)$ are precisely the hypercohomology groups
$\mathcal{H}^k (U;S)$. We obtain the formula
\[ \mathbf{H}^k (S)_x = \underset{\longrightarrow}{\lim}~ \mathcal{H}^k (U;S). \]
The restriction of an injective sheaf to an open subset is injective
(\cite[p. 109, Cor. 6.10]{iversen}).
Therefore, $i^* S \to i^* I$ is an injective resolution of $i^* S$.
We deduce that $Ri_* i^* S = i_* i^* I$ and the
canonical adjunction morphism $S \to Ri_* i^* S$ is given by the composition
\[ S \longrightarrow I \longrightarrow i_* i^* I = Ri_* i^* S. \]
The cohomology stalks of $Ri_* i^* S$ are given by
\begin{align*}
\mathbf{H}^k (Ri_* i^* S)_x
&= \mathbf{H}^k (i_* i^* I)_x = 
  \underset{\longrightarrow}{\lim}~ H^k \Gamma (U; i_* i^* I) \\
&= \underset{\longrightarrow}{\lim}~ H^k \Gamma (U-\partial X; i^* I)
  = \underset{\longrightarrow}{\lim}~ \mathcal{H}^k (U-\partial X; S).
\end{align*}
At $x\in \partial X$, the map
\[ \mathbf{H}^k (S)_x \longrightarrow \mathbf{H}^k (Ri_* i^* S)_x \]
is hence given by the map
\[ \underset{\longrightarrow}{\lim}~ \mathcal{H}^k (U;S) \longrightarrow
   \underset{\longrightarrow}{\lim}~ \mathcal{H}^k (U-\partial X;S) \]
induced on direct limits by the restriction maps
\begin{equation} \label{equ.restrfulldirectsystem} 
\mathcal{H}^k (U;S) \longrightarrow
   \mathcal{H}^k (U-\partial X;S). 
\end{equation}   
Since an arbitrary neighborhood $U$ contains one of the form $U'$,
neighborhoods of the latter type are cofinal in the directed set of
all open neighborhoods. So the direct limit can be computed on neighborhoods
of type $U'$.
When $S=\rat_X$ is the constant sheaf, the collar homeomorphism $\phi$ identifies 
the restriction map
\[ \mathcal{H}^k (U';\rat_X) \longrightarrow
   \mathcal{H}^k (U'-\partial X; \rat_X) \] 
with the restriction map
\begin{equation} \label{equ.sheafcoheps} 
\mathcal{H}^k ([0,\epsilon)\times V; \rat_X) \longrightarrow
   \mathcal{H}^k ((0,\epsilon)\times V; \rat_X). 
\end{equation}   
Consider the commutative diagram
\[ \xymatrix{
(0,\epsilon)\times V \ar[rd] \ar@{^{(}->}[rr] & & 
        [0,\epsilon)\times V \ar[ld] \\
& V, &
} \]
where the maps to $V$ are the second factor projections.
These projections have contractible fibers and are homotopy equivalences.
It follows that the horizontal inclusion is a homotopy equivalence.
On singular cohomology, which is homotopy invariant, we get an induced
restriction isomorphism
\begin{equation} \label{equ.singcoheps}
H^k ([0,\epsilon)\times V; \rat) 
   \stackrel{\simeq}{\longrightarrow}
   H^k ((0,\epsilon)\times V; \rat). 
\end{equation}   
For semi-locally contractible topological spaces, there
is a natural isomorphism between singular cohomology $H^* (-;\rat)$
and sheaf cohomology $H^* (-;\rat_X)$ (Sella \cite{sella}).
Locally contractible spaces are in particular semi-locally contractible.
The space $[0,\epsilon)\times V$ is locally contractible, since $V$ is 
an open subset of the locally contractible space $\partial X$.
The natural isomorphism between singular and sheaf cohomology thus
identifies the maps (\ref{equ.singcoheps}) and (\ref{equ.sheafcoheps}).
Therefore, the latter is an isomorphism as well.
The basic principle on direct limits recalled earlier
implies that the direct limit of the maps
(\ref{equ.restrfulldirectsystem}) is an isomorphism.
\end{proof}

We shall use the term \emph{cohomologically constructible}, a property
of differential graded sheaves, in the sense of Borel
\cite[p. 69]{borel}. We will not recall the definition here, but point
out that this notion does not require a stratification.
According to \cite[p. 79, Cor. 3.11 (i)]{borel},
the constant sheaf on a stratified pseudomanifold without boundary
is cohomologically constructible. We need to extend this statement
to pseudomanifolds with boundary:
\begin{lemma} \label{lem.constcohomconstr}
Let $(X,\partial X)$ be a topologically stratified pseudomanifold
with boundary. Then the constant sheaf $\rat_X$ on $X$ is
cohomologically constructible.
\end{lemma}
\begin{proof}
Since the interior $X-\partial X$ is a pseudomanifold without boundary,
$\rat_X$ is cohomologically constructible at every point
of the interior by \cite[p. 79, Cor. 3.11 (i)]{borel}.
It remains to verify that $\rat_X$ is cohomologically constructible
at boundary points $x\in \partial X$.
Using a collar as in the proof of Lemma \ref{lem.ratxadjiso}, there exists a
neighborhood basis $\{ U_i \}_{i\in I}$ of $x$ in $X$ consisting of
open sets $U_i$ that are homeomorphic to $[0,\epsilon_i) \times V_i$,
$V_i \subset \partial X$, $V_i$ open and contractible, $\epsilon_i >0$.
If $i<j$ so that $U_i \supset U_j$, then the restriction map
\[ \mathcal{H}^k (U_i;\rat_X) \longrightarrow
   \mathcal{H}^k (U_j;\rat_X) \]
is an isomorphism, since in the commutative diagram
\[ \xymatrix{
U_j \ar[rd] \ar@{^{(}->}[rr] & & 
        U_i \ar[ld] \\
& \pt &
} \]
the constant maps to the point are homotopy equivalences, and thus
the horizontal inclusion is a homotopy equivalence.   
(Here, as earlier, we may identify singular and sheaf cohomology.)
Thus, the direct system $\{ \mathcal{H}^k (U_i;\rat_X) \}_{i\in I}$
is constant, hence also essentially constant in the sense of
Borel's definition \cite[p. 68]{borel}.
Since $\{ U_i \}_{i\in I}$ is cofinal in the directed set of 
all open neighborhoods of $x$, the direct system
$\{ \mathcal{H}^k (U;\rat_X) \}$ over all neighborhoods of $x$
is also essentially constant (\cite[p. 69]{borel}).
We observe furthermore that the direct limit is finitely generated
as it is one-dimensional for $k=0$ and $0$ otherwise. 
This proves condition CC2 of Borel. 
 
We shall next treat the dual case of inverse systems given by
compactly supported cohomology.
Thus we consider the inverse system with groups
$\mathcal{H}^k_c (U_i; \rat_X),$ $i\in I,$ where $\{ U_i \}$
is the same neighborhood basis of $x\in \partial X$ as before. For $i<j$,
so that $U_i \supset U_j$, the transition maps are given by
extension by zero
\[  \mathcal{H}^k_c (U_j;\rat_X) \longrightarrow
   \mathcal{H}^k_c (U_i;\rat_X).   \]
We claim that in fact all of these groups vanish.
Indeed, we may use the K\"unneth formula
\[ R\Gamma_c^\bullet (A\times B;\rat) =
   R\Gamma_c^\bullet (A;\rat) \otimes_\rat
   R\Gamma_c^\bullet (B;\rat) \] 
valid for finite dimensional locally compact spaces $A,B$,
\cite[p. 323]{iversen}.
The spaces $[0,\epsilon_i)$ and $V_i$ satisfy these assumptions
and therefore
\[ \mathcal{H}^*_c (U_i;\rat_X) \cong
   \mathcal{H}^*_c ([0,\epsilon_i) \times V_i;\rat_X) \cong
   \mathcal{H}^*_c ([0,\epsilon_i);\rat) \otimes 
      \mathcal{H}^*_c (V_i;\rat). \]
Now, the compactly supported cohomology with constant coefficients 
of a closed halfspace vanishes, \cite[p. 189, 8.4]{iversen}.
So $\mathcal{H}^*_c ([0,\epsilon_i);\rat)=0$ and hence
$\mathcal{H}^*_c (U_i;\rat_X)=0$ as claimed.
The inverse system
$\{ \mathcal{H}^k_c (U_i; \rat_X) \}_{i\in I}$ is thus constant,
hence also essentially constant.
Since $\{ U_i \}_{i\in I}$ is cofinal in the directed set of 
all open neighborhoods of $x$, the inverse system
$\{ \mathcal{H}^k_c (U;\rat_X) \}$ over all open neighborhoods of $x$
is also essentially constant (\cite[p. 69]{borel}).
The inverse limit is zero, so in particular
finitely generated. This proves condition CC1 of Borel.

We have verified Borel's conditions CC1 and CC2. These imply the
remaining conditions CC3 and CC4, as pointed out by Borel.
\end{proof}

Let $(X,\partial X)$ be a rational homology $n$-manifold with boundary.
By definition, $X-\partial X$ is a rational homology $n$-manifold
(without boundary). Thus the dualizing complex 
$\mathbb{D}_{X-\partial X}[-n]$ is naturally quasi-isomorphic to
the orientation sheaf $\operatorname{or}_{X-\partial X}$ of $X-\partial X$.
\begin{definition}
The homology manifold $(X,\partial X)$ with boundary is called \textbf{orientable} if there
exists an isomorphism $\operatorname{or}_{X-\partial X} \cong \rat_{X-\partial X}$.
A choice of such an isomorphism, if it exists, is called an
\textbf{orientation}. 
\end{definition}

Let $(X,\partial X)$ be an oriented rational homology $n$-manifold
with boundary.
The orientation induces a quasi-isomorphism
\[ \mathbb{D}_{X-\partial X}[-n] \cong
   \operatorname{or}_{X-\partial X} \cong \rat_{X-\partial X}. \]
Using Verdier's dualizing functor $\mathcal{D}$, we may express the
dualizing complex as
$\mathbb{D}_{X-\partial X} = \mathcal{D} \rat_{X-\partial X}$
(\cite[p. 90]{gmih2}).
The orientation therefore yields a self-duality isomorphism
\[ (\mathcal{D} \rat_{X-\partial X}) [-n] \cong \rat_{X-\partial X}.  \]

\begin{lemma} \label{lem.switchdandristar}
(Borel \cite[p. 137, Prop. 8.8 (3)]{borel}.)
Let $Y$ be a first-countable locally compact space.
Let $i:U\subset Y$ be an open inclusion and
let $B$ be a differential graded sheaf on $U$ such that
$Ri_* B$ is cohomologically constructible. Then
\[ \mathcal{D} Ri_* B = i_! \mathcal{D} B.  \] 
\end{lemma}

We apply this lemma to $i:X-\partial X \hookrightarrow X$, observing that 
$X$ is locally compact and first-countable. 
We take $B=\rat_{X-\partial X}$.
By Lemma \ref{lem.ratxadjiso},
\[ Ri_* B = Ri_* i^* \rat_X \cong \rat_X, \]
the constant sheaf on $X$, which is cohomologically constructible
by Lemma \ref{lem.constcohomconstr}.
By Lemma \ref{lem.switchdandristar},
\[ \mathcal{D} Ri_* \rat_{X-\partial X} 
  = i_! \mathcal{D} \rat_{X-\partial X}. \]
Using the above self-duality isomorphism,
\begin{align*}
i_! i^* \rat_X
&= i_! \rat_{X-\partial X} \cong
   i_! \mathcal{D} \rat_{X-\partial X} [-n] \\
&= \mathcal{D} Ri_* \rat_{X-\partial X} [-n]
  = \mathcal{D} Ri_* i^* \rat_X [-n] \\
&= \mathcal{D} \rat_X [-n].    
\end{align*}
On hypercohomology, there is an induced isomorphism
\[ \mathcal{H}^k (X; i_! i^* \rat_X) \cong
   \mathcal{H}^{k-n} (X; \mathcal{D} \rat_X).  \]
The global effect of dualizing a differential graded sheaf is
\[ \mathcal{H}^{k-n} (X; \mathcal{D} \rat_X)
  = \Hom (\mathcal{H}^{n-k}_c (X;\rat_X), \rat). \]
(This holds on any locally compact space which has finite
cohomological dimension over $\rat$.)
For compact $X$, 
$\mathcal{H}^{n-k}_c (X;\rat_X) = \mathcal{H}^{n-k} (X;\rat_X),$
so we obtain an isomorphism
\[ \mathcal{H}^{k-n} (X; \mathcal{D} \rat_X)
  = \Hom (\mathcal{H}^{n-k} (X;\rat_X), \rat). \]
Composing, we receive a duality isomorphism
\[ \mathcal{H}^k (X; i_! i^* \rat_X) \cong
  \Hom (\mathcal{H}^{n-k} (X;\rat_X), \rat) \]
for compact $X$.
The left hand group is nothing but the relative group
$H^k (X,\partial X;\rat)$
(Iversen \cite[p. 249]{iversen}), 
while the universal coefficient theorem
identifies the right hand group with the homology group
$H_{n-k} (X;\rat)$. The resulting isomorphism
\[ H^k (X,\partial X;\rat) \cong H_{n-k} (X;\rat) \]
is Lefschetz duality for
an oriented compact rational homology $n$-manifold
$(X,\partial X)$ with boundary.
We have proved:
\begin{theorem}\label{LFduality} \label{thm.lefschetzdualhomolmfdbndry}
Let $(X,\partial X)$ be an oriented compact rational homology $n$-manifold
with boundary. Then $(X,\partial X)$ has a Lefschetz duality isomorphism
\[ H^k (X,\partial X;\rat) \cong H_{n-k} (X;\rat).  \]
\end{theorem}

\begin{corollary}\label{LefDualityCor}
Let $X$ be an $n$-dimensional compact oriented $\mathbb{Q}$-pseudomanifold with 
$\mathbb{Q}$-stratification
	\begin{align*}
	X=X_{n}  \supset X_{0}.
	\end{align*}
Let $(\mathcal{M}, \partial \mathcal{M})$ be the pseudomanifold with boundary
obtained by cutting out the 0-dimensional stratum in the $\mathbb{Q}$-stratification. 
Then $(\mathcal{M}, \partial \mathcal{M})$ satisfies Lefschetz duality 
with respect to ordinary rational homology.
\end{corollary}
\begin{proof}
	The pair $(\mathcal{M}, \partial \mathcal{M})$ is a compact oriented rational homology manifold with boundary. The result follows from Theorem \ref{thm.lefschetzdualhomolmfdbndry}.
\end{proof}
\begin{corollary}\label{PDRationally}
Let $X$ be an $n$-dimensional closed oriented $\mathbb{Q}$-pseudomanifold 
with $\mathbb{Q}$-stratification $X=X_{n}$. (Such an $X$ is a rational homology manifold.) Then $X$ satisfies Poincar\'e duality rationally. 
\end{corollary}

\begin{proposition} \label{prop.bndryofrathomolmfdisrathomolmfd}
If $(X,\partial X)$ is a rational homology $n$-manifold with boundary,
then $\partial X$ is a rational homology $(n-1)$-manifold (without boundary).
\end{proposition}
\begin{proof}
Sheaf-theoretically, this is a direct consequence of
\cite[Chapter 4]{banaglmem} and Lemma \ref{lem.ratxadjiso}.
The subspace $\partial X$ has a stratum preserving collar neighborhood in $X$. This implies that links in $\partial X$ can be taken to be links in $X - \partial X$. 
\end{proof}

\begin{example}
	Let $X$ be a 4-dimensional toric variety. Then $X$ satisfies Poincar\'e duality rationally.
\end{example}
\begin{example}\label{6DimTV}
	Let $X$ be a 6-dimensional toric variety. The pseudomanifold obtained by cutting out $X_{0}$ satisfies Lefschetz duality, rationally.
\end{example}

Concerning intersection homology, we shall use the following standard notation.
\begin{definition}(Goresky-MacPherson.)	
A \textbf{perversity} 
	\begin{align*}
	\bar{p}: \mathbb{Z}_{ \geq 2} \longrightarrow \mathbb{Z}
	\end{align*}
	is a function such that
	$\bar{p}(2)=0 \; \text{and} \;
	\bar{p}(k+1)- \bar{p}(k) \in \{1,0\}.	$
	The \textbf{complementary perversity} $\bar{q}$ of $\bar{p}$ 
	is the one with $\bar{p}(k)+\bar{q}(k)=k-2$.
\end{definition}

We denote the $i$-th intersection homology group of $X$ with coefficients in $A$ and the perversity $\bar{p}$ with
$I^{\bar{p}} H_{i}(X;A).$
Similarly, for the $i$-th intersection cohomology group, we write
$I_{\bar{p}}H^{i}(X;A).$
For an introduction to the theory of intersection homology, the reader may consult \cite{friedman2020singular} by Friedman, \cite{kirwan2006introduction} by 
Kirwan and Woolf, or \cite{BanaglTISS}.

\section{Intersection Spaces and Isolated Singularities}

In this section, first, we briefly recall the theory of intersection spaces introduced by the first named author in \cite{banagl2010intersection}. For toric varieties, we only consider the duality of Betti numbers for the middle perversity. Throughout this section, we work with rational coefficients unless otherwise specified.

\begin{definition}\label{nseq}
	Let $n$ be a natural number. A CW complex $\mathcal{K}$ is called \textbf{rationally $n$-segmented} if it contains a sub-complex $\mathcal{K}_{<n} \subset \mathcal{K}$ such that
	$
	H_{r}(\mathcal{K}_{<n})=0 \;\; \text{for} \; r \geq n \;
	\text{and} \;
	i_{\ast}:H_{r}(\mathcal{K}_{<n}) \xrightarrow{\;\;\cong\;\;} H_{r}(\mathcal{K}) \; \text{for} \; r <n,
	$
	where $i$ is the inclusion of $\mathcal{K}_{<n}$ into $\mathcal{K}$.
\end{definition}
Given any CW complex $\mathcal{K}$ and natural number $n$, there exists a \textbf{homology $n$-truncation} (Moore approximation) $f:\mathcal{K}_{<n} \to\mathcal{K}$, i.e. a continuous map from a CW complex $\mathcal{K}_{<n}$ to $\mathcal{K}$ such that $f_{\ast}:H_{\ast}(\mathcal{K}_{<n}; \mathbb{Z}) \to H_{\ast}(\mathcal{K}; \mathbb{Z})$ is an isomorphism for $\ast <n$ and $H_{\ast}(\mathcal{K}_{<n}; \mathbb{Z}) = 0$ for $\ast \geq n$, see \cite{banagl2010intersection} (for the simply connected case) and Wrazidlo \cite{WrazidloMA} (in general). The map $f$ cannot in general be taken to be a sub-complex inclusion. We shall prove in Proposition \ref{Ltrunc} that the 5-dimensional links in a 6-dimensional toric variety are rationally 3-segmented.

Let $X \supset X_{0}$ be an $n$-dimensional compact oriented topological pseudomanifold with isolated singularities. For $x_{i} \in X_{0}$, let $\mathcal{L}_{i}$ be an associated link. Although two different links associated with the same isolated singularity need not be homeomorphic, we shall justify that the homology of links is well-defined. Let $\mathcal{L}_{i}$ be as above and $\mathcal{U}$ an open distinguished neighborhood of $x_{i}$. Then there exists a homeomorphism $\phi: \mathcal{U} \xrightarrow{\;\;\cong\;\;} {\mathcal{C}}(\mathcal{L})$ which sends $x_{i}$ to the cone point, $c$. We define ${\mathcal{C}}_{1/2}(\mathcal{L}_{i}) \subset {\mathcal{C}}(\mathcal{L}_{i})$ as the open cone over $\mathcal{L}_{i}$ by considering the interval $[0,\frac{1}{2}) \subset [0,1)$. Let $\mathcal{U}_{1/2} \subset \mathcal{U}$ be the preimage of ${\mathcal{C}}_{1/2}(\mathcal{L}_{i})$ under the homeomorphism $\phi$. By employing the excision axiom, we have
$
	H_{\ast}(X, X-x_{i}) \cong H_{\ast}(X-(X-\mathcal{U}_{1/2},(X-x_{i})-(X-\mathcal{U}_{1/2}))=H_{\ast}(\mathcal{U}_{1/2},\mathcal{U}_{1/2}-x_{i}).
$
By using the homeomorphism $\phi$ we arrive at $H_{\ast}(X,X - x_{i}) \cong H_{\ast}({\mathcal{C}}_{1/2}(\mathcal{L}_{i}),{\mathcal{C}}_{1/2}(\mathcal{L}_{i})-\{c\})$. As ${\mathcal{C}}_{1/2}(\mathcal{L}_{i}) \simeq \ast$ and ${\mathcal{C}}_{1/2}(\mathcal{L}_{i})-\{c\} \simeq \mathcal{L}_{i}$, the connecting homomorphism in the long exact sequence of the relative homology groups of the pair $({\mathcal{C}}_{1/2}(\mathcal{L}_{i}),{\mathcal{C}}_{1/2}(\mathcal{L}_{i})-\{c\})$ is an isomorphism. If $\mathcal{L}^{\prime}_{i}$ is another link of $x_{i}$, then the above procedure yields $H_{\ast}(\mathcal{L}_{i}) \cong H_{\ast}(\mathcal{L}^{\prime}_{i})$. Hence the homology of links is well-defined.

 Assume that all links $\mathcal{L}_{i}$ can be equipped with CW structures such that they are rationally $k$-segmented, where $k=n-1-\bar{p}(n),$ for the perversity $\bar{p}$. Let $(\mathcal{M}, \partial \mathcal{M} )$ be the manifold with boundary obtained by cutting out all isolated singularities of $X$. Let $(\mathcal{L}_{i})_{<k}$ be a sub-complex of $\mathcal{L}_{i}$ that truncates the homology. Then, we have
$
\partial \mathcal{M}= \bigsqcup_{i} \mathcal{L}_{i}.
$
Let
$
\mathcal{L}_{<k}= \bigsqcup_{i} (\mathcal{L}_{i})_{<k}
$
and define a homotopy class $g$ by the composition
$
g: \mathcal{L}_{<k} \xhookrightarrow{\phantom{-}\phantom{-}} \partial \mathcal{M} \xhookrightarrow{\phantom{-} \operatorname{incl.}  \phantom{-}} \mathcal{M}$.
For the purposes of the present paper we adopt the following definition:   
\begin{definition}
	The perversity $\bar{p}$ rational intersection space $I^{\bar{p}}X$ of $X$ is defined to be
	\begin{align*}
	I^{\bar{p}}X=\operatorname{cone}(g)=\mathcal{M} \bigcup_{g} \mathcal{C}(\mathcal{L}_{<k}).
	\end{align*} 
\end{definition}
Due to the use of rational coefficients, $I^{\bar{p}}X$ as defined above need not be integrally homotopy equivalent to the construction of \cite{banagl2010intersection}, but the rational homology groups are isomorphic.

\begin{theorem}\label{Duality}
	Let $X$ be an $n$-dimensional compact oriented topological pseudomanifold with only isolated singularities. Let $\bar{p}$ and $\bar{q}$ be complementary perversities. Then, we have the duality isomorphism
	\begin{align*}
	d: \widetilde{H}_{r}(I^{\bar{p}}X)^{\ast} \xrightarrow{\phantom{-} \cong \phantom{-}} \widetilde{H}_{n-r}(I^{\bar{q}}X),
	\end{align*}
	where
	\begin{align*}
	\widetilde{H}_{r}(I^{\bar{p}}X)^{\ast} = \operatorname{Hom}(\widetilde{H}_{r}(I^{\bar{p}}X),\mathbb{Q}).
	\end{align*}	
\end{theorem}
\begin{proof}
	The detailed construction of the above duality isomorphism can be found in the proof of \cite[Theorem~2.12]{banagl2010intersection} by the first named author. However, as mentioned, $I^{\bar{p}}X$, as defined above need not be homotopy equivalent to the construction of \cite{banagl2010intersection}. Hence, we need to adjust the previous proof slightly.
	In our setting we have the inclusion $\mathcal{L}_{<k} \hookrightarrow \mathcal{L}=\partial \mathcal{M}$. We consider the same braid as in the proof in loc. cit. One can then proceed along the same line as in \cite{banagl2010intersection}.
\end{proof}
\begin{remark}\label{IsowithMfd}
	Let $k=n-1-\bar{p}(n)$. Then, we have the following isomorphisms.\\
	For $r>k$,
	\begin{align*}
	H_{r}(\mathcal{M}) &\xrightarrow{\phantom{-} \cong \phantom{-}} \widetilde{H}_{r}(I^{\bar{p}}X) \\
	\widetilde{H}_{n-r}(I^{\bar{q}}X) &\xrightarrow{\phantom{-} \cong \phantom{-}} H_{n-r}(\mathcal{M}, \partial \mathcal{M}).
	\end{align*}
	For $r<k$,
	\begin{align*}
	\widetilde{H}_{r}(I^{\bar{p} }X)&\xrightarrow{\phantom{-} \cong \phantom{-}} H_{r}(\mathcal{M},\partial \mathcal{M}) \\
	H_{n-r}(\mathcal{M}) &\xrightarrow{\phantom{-} \cong \phantom{-}} \widetilde{H}_{n-r}(I^{\bar{q}}X ).
	\end{align*}
\end{remark}

\begin{corollary}\label{EulerCh}
	If $n=\dim(X)$ is even, then the difference between the Euler characteristic of $\widetilde{H}_{\ast}(I^{\bar{p}}X)$ and $I^{\bar{p}}H_{\ast}(X)$ is given by
	\begin{align*}
	\chi(\widetilde{H}_{\ast}(I^{\bar{p}}X))-\chi(I^{\bar{p}}H_{\ast}(X))=-2 \chi_{<n-1-\bar{p}(n)}(\mathcal{L}),
	\end{align*}
	where $\mathcal{L}$ is the disjoint union of the links of the isolated singularities of $X$. \\
	If $n=\dim(X)$ is odd, then 
	\begin{align*}
	\chi(\widetilde{H}_{\ast}(I^{\bar{n}}X))-\chi(I^{\bar{n}}H_{\ast}(X))=(-1)^{\frac{n-1}{2}} b_{\frac{n-1}{2}}(\mathcal{L}),
	\end{align*}
	where $b_{\frac{n-1}{2}}(\mathcal{L})$ is the middle dimensional Betti number of $\mathcal{L}$ and $\bar{n}$ is the upper middle perversity. \\
	Regardless of the parity of $n$, the identity
	\begin{align*}
	\operatorname{rk}(\widetilde{H}_{k}(I^{\bar{p}}X))+ \operatorname{rk}(I^{\bar{p}}H_{k}(X))=\operatorname{rk}(H_{k}(\mathcal{M}))+\operatorname{rk}(H_{k}(\mathcal{M},\mathcal{L}))
	\end{align*}
	always holds in degree $k=n-1-\bar{p}(n)$, where $\mathcal{M}$ is the exterior of the singular set of $X$.
\end{corollary}
\begin{proof}
	A proof of the above corollary can be found in \cite{banagl2010intersection} by the first named author.
\end{proof}

\begin{lemma}
	The rational homology of  $I^{\bar{p}}X$ is independent of the choice of $(\mathcal{L}_{i})_{<k}$.
\end{lemma}
\begin{proof}
	From Remark \ref{IsowithMfd} it follows immediately that $H_{r}(I^{\bar{p}}X)$ does not depend of the choice of $(\mathcal{L}_{i})_{<k}$ for $r<k$ and $r>k$. For the case $r=k$ the claim follows from the last identity in Corollary \ref{EulerCh}.  
\end{proof}

\section{Construction of Intersection Spaces for \texorpdfstring{$\mathbb{Q}$}--Pseudomanifolds}\label{QPseudoGen}

 In this section, we generalize the theory of intersection spaces to $\mathbb{Q}$-pseudomanifolds with $\mathbb{Q}$-isolated singularities. To do this, we revisit the proof of Theorem \ref{Duality} introduced by the first named author in \cite{banagl2010intersection}. The central idea is to indicate that the main ingredients of that proof remain intact, even for $\mathbb{Q}$-pseudomanifold with isolated singularities.
Before proving our main theorem, we want to compute the intersection homology of $\mathbb{Q}$-pseudomanifolds with $\mathbb{Q}$-isolated singularities. But first of all, we need to generalize the previous definition of intersection spaces to the $\mathbb{Q}$-isolated setting. \\
Let $X$ be an $n$-dimensional compact oriented $\mathbb{Q}$-pseudomanifold with $\mathbb{Q}$-isolated singularities, which means $X$ has a $\mathbb{Q}$-stratification of the form
\begin{align*}
X=X_{n} \supset X_{0}.
\end{align*}
For $x_{i} \in X_{0}$, let $\mathcal{L}_{i}$ be an associated link. Assume that all links are rationally $k$-segmented, where $k=n-1-\bar{p}(n),$ for the perversity $\bar{p}$. Let $\mathcal{M}$ be the compact $\mathbb{Q}$-manifold with boundary obtained by cutting out all $\mathbb{Q}$-isolated singularities of $X$. Then, we have
$
\partial \mathcal{M}= \bigsqcup_{i} \mathcal{L}_{i}.
$
Let
$
\mathcal{L}_{<k}= \bigsqcup_{i} (\mathcal{L}_{i})_{<k},
$
and define a homotopy class 
$
g: \mathcal{L}_{<k} \xhookrightarrow{\phantom{-}\phantom{-}} \partial \mathcal{M} \xhookrightarrow{\phantom{-} \operatorname{incl.}  \phantom{-}} \mathcal{M}
$.

\begin{definition}\label{DefGenIX}
	The perversity $\bar{p}$ \textbf{generalized intersection space} $I^{\bar{p}}X$ of $X$ is defined to be
	\begin{align*}
	I^{\bar{p}}X=\operatorname{cone}(g)=\mathcal{M} \bigcup_{g} \mathcal{C}(\mathcal{L}_{<k}).
	\end{align*} 
\end{definition}
We shall use the standard cone computation for intersection homology: 
	\begin{proposition}\label{ConeIH}
	Suppose $Y$ is a compact topological pseudomanifold of dimension $m-1 \geq 0$. Then for a perversity $\bar{p}$,
	\begin{align*}
		I^{\bar{p}}H_{r}(\mathcal{C}(Y)) = \begin{cases}
			I^{\bar{p}}H_{r}(Y) ,& \phantom{-- } r<m-1-\bar{p}(m) \\
			0,              & \phantom{--} \text{otherwise.}
		\end{cases}
	\end{align*}
\end{proposition}

\begin{proposition}\label{IHDu}
	Let $X$ be an $n$-dimensional $\mathbb{Q}$-pseudomanifold with $\mathbb{Q}$-isolated singularities, where $n \geq  1$. Then for perversity $\bar{p}$, we have
	\begin{align*}
	I^{\bar{p}}H_{r}(X) \cong \begin{cases}
	H_{r}(\mathcal{M},\mathcal{L}) ,& \phantom{-- } r>k \\
	H_{r}(\mathcal{M}),              & \phantom{--} r<k,
	\end{cases}
	\end{align*}
	where $\mathcal{M}$ is again the $\mathbb{Q}$-manifold obtained by cutting out the $\mathbb{Q}$-isolated singularities and $\mathcal{L}=\bigsqcup \mathcal{L}_{x_{i}}$, where for each $x_{i} \in X_{0}$, $\mathcal{L}_{x_{i}}$ is a link associated to $x_{i}$, and $k=n-1-\bar{p}(n)$.
\end{proposition}
\begin{proof}
 We cut out cone-like neighborhoods of $\mathbb{Q}$-isolated singularities. The obtained topological space is a pseudomanifold with boundary, $(\mathcal{M},\partial\mathcal{M})$. In fact $(\mathcal{M},\partial\mathcal{M})$ is a $\mathbb{Q}$-manifold with boundary. Let $\mathcal{C}(\mathcal{L}_{x_{i}})$ be the removed open cone-like neighborhood of $x_{i}$. Let  $\mathcal{V}$ be an open neighborhood of $\mathcal{M}$ in $X$ that deformation retracts in a stratum preserving way to $\mathcal{M}$ and set $\mathcal{U}=\bigsqcup_{i} \mathcal{C}(\mathcal{L}_{x_{i}})$. Thus, we have $\cap := \mathcal{U} \cap \mathcal{V} \cong \bigsqcup_{i} \mathcal{L}_{x_{i}} \times (\frac{1}{2},1)$. The Mayer–Vietoris sequence for intersection homology is a long exact sequence
	\begin{align*}
	\dots \longrightarrow I^{\bar{p}}H_{r}(\cap) \longrightarrow I^{\bar{p}}H_{r}(\mathcal{U}) \oplus I^{\bar{p}}H_{r}(\mathcal{V}) \longrightarrow I^{\bar{p}}H_{r}(X) \longrightarrow \dots.
	\end{align*}
	For $r>k$, we have $I^{\bar{p}}H_{r}(\mathcal{U})=0$ by Proposition \ref{ConeIH}. Consider the long exact sequence for relative ordinary homology
	\begin{align*}
	\dots \longrightarrow H_{r}(\partial \mathcal{M}) \longrightarrow H_{r}(\mathcal{\mathcal{M}}) \longrightarrow H_{r}(\mathcal{M}, \partial \mathcal{M}) \longrightarrow \dots.
	\end{align*}  
    We construct a commutative diagram as follows: 
\begin{align}\label{COMMDIA}
	\xymatrix{
		\dots  \ar[r] &  I^{\bar{p}}H_{r}(\cap) \ar[d] \ar[r]  &  I^{\bar{p}}H_{r}(\mathcal{V}) \ar[d]  \ar[r] & I^{\bar{p}}H_{r}(X) \ar[r] \ar[d]  & I^{\bar{p}}H_{r-1}(\cap) \ar[d] \ar[r] &I^{\bar{p}}H_{r-1}(\mathcal{V}) \ar[d]  \ar[r]  &\dots \\
		\dots \ar[r] & H_{r}(\partial \mathcal{M}) \ar[r] & H_{r}(\mathcal{\mathcal{M}})  \ar[r] & H_{r}(\mathcal{M}, \partial \mathcal{M}) \ar[r] & H_{r-1}(\partial \mathcal{M}) \ar[r]  & H_{r-1}(\mathcal{\mathcal{M}})  \ar[r] & \dots
	} .
\end{align}

We look at the morphism $I^{\bar{p}}C_{\ast}(\partial \mathcal{M}) \hookrightarrow C_{\ast}(\partial \mathcal{M})$ of singular ordinary and intersection chain complexes. Since $\mathcal{M}$ is a $\mathbb{Q}$-manifold, this morphism induces an isomorphism $I^{\bar{p}}H_{\ast}(\partial \mathcal{M}) \xrightarrow{\cong} H_{\ast}(\partial \mathcal{M})$. The outer left vertical homomorphism, which turns out to be an isomorphism, is the composition $I^{\bar{p}}H_{\ast}(\cap) \xrightarrow{\cong} I^{\bar{p}}H_{\ast}(\partial \mathcal{M}) \xrightarrow{\cong} H_{\ast}(\partial \mathcal{M})$. The stratum preserving deformation retraction $\mathcal{V} \rightarrow \mathcal{M}$, induces an isomorphism $I^{\bar{p}}H_{\ast}(\mathcal{V}) \xrightarrow{\cong} I^{\bar{p}} H_{\ast}( \mathcal{M})$. The morphism $I^{\bar{p}}C_{\ast}(\mathcal{M}) \hookrightarrow C_{\ast}(\mathcal{M})$ induces an isomorphism $I^{\bar{p}}H_{\ast}( \mathcal{M}) \xrightarrow{\cong} H_{\ast}( \mathcal{M})$ since $\mathcal{M}$ is a $\mathbb{Q}$-manifold. The second left vertical homomorphism is the composition $I^{\bar{p}}H_{\ast}(\mathcal{V}) \xrightarrow{\cong}  I^{\bar{p}}H_{\ast}( \mathcal{M}) \xrightarrow{\cong} H_{\ast}( \mathcal{M})$. The construction of the two right hand vertical homomorphisms goes along the same lines.
Describing the middle homomorphism requires a bit more effort. First, note that $H_{\ast}(\mathcal{M}, \partial \mathcal{M}) \cong \tilde{H}_{\ast}(\sfrac{\mathcal{M}}{\partial \mathcal{M}})$. We show that $H_{j}(X) \cong H_{j}(\sfrac{\mathcal{M}}{\partial \mathcal{M}})$ for $j \geq 2$. We use the Mayer-Vietories sequence for ordinary homology groups. Let $A_{\sfrac{\mathcal{M}}{\partial \mathcal{M}}} \cong \mathcal{C}(\bigsqcup_{i} \mathcal{L}_{x_{i}}) \subset \sfrac{\mathcal{M}}{\partial \mathcal{M}}$ and $B_{\sfrac{\mathcal{M}}{\partial \mathcal{M}}} \cong \mathcal{M} \subset \sfrac{\mathcal{M}}{\partial \mathcal{M}}$ such that $A_{\sfrac{\mathcal{M}}{\partial \mathcal{M}}} \cap B_{\sfrac{\mathcal{M}}{\partial \mathcal{M}}} \cong \bigsqcup_{i} \mathcal{L}_{x_{i}} \times (0,1)$. We also set $A_{X} \cong \bigsqcup_{i} \mathcal{C}(\mathcal{L}_{x_{i}}) \subset X$ and $B_{X} \cong \mathcal{M} \subset X$ such that $A_{X} \cap B_{X} \cong \bigsqcup_{i} \mathcal{L}_{x_{i}} \times (0,1)$. For $j \geq 1$, we have $H_{j}(A_{X}) \oplus H_{j}(B_{X}) \cong H_{j}(A_{\sfrac{\mathcal{M}}{\partial \mathcal{M}}}) \oplus H_{j}(B_{\sfrac{\mathcal{M}}{\partial \mathcal{M}}}) \cong H_{j}(\mathcal{M})$. The claim follows by using 5-lemma. Note that $k \geq 1$. Let $I^{\bar{p}}H_{\ast}(X) \rightarrow H_{\ast}(X)$ be the homomorphism induced by the morphism $I^{\bar{p}} C_{\ast} (X) \hookrightarrow C_{\ast}(X)$. The middle vertical homomorphism is the composition 
\begin{align*}
	I^{\bar{p}}H_{\ast}(X) \rightarrow H_{\ast}(X) \xrightarrow{\cong}H_{\ast}(\sfrac{\mathcal{M}}{\partial \mathcal{M}}) \xrightarrow{\cong} H_{\ast}(\mathcal{M}, \partial \mathcal{M}), \; \ast \geq 2.
\end{align*}
We still need to check the commutativity of the above diagram.
Consider the commutative diagram
\begin{align*}
	\xymatrix{ \partial \mathcal{M} \times (\frac{1}{2},1)   \ar@{^{(}->}[r] \ar@{^{(}->}[d] & \mathcal{V} \\
		\partial \mathcal{M} \ar@{^{(}->}[r]	& \mathcal{M} \ar@{^{(}->}[u]
	}.
\end{align*} 
Each inclusion induces a morphism between intersection chain complexes, and the induced diagram on intersection chain complexes commutes. Hence, the induced diagram on intersection homology groups
\begin{align*}
	\xymatrix{ I^{\bar{p}}H_{\ast}( \partial \mathcal{M} \times (\frac{1}{2},1) )   \ar[r]  & I^{\bar{p}}H_{\ast}( \mathcal{V}) \\
		I^{\bar{p}}H_{\ast}( \partial \mathcal{M}) \ar[r] \ar[u]^{\cong}	& I^{\bar{p}}H_{\ast}( \mathcal{M} ) \ar[u]_{\cong}
	}.
\end{align*} 
commutes. Recall that vertical homomorphisms are isomorphisms. Furthermore, the diagram
\begin{align*}
	\xymatrix{ I^{\bar{p}}H_{\ast}( \partial \mathcal{M} )   \ar[r] \ar[d]_{\cong}  & I^{\bar{p}}H_{\ast}( \mathcal{M}) \ar[d]^{\cong} \\
		H_{\ast}( \partial \mathcal{M}) \ar[r] 	& H_{\ast}( \mathcal{M})
	}
\end{align*}
commutes. Note that the vertical morphisms induced by inclusions of singular intersection chain complexes into ordinary singular chain complexes are isomorphisms, as mentioned. Hence, the outer left square commutes. Using the same procedure, one can show that the outer right square commutes. Consider the commutative diagram
\begin{align*}
	\xymatrix{ \mathcal{V} \ar@{^{(}->}[r] & X \\
	            \mathcal{M}    \ar@{^{(}->}[u] \ar@{^{(}->}[ru]       &   }.
\end{align*}
Each inclusion induces a morphism between intersection chain complexes, and the induced diagram on intersection chain complexes commutes. Hence, the induced diagram on intersection homology groups
\begin{align*}
	\xymatrix{ I^{\bar{p}}H_{\ast} (\mathcal{V}) \ar[r] & I^{\bar{p}}H_{\ast}(X) \\
		I^{\bar{p}}H_{\ast}(\mathcal{M})    \ar[u] \ar[ru]       &   }.
\end{align*} 
commutes.
The commutativity of
 \begin{align*}
 	\xymatrix{ I^{\bar{p}}C_{\ast}( \mathcal{M} )   \ar[r] \ar@{^{(}->}[d]  & I^{\bar{p}}C_{\ast}( X) \ar@{^{(}->}[d] \\
 		C_{\ast}( \mathcal{M}) \ar[r] 	& C_{\ast}( X)
 	}
 \end{align*}
ensures that 
 \begin{align*}
	\xymatrix{ I^{\bar{p}}H_{\ast}( \mathcal{M} )   \ar[r] \ar[d]  & I^{\bar{p}}H_{\ast}( X) \ar[d] \\
		H_{\ast}( \mathcal{M}) \ar[r] 	& H_{\ast}( X)
	}
\end{align*}
commutes. Thus, the second left square commutes.
Consider the short exact sequence
\begin{align*}
	0 \longrightarrow I^{\bar{p}}C_{i}(\cap) \longrightarrow I^{\bar{p}}C_{i}(\mathcal{U}) \oplus I^{\bar{p}}C_{i}(\mathcal{V}) \longrightarrow I^{\bar{p}}C_{i}(\mathcal{U}) +I^{\bar{p}}C_{i}(\mathcal{V}) \longrightarrow 0.
\end{align*}
Note that $I^{\bar{p}}C_{i}(\mathcal{U}) +I^{\bar{p}}C_{i}(\mathcal{V})$ is the sub-complex of $I^{\bar{p}}C_{i}(X)$ generated by allowable chains supported in $\mathcal{U}$ or in $\mathcal{V}$. The commutativity of the diagram
\begin{align*}
	\xymatrix{
		0 \ar[r] & I^{\bar{p}}C_{i}(\cap) \ar[r] \ar@{^{(}->}[d] & I^{\bar{p}}C_{i}(\mathcal{U}) \oplus I^{\bar{p}}C_{i}(\mathcal{V}) \ar[r] \ar@{^{(}->}[d]  & I^{\bar{p}}C_{i}(\mathcal{U}) +I^{\bar{p}}C_{i}(\mathcal{V}) \ar[r] \ar@{^{(}->}[d]  & 0 \\
	 0 \ar[r] & C_{i}(\cap) \ar[r] & C_{i}(\mathcal{U}) \oplus C_{i}(\mathcal{V}) \ar[r] & C_{i}(\mathcal{U}) +C_{i}(\mathcal{V}) \ar[r] &0 }.
\end{align*}
implies that the diagram
\begin{align*}
	\xymatrix{
		\dots \ar[r] & I^{\bar{p}}H_{r}(\mathcal{V}) \ar[r] \ar[d] & I^{\bar{p}}H_{r}(X)  \ar[r] \ar[d]  & I^{\bar{p}}H_{r-1}(\cap) \ar[r] \ar[d]  & \dots \\
		\dots \ar[r] & H_{r}(\mathcal{V}) \ar[r] & H_{r}(X) \ar[r] & H_{r-1}(\cap) \ar[r] &\dots }
\end{align*}
commutes. The commutativity of the right square in the above diagram ensures that the second right square in Diagram \ref{COMMDIA} commutes. The desired result follows from 5-lemma. Thus, we have
	\begin{align*}
	I^{\bar{p}}H_{r}(X)  \cong H_{r}(\mathcal{M}, \partial \mathcal{M}),\text{for} \phantom{..} r>k.   \end{align*} 

For $r<k$, the proof goes as follows. Let $\mathcal{U}$ and $\mathcal{V}$ be as above. Employing Proposition \ref{ConeIH} yields $I^{\bar{p}}H_{r}(\mathcal{U}) \cong I^{\bar{p}}H_{r}(\mathcal{L})$. Recall that $\mathcal{L}$ is a $\mathbb{Q}$-manifold. Hence, we have the following diagram.
\begin{small}
\begin{align*}
		\xymatrix@C=10pt{
		\dots  \ar[r] &  I^{\bar{p}}H_{r}(\cap) \ar[d] \ar[r]  &  I^{\bar{p}}H_{r}(\mathcal{M}) \oplus I^{\bar{p}}H_{r}(\mathcal{L}) \ar[d]  \ar[r] & I^{\bar{p}}H_{r}(X) \ar[r] \ar[d]  & I^{\bar{p}}H_{r-1}(\cap) \ar[d] \ar[r] &I^{\bar{p}}H_{r-1}(\mathcal{M}) \oplus I^{\bar{p}}H_{r-1}(\mathcal{L}) \ar[d]  \ar[r]  &\dots \\
		\dots \ar[r] & H_{r}(\cap) \ar[r] & H_{r}(\mathcal{\mathcal{M}}) \oplus H_{r}(\mathcal{\mathcal{L}})  \ar[r] & H_{r}(\mathcal{M}) \ar[r] & H_{r-1}(\cap) \ar[r]  & H_{r-1}(\mathcal{\mathcal{M}}) \oplus H_{r-1}(\mathcal{\mathcal{L}}) \ar[r] & \dots
	} .
\end{align*}
\end{small}

 The commutativity of the left hand and right hand squares follows from the commutativity of
  \begin{align*}
 	\xymatrix{ I^{\bar{p}}C_{\ast}( \cap )   \ar[r] \ar@{^{(}->}[d]  & I^{\bar{p}}C_{\ast}( \mathcal{L}) \oplus I^{\bar{p}}C_{\ast}(\mathcal{M}) \ar@{^{(}->}[d] \\
 		C_{\ast}( \cap) \ar[r] 	& C_{\ast}( \mathcal{L}) \oplus C_{\ast}( \mathcal{M})
 	}.
 \end{align*}
 
 The vertical monomorphisms induce isomorphisms between the intersection and ordinary homology groups. By \cite[Lemma 2.46]{banagl2010intersection}, we can construct the middle vertical homomorphism such that the diagram commutes. Finally, by 5-lemma, we have
	\begin{align*}
	I^{\bar{p}}H_{r}(X)  \cong H_{r}(\mathcal{M}),\text{for} \phantom{..} r<k.   \end{align*} 
	This concludes the proof.
\end{proof}
We are now at the point where we can start with the proof of the main theorem of this section. In \cite{banagl2010intersection}, the first named author shows more than just the duality of Betti numbers of intersection spaces. Although we are only interested in the duality of the Betti numbers of intersection spaces, we can generalize \cite[Theorem~2.12]{banagl2010intersection} to $\mathbb{Q}$-pseudomanifolds. For a brief introduction to reflective algebra, the reader may consult \cite[Section 2.1]{banagl2010intersection}.
\begin{remark}
	As mentioned earlier, we only consider homology with rational coefficients. Note also that because the underlying space, $X$, is a $\mathbb{Q}$-pseudomanifold with $\mathbb{Q}$-isolated singularities, $I^{\bar{p}}X$ refers to the generalized intersection space of $X$ in the sense of Definition \ref{DefGenIX}.
\end{remark}
\begin{theorem}\label{Duality1}
	Let $X$ be an $n$-dimensional compact oriented $\mathbb{Q}$-pseudomanifold with only $\mathbb{Q}$-isolated singularities. Let $\bar{p}$ and $\bar{q}$ be complementary perversities. Assume that the link of each $\mathbb{Q}$-isolated singularity is rationally $(n-1-\bar{p}(n))$-segmented. \\
	Then
	\begin{enumerate}
		\item The pair $(\widetilde{H}_{\ast}(I^{\bar{p}}X),I^{\bar{p}}H_{\ast}(X))$ is $(n-1-\bar{p}(n))$-reflective across the homology of the links and
		\item $(\widetilde{H}_{\ast}(I^{\bar{p}}X),I^{\bar{p}}H_{\ast}(X))$ and $(\widetilde{H}_{\ast}(I^{\bar{q}}X),I^{\bar{q}}H_{\ast}(X))$ are $n$-dual reflective pairs.
	\end{enumerate}
\end{theorem}

\begin{proof}
	We mainly mimic the proof of \cite[Theorem~2.12]{banagl2010intersection}. By the above considerations, the main ingredients of the proof remain intact for $\mathbb{Q}$-Pseudomanifolds with $\mathbb{Q}$-isolated singularities.\\
	We start with the homology braid of the triple
	\begin{align*}
	\mathcal{L}_{<k} \xhookrightarrow{ \phantom{-f} \phantom{-}  } \mathcal{L} \xrightarrow{ \phantom{-}j \phantom{-}  } \mathcal{M}.
	\end{align*}
	By Proposition \ref{IHDu}, the same argument as in the proof of \cite[Theorem~2.12]{banagl2010intersection} can be applied. The rest of the proof relies on Lefschetz duality for $(\mathcal{M}, \partial \mathcal{M})$ and Poincaré duality for $\partial \mathcal{M}$. For Lefschetz duality, we employ Theorem \ref{LFduality}, The topological space $\partial \mathcal{M}$ is a $\mathbb{Q}$-manifold by Proposition \ref{prop.bndryofrathomolmfdisrathomolmfd}, and therefore it satisfies Poincaré duality rationally, using Corollary \ref{PDRationally}. This concludes our proof.
\end{proof}
\begin{corollary}\label{Duality2}
	Let $X$ be an $n$-dimensional, compact, oriented, topological $\mathbb{Q}$-pseudomanifold with only $\mathbb{Q}$-isolated singularities, which means there is a $\mathbb{Q}$-stratification of the form $X_{n} \supset X_{0}.$ Let $\bar{p}$ and $\bar{q}$ be complementary perversities. Then, we have the duality isomorphism
	\begin{align*}
	d: \widetilde{H}_{r}(I^{\bar{p}}X)^{\ast} \xrightarrow{\phantom{-} \cong \phantom{-}} \widetilde{H}_{n-r}(I^{\bar{q}}X),
	\end{align*}
	where
	\begin{align*}
	\widetilde{H}_{r}(I^{\bar{p}}X)^{\ast} = \operatorname{Hom}(\widetilde{H}_{r}(I^{\bar{p}}X),\mathbb{Q}).
	\end{align*}	
\end{corollary}
\begin{proof}
The	construction of the duality isomorphism goes along the same line as in the proof of \cite[Theorem~2.12]{banagl2010intersection} and using Corollary \ref{PDRationally} and Theorem \ref{LFduality}.
\end{proof}
Corollary \ref{EulerCh} continues to hold in the $\mathbb{Q}$-isolated setting for the generalized intersection spaces: 
\begin{corollary}\label{GenCor}
	Let $X$ be an $n$-dimensional, compact, oriented $\mathbb{Q}$-pseudomanifold with only $\mathbb{Q}$-isolated singularities. Let $I^{\bar{p}}X$ be the associated generalized intersection spaces. If $n=\dim(X)$ is even, then the difference between the Euler characteristic of $\widetilde{H}_{\ast}(I^{\bar{p}}X)$ and $I^{\bar{p}}H_{\ast}(X)$ is given by
	\begin{align*}
	\chi(\widetilde{H}_{\ast}(I^{\bar{p}}X))-\chi(I^{\bar{p}}H_{\ast}(X))=-2 \chi_{<n-1-\bar{p}(n)}(\mathcal{L}),
	\end{align*}
	where $\mathcal{L}$ is the disjoint union of the links of the $\mathbb{Q}$-isolated singularities of $X$. \\
	If $n=\dim(X)$ is odd, then 
	\begin{align*}
	\chi(\widetilde{H}_{\ast}(I^{\bar{n}}X))-\chi(I^{\bar{n}}H_{\ast}(X))=(-1)^{\frac{n-1}{2}} b_{\frac{n-1}{2}}(\mathcal{L}),
	\end{align*}
	where $b_{\frac{n-1}{2}}(\mathcal{L})$ is the middle dimensional Betti number of $\mathcal{L}$ and $\bar{n}$ is the upper middle perversity. \\
	Regardless of the parity of $n$, the identity
	\begin{align*}
	\operatorname{rk}(\widetilde{H}_{k}(I^{\bar{p}}X))+ \operatorname{rk}(I^{\bar{p}}H_{k}(X))=\operatorname{rk}(H_{k}(\mathcal{M}))+\operatorname{rk}(H_{k}(\mathcal{M},\mathcal{L}))
	\end{align*}
	always holds in degree $k=n-1-\bar{p}(n)$, where $\mathcal{M}$ is the $\mathbb{Q}$-manifold obtained by cutting out the $\mathbb{Q}$-isolated singularities of $X$.
\end{corollary}

\section{Intersection Spaces of 6-dimensional Toric Varieties}

As mentioned in Section \ref{Sing6D}, a 6-dimensional toric variety is a $\mathbb{Q}$-pseudomanifold with $\mathbb{Q}$-isolated singularities. Such varieties are therefore amenable to the theory developed in Section \ref{QPseudoGen}. Although we will not compute the homology groups of the intersection spaces directly via CW structures, we will show that considering the middle perversity, links of $\mathbb{Q}$-isolated singularities are rationally 3-segmented. Using Lefschetz duality, we then compute the rational homology groups of generalized intersection spaces of 6-dimensional toric varieties. As a side remark, we will also show that Corollary \ref{GenCor} yields the well-known Euler formula for 3-dimensional polytypes.\\
Let $X_{\mathcal{P}}$ be a 6-dimensional toric variety associated with $\mathcal{P}$, the underlying polytope. We can endow $X_{\mathcal{P}}$ with a $\mathbb{Q}$-stratification of the form $
X_{\mathcal{P}}=X_{6} \supset X_{0}.$
We have shown in Section \ref{Sing6D} that links of $\mathbb{Q}$-isolated singularities are $\mathbb{Q}$-manifolds and satisfy Poincaré duality rationally.
\begin{proposition}\label{Ltrunc}
	Let $X$ be a 6-dimensional toric variety with a $\mathbb{Q}$-stratification of the form $X_{6} \supset X_{0}$.	For each $x_{i} \in X_{0}$, let $\mathcal{L}_{i}$ be the associated link. Then $\mathcal{L}_{i}$ is rationally 3-segmented.
\end{proposition}
\begin{proof}
	We endow each $\mathcal{L}_{i}$ with the CW structure described in Section \ref{Sing6D}. Recall that we have computed $\partial_{3}$ of the $\mathcal{L}_{i}$ in Equation \ref{Link6D}. From each $T^2$ with one 2-cell, we remove this 2-cell. From each $T^2$ with two 2-cells, we omit the 2-cell with the negative sign in the boundary operator. We denote the obtained 3-dimensional CW sub-complex by $\mathcal{L}_{i}^{\prime}$. Despite the removal of the 2-cells from the tori, the 2-skeleton of the link has not changed, i.e $\mathcal{L}_{i}$ and $\mathcal{L}_{i}^{\prime}$ have the same 2-skeleton.
	Hence, we get 
	\begin{align}
	\partial_{3}^{\mathcal{L}_{i}^{\prime}} &= \kbordermatrix{
		& 
		e^{1}_{T^{3}_{1}} \times e^{2}_{\mathcal{M}_{x}} & 		e^{1}_{T^{3}_{2}} \times e^{2}_{\mathcal{M}_{x}} & 		e^{1}_{T^{3}_{3}} \times e^{2}_{\mathcal{M}_{x}} & e^{2}_{(T^{2}_{a_{i}})_{1} } \times e^{1}_{(\mathcal{M}_{x_{a } })_{i} }   \\
		e^{0}_{T^{3}} \times e^{2}_{\mathcal{M}_{x}} &   0  &  0  &  0 &  \overbrace{0 }^{ \tilde{a}  }  \\
		e^{1}_{(T^{2}_{a_{i}})_{1} } \times e^{1}_{(\mathcal{M}_{x_{a } })_{i} } \tilde{a} \Big\{ &m_{a_{i}} l_{a_{i}} & 0 & 0 & 1   \\	
		e^{1}_{(T^{2}_{a_{i}})_{2} } \times e^{1}_{(\mathcal{M}_{x_{a } })_{i} } \tilde{a}  \Big\{ &0 & n_{a_{i}} l_{a_{i}} & 0 & 1  \\
		e^{1}_{(T^{2}_{a_{i}})_{3} } \times e^{1}_{(\mathcal{M}_{x_{a } })_{i} } \tilde{a}  \Big\{ &0 & 0 & n_{a_{i}} m_{a_{i}} & 1  \\
		e^{1}_{(T^{2}_{b_{j}})_{1} } \times e^{1}_{(\mathcal{M}_{x_{b } })_{j} } \tilde{b} \Big\{ &-m_{b_{j}} & 0 & 0 & 0  \\
		e^{1}_{(T^{2}_{b_{j}})_{2} } \times e^{1}_{(\mathcal{M}_{x_{b } })_{j} } \tilde{b} \Big\{ &0 & n_{b_{j}} & 0 & 0   \\
	}.
	\end{align}
	We use the same notation as in Section \ref{Sing6D}. Recall that each 1-dimensional face of $\mathcal{M}_{x}$, is associated to a 2-dimensional face of $\mathcal{P}$, which is dual to a 1-dimensional cone in $\Sigma$ such that the dual cones lie in $\mathcal{S}_{x}$. Here $\tilde{a}$ and $\tilde{b}$ are the numbers of such 1-dimensional cones in $\Sigma$, whose generators have no zero entry and at least one zero-entry, respectively. Recall also that  $e^{1}_{(\mathcal{M}_{x_{a}})_{i}  }$ and $e^{1}_{(\mathcal{M}_{x_{b}})_{j}  }$ are those 1-cells of $\mathcal{M}_{x}$, whose associated $T^{2}$ in $\mathcal{L}_{x}$ has three and two 1-cells, respectively.
	Clearly, we have
	\begin{align*}
	\operatorname{Im}(\partial_{3}^{\mathcal{L}_{i}^{\prime}})=\operatorname{Im}(\partial_{3}^{\mathcal{L}_{i}}) \:\: \text{and} \:\: \operatorname{ker}(\partial_{3}^{\mathcal{L}_{i}^{\prime}})=0,
	\end{align*}
	and the inclusion $incl:\mathcal{L}_{i}^{\prime} \xhookrightarrow{\phantom{---}} \mathcal{L}_{i} $ induces
	\begin{align*}
	incl_{j}: H_{j}(\mathcal{L}_{i}^{\prime}) \xrightarrow{\phantom{-}\cong \phantom{-}} H_{j}(\mathcal{L}_{i}) \:\: \text{for} \:\: j \leq 2.
	\end{align*}
	Obviously,
	\begin{align*}
	H_{j}(\mathcal{L}_{i}^{\prime})=0 \:\:\text{for} \:\: j > 2.
	\end{align*}
\end{proof}

\begin{theorem}\label{IntSecBetti}
	Let $X_{\mathcal{P}}$ be an arbitrary compact 6-dimensional toric variety with $m$ $\mathbb{Q}$-isolated singularities, $m \geq 1$, where $\mathcal{P}$ is the underlying polytope. Let $\Sigma$ be the dual fan to $\mathcal{P}$. We denote the number of 1-dimensional and 2-dimensional cones of $\Sigma$ with $f_{1}$ and $f_{2}$, respectively. Then,
	\begin{align*}
	\operatorname{rk}(\widetilde{H}_{6}(I^{\bar{n}}X))&=0 \\
	\operatorname{rk}(\tilde{H}_{5}(I^{\bar{n}}X))&=m-1 \\
	\operatorname{rk}(\tilde{H}_{4}(I^{\bar{n}}X))&= f_{1}-3-b\\
	\operatorname{rk}(\tilde{H}_{3}(I^{\bar{n}}X))&=2(3f_{1}-f_{2}-b-6) \\
	\operatorname{rk}(\tilde{H}_{2}(I^{\bar{n}}X))&=f_{1}-3-b \\
	\operatorname{rk}(\tilde{H}_{1}(I^{\bar{n}}X))&=m-1\\
	\operatorname{rk}(\tilde{H}_{0}(I^{\bar{n}}X))&=0,
	\end{align*}
	where $I^{\bar{n}}X$ is the generalized intersection space associated to the $\mathbb{Q}$-pseudomanifold $X_{\mathcal{P}}$, and the parameter $b$ is introduced and has been studied in Section \ref{6DHoGr}.
\end{theorem}
\begin{proof}
	We start by cutting out the $\mathbb{Q}$-isolated singularities of $X_{\mathcal{P}}$, in the sense of Definition \ref{Cout}. We denote the resulting $\mathbb{Q}$-manifold with boundary by ($\mathcal{M},\partial \mathcal{M}$). It satisfies Lefschetz duality rationally by Corollary \ref{LefDualityCor}. Now, let $\sfrac{\mathcal{M}}{\partial \mathcal{M}}$ be the topological pseudomanifold obtained by coning off the boundary of $(\mathcal{M}, \partial \mathcal{M})$. As shown, we have
	\begin{align*}
	H_{i}(X) \cong H_{i}(\sfrac{\mathcal{M}}{\partial \mathcal{M}}) \cong H_{i}(\mathcal{M}, \partial \mathcal{M}) \:\: \text{for} \: i \geq 2.
	\end{align*}
	Hence by Proposition \ref{OrdiHom}, we get 
	\begin{align*}
	\operatorname{rk}(H_{6}(\mathcal{M}, \partial \mathcal{M}))=\operatorname{rk}(H_{0}(\mathcal{M}))&=1\\
	\operatorname{rk}(H_{5}(\mathcal{M}, \partial \mathcal{M}))=\operatorname{rk}(H_{1}(\mathcal{M}))&=0\\
	\operatorname{rk}(H_{4}(\mathcal{M}, \partial \mathcal{M}))=\operatorname{rk}(H_{2}(\mathcal{M}))&=f_{1}-3\\
	\operatorname{rk}(H_{3}(\mathcal{M}, \partial \mathcal{M}))=\operatorname{rk}(H_{3}(\mathcal{M}))&=3f_{1}-f_{2}-b-6 \\
	\operatorname{rk}(H_{2}(\mathcal{M}, \partial \mathcal{M}))=\operatorname{rk}(H_{4}(\mathcal{M}))&=f_{1}-3-b.
	\end{align*}
The group $H_{6}(\mathcal{M}))$ vanishes since by Lefschetz duality $\operatorname{rk}(H_{6}(\mathcal{M}))= \operatorname{rk}(H_{0}(\mathcal{M},\partial \mathcal{M}))=0$ (Note that $\partial \mathcal{M}$ is not empty as $m \geq 1$.) Consequently, by considering the long exact sequence of the pair $(\mathcal{M}, \partial \mathcal{M})$, we arrive at the short exact sequence 
	\begin{align*}
	0 \longrightarrow H_{6}(\mathcal{M},\partial \mathcal{M}) \longrightarrow H_{5}(\partial \mathcal{M}) \longrightarrow H_{5}(\mathcal{M}) \longrightarrow 0.
	\end{align*}
	Since $\partial \mathcal{M} = \bigsqcup_{i=1}^{m} \mathcal{L}_{i}$ and each $\mathcal{L}_{i}$ is orientable, connected and 5-dimensional, $\operatorname{rk}(H_{5}(\partial \mathcal{M})) = m$. This implies
	\begin{align*}
	\operatorname{rk}(H_{1}(\mathcal{M}, \partial \mathcal{M}))=\operatorname{rk}(H_{5}(\mathcal{M}))&= m -1.
	\end{align*}
	Before dealing with the homology groups of the generalized intersection space, we need to compute the intersection homology groups of $X_{\mathcal{P}}$. In \cite{stanley1987generalized}, Stanley shows that the odd degree intersection homology groups of toric varieties vanish. We merely need this result for the middle degree. The relevant cut off degree is $k=6-1-\bar{n}(6)=3$. Using Proposition \ref{IHDu}, we finally get
	\begin{align*}
	\operatorname{rk}(I^{\bar{n}} H_{6}(X))&=1 \\
	\operatorname{rk}(I^{\bar{n}} H_{5}(X))&=0 \\
	\operatorname{rk}(I^{\bar{n}} H_{4}(X))&=f_{1}-3\\
	\operatorname{rk}(I^{\bar{n}} H_{3}(X))&=0 \\
	\operatorname{rk}(I^{\bar{n}} H_{2}(X))&=f_{1}-3\\
	\operatorname{rk}(I^{\bar{n}} H_{1}(X))&=0\\
	\operatorname{rk}(I^{\bar{n}} H_{0}(X))&=1.
	\end{align*}
	Proposition \ref{Ltrunc} ensures the existence of the associated generalized intersection space of $X_{\mathcal{P}}$.
	By using the duality isomorphisms of Theorem \ref{Duality1}, we can read off the homology groups of $I^{\bar{n}}X$ above and below the middle degree:
	\begin{align*}
	&H_{r}(\mathcal{M}) \xrightarrow{\phantom{-- } \cong \phantom{-- }} \widetilde{H}_{r}(I^{\bar{n}}X) \:\: \text{for} \:\: r>k=3 \\
	&H_{r}(\mathcal{M},\partial \mathcal{M}) \xrightarrow{\phantom{-- }\cong \phantom{-- }} \widetilde{H}_{r}(I^{\bar{n}}X) \:\: \text{for} \:\: r<k=3. \\
	\end{align*} 
	Finally, we use Proposition \ref{Link3D} and the last identity in Corollary \ref{GenCor} and compute $\operatorname{rk}(\widetilde{H}_{3}(I^{\bar{n}}X))$.
\end{proof}

\begin{corollary}\label{AllSing}
		Let $X_{\mathcal{P}}$ be an arbitrary compact 6-dimensional toric variety, where $\mathcal{P}$ is the underlying polytope. Let $\Sigma$ be the fan dual to $\mathcal{P}$. We denote the 1-dimensional, 2-dimensional, and 3-dimensional cones of $\Sigma$ with $f_{1}$, $f_{2}$, and $f_{3}$, respectively. We further assume that all vertices in $\mathcal{P}$ are $\mathbb{Q}$-isolated singularities. Then,
		\begin{align*}
			\operatorname{rk}(\widetilde{H}_{6}(I^{\bar{n}}X))&=0 \\
			\operatorname{rk}(\tilde{H}_{5}(I^{\bar{n}}X))&=f_{3}-1 \\
			\operatorname{rk}(\tilde{H}_{4}(I^{\bar{n}}X))&= f_{1}-3-b\\
			\operatorname{rk}(\tilde{H}_{3}(I^{\bar{n}}X))&=2(3f_{1}-f_{2}-b-6) \\
			\operatorname{rk}(\tilde{H}_{2}(I^{\bar{n}}X))&=f_{1}-3-b \\
			\operatorname{rk}(\tilde{H}_{1}(I^{\bar{n}}X))&=f_{3}-1\\
			\operatorname{rk}(\tilde{H}_{0}(I^{\bar{n}}X))&=0.
		\end{align*}
\end{corollary}
\begin{proof}
	The number of vertices in $\mathcal{P}$ is equal to $f_{3}$, which equals $m$ by the assumption.
\end{proof}

\begin{remark}
    For $X_{\mathcal{P}}$ an arbitrary compact 6-dimensional toric variety with $m$ $\mathbb{Q}$-isolated singularities, $m \geq 1$, we employ the first identity on the Euler characteristic in Corollary \ref{GenCor} to obtain
	\begin{align*}
	&\big( 0-(m -1)+(f_{1}-3-b)-2(3f_{1}-f_{2}-b-6)+(f_{1}-3-b)\\
	&-(m-1)+0 \big)-\big(1+(f_{1}-3)+(f_{1}-3)+1 \big)\\
	&=-2\big(\sum_{i=1}^{m}1-0+\sum_{i=1}^{m}(f_{1}^{i}-3)\big),
	\end{align*}
	where $f_{1}^{i}$ denotes the number of 1-dimensional neighboring faces of the $i$-th vertex of $\mathcal{P}$.
	We can also interpret $f_{1}^{i}$ as the number of 1-dimensional faces of $\mathcal{M}_{i}$, where $\mathcal{M}_{i}$ is the underlying 2-dimensional polygon of the link of the $i$-th vertex. Note that for a $\mathbb{Q}$-isolated non-singular point $x_{j}$, we have $f_{1}^{j}=3$. 
	This implies $\sum_{i=1}^{m}(f_{1}^{i}-3)=	\sum_{i=1}^{f_{3}}(f_{1}^{i}-3)$. Counting the 1-dimensional faces of $\mathcal{M}_{i}$ for each vertex $x_{i}$ equals counting the neighboring 1-dimensional faces of each vertex in the underlying polytope $\mathcal{P}$. Since each 1-dimensional face of the polytype has precisely two vertices in its topological boundary, we have $\sum_{i=1}^{f_{3}}f_{1}^{i}=2 f_{2}$. Thus, we get
\begin{align*}
	\sum_{i=1}^{f_{3}}(f_{1}^{i}-3)=2f_{2}-3f_{3}.
\end{align*}	
	Then, we have
	$
	\big(-2 m -4 f_{1}+2 f_{2}+8 \big) 
	-\big(2 f_{1}-4 \big) 
	=-2 \big(m+ 2 f_{2}-3 f_{3} \big).
	$
	Finally, we get
	\begin{align*}
	f_{3}-f_{2} +f_{1}=2,
	\end{align*}
	which is the well-known Euler formula for 3-dimensional polytopes. 
\end{remark}
\begin{corollary}\label{AllSing2}
	We assume that all vertices in $\mathcal{P}$ are $\mathbb{Q}$-isolated singularities. Considering the above remark, we arrive at
	\begin{align*}
	\operatorname{rk}(\widetilde{H}_{6}(I^{\bar{n}}X))&=0 \\
	\operatorname{rk}(\widetilde{H}_{5}(I^{\bar{n}}X))&=f_{2}-f_{1}+1 \\
	\operatorname{rk}(\widetilde{H}_{4}(I^{\bar{n}}X))&= f_{1}-3-b\\
	\operatorname{rk}(\widetilde{H}_{3}(I^{\bar{n}}X))&=2(3f_{1}-f_{2}-b-6) \\
	\operatorname{rk}(\widetilde{H}_{2}(I^{\bar{n}}X))&=f_{1}-3-b \\
	\operatorname{rk}(\widetilde{H}_{1}(I^{\bar{n}}X))&=f_{2}-f_{1}+1\\
	\operatorname{rk}(\widetilde{H}_{0}(I^{\bar{n}}X))&=0.
	\end{align*}
\end{corollary}
\begin{remark}
	In Theorem \ref{IntSecBetti}, we have
	\begin{align*}
	\chi(\tilde{H}_{\ast}(I^{\bar{n}}X))=2(-2f_{1}+f_{2}-m+4),
	\end{align*}
	which is even.
	This observation is consistent with the following more general principle. 
	Let $X$ be a $(4n+2)$-dimensional $\mathbb{Q}$-pseudomanifold with $\mathbb{Q}$-isolated singularities. 
	 By using the duality on the intersection homology groups, we find
	 \begin{align*}
	 	\chi(I^{\bar{n}}H_{\ast}(X)) = \sum_{i=1}^{4n+2}(-1)^{i} \operatorname{rk}(I^{\bar{n}}H_{i}(X))=2 \sum_{i=1}^{2n}(-1)^{i} \operatorname{rk}(I^{\bar{n}}H_{i}(X))-\operatorname{rk}(I^{\bar{n}}H_{2n+1}(X)).
	 \end{align*}
	  The intersection form on $I^{\bar{n}}H_{2n+1}(X)$ is skew-symmetric and non-degenerate. Hence, there is a basis with respect to which the intersection form has the matrix representation
	  \begin{align*}
	\left( \begin{matrix}
		\begin{matrix}
			0 & 1 \\
		   -1 & 0 \\
		\end{matrix} & 0 & \cdots &  0\\
	0 & \begin{matrix}
		0 & 1 \\
		-1 & 0 \\
	\end{matrix}  &  & \vdots \\
       \vdots &  & \ddots & 0 \\
       0 & \cdots & 0 & \begin{matrix}
       	0 & 1 \\
       	-1 & 0 \\
       \end{matrix} 
	\end{matrix} \right).  
	  \end{align*}
	  Consequently, we have $\operatorname{rk}(I^{\bar{n}}H_{2n+1}(X)) \equiv 0(2)$
	  and thus $\chi(I^{\bar{n}}H_{\ast}(X)) \equiv 0 (2).$
      For an intersection space $I^{\bar{n}}X$ of $X$, it follows from Lefschetz duality 
      that 
      $\operatorname{rk}(H_{2n+1}(I^{\bar{n}}X))=- \operatorname{rk}(I^{\bar{n}}H_{2n+1}(X))+2 \operatorname{rk}(H_{2n+1}(\mathcal{M}))$. Therefore,
      \begin{align*}
      	\chi(H_{\ast}(I^{\bar{n}}X)) \equiv 0 (2).
      \end{align*}
\end{remark}

\begin{remark}
	In the context of Corollary \ref{AllSing2}, the homology groups of $I^{\bar{n}}X$ determine the ordinary homology groups of $X_{\mathcal{P}}$ up to and including the middle degree. In contrast, the intersection homology groups do not determine the ordinary homology below the middle degree. In this sense, 
	the homology of intersection spaces of toric varieties sees more structure than the intersection homology.
\end{remark}
\begin{remark}\label{MVIXB}
	We can nicely fit all of the above data into a Mayer-Vietoris sequence. 
	Let $I^{\bar{n}}X$ be the associated generalized intersection space of a 
	6-dimensional toric variety $X$. Then we take $\mathcal{U} \simeq \mathcal{M}$, 
	where $\mathcal{M}$ is the $\mathbb{Q}$-manifold obtained by cutting out the 
	$\mathbb{Q}$-isolated singularities of $X$. We also take $\mathcal{V} \cong \mathcal{C}(\displaystyle{\sqcup} \mathcal{L}_{<3})$. The intersection $\cap := \mathcal{U} \cap \mathcal{V}$
	is homotopy equivalent to $\mathcal{L}_{<3}$. Thus, we have
	\begin{footnotesize}
		\begin{align*}
		0 &\xrightarrow{\phantom{-\cong-}} \overbrace{H_{6}(\cap)}^{=0} \xrightarrow{\phantom{-\cong-}} \overbrace{H_{6}(\mathcal{U})}^{=0} \xrightarrow{\phantom{-} \cong \phantom{-}} \overbrace{H_{6}(I^{\bar{n}}X)}^{=0} 
		\xrightarrow{\phantom{-\cong-}} \overbrace{H_{5}(\cap)}^{=0} \xrightarrow{\phantom{-\cong-}} \overbrace{H_{5}(\mathcal{U})}^{\cong \mathbb{Q}^{m-1}} \xrightarrow{\phantom{-} \cong \phantom{-}} \overbrace{H_{5}(I^{\bar{n}}X)}^{\cong \mathbb{Q}^{m -1}} \\
		&\xrightarrow{\phantom{-\cong-}} \overbrace{H_{4}(\cap)}^{=0} \xrightarrow{\phantom{-\cong-}} \overbrace{H_{4}(\mathcal{U})}^{\cong \mathbb{Q}^{f_{1}-3-b} } \xrightarrow{\phantom{-} \cong \phantom{-}} \overbrace{ H_{4}(I^{\bar{n}}X)}^{\cong \mathbb{Q}^{f_{1}-3-b} } \xrightarrow{\phantom{-\cong-}} \overbrace{H_{3}(\cap)}^{=0} \xrightarrow{\phantom{-\cong-}} \overbrace{H_{3}(\mathcal{U})}^{\cong \mathbb{Q}^{3f_{1}-f_{2}-b-6} } \xrightarrow{\phantom{-\cong-} } \overbrace{H_{3}(I^{\bar{n}}X)}^{\cong \mathbb{Q}^{2(3f_{1}-f_{2}-b-6)}  } \\
		&\xrightarrow{\phantom{-\cong-} } \overbrace{H_{2}(\cap)}^{\cong \mathbb{Q}^{ r} } \xrightarrow{\phantom{-\cong-} } \overbrace{ H_{2}(\mathcal{U})}^{\cong \mathbb{Q}^{f_{1}-3 }} \xrightarrow{\phantom{-\cong-} } \overbrace{H_{2}(I^{\bar{n}}X)}^{\cong \mathbb{Q}^{f_{1}-3-b }} \xrightarrow{\phantom{-\cong-} } \overbrace{H_{1}(\cap)}^{=0} \xrightarrow{\phantom{-\cong-} } \overbrace{H_{1}(\mathcal{U})}^{=0} \xrightarrow{\phantom{-\cong-} } \overbrace{H_{1}(I^{\bar{n}}X)}^{\cong \mathbb{Q}^{m -1} } \\
		& \xrightarrow{\phantom{-} \cong \phantom{-} } \overbrace{\widetilde{H}_{0}(\cap)}^{\cong \mathbb{Q}^{m -1}} \xrightarrow{\phantom{-\cong-} } \overbrace{\widetilde{H}_{0}(\mathcal{U})}^{=0} \oplus \overbrace{\widetilde{H}_{0}(\mathcal{V})}^{=0} \xrightarrow{\phantom{-\cong-} } \overbrace{ \widetilde{H}_{0}(I^{\bar{n}}X) }^{=0} \xrightarrow{\phantom{-\cong-} } 0,
		\end{align*}
	\end{footnotesize}
	where $r := \sum_{i=1}^{m}(f_{1}^{i}-3)=2 f_{2}-3f_{3}$.
\end{remark}
\begin{corollary}
	For each 6-dimensional toric variety $X$ with $H_{3}(X)=0$, the equality $b=r$ holds and
	$b$ is combinatorially invariant. 
\end{corollary}
\begin{proof}
	The assumption implies that $H_{3}(\mathcal{U})=0$ and hence $H_{3}(I^{\bar{n}}X)=0$. By using the Mayer-Vietoris sequence of Remark \ref{MVIXB}, we get $b=r=2f_{2}-3 f_{3}.$
Reformulating the assumption yields $b=3f_{1}-f_{2}-6.$
	Setting both expressions equal to each other will result in the Euler formula.
\end{proof}

\section{Conclusion}
\label{conclusion}

We summarize the results obtained in this work. We started by considering toric varieties as topological pseudomanifolds. We introduced the construction of a link for an arbitrary stratum in a given toric variety. We also endowed links of strata with real co-dimensions 4 and 6 with CW structures and computed their homology groups. We showed that the link of each connected component of the real 4-co-dimensional stratum is always a rational homology 3-sphere. Thus, for an arbitrary $2n$-dimensional toric variety $X$, we get a $\mathbb{Q}$-stratification whose depth 1 stratum is of real co-dimension 6. We generalized the theory of intersection spaces to rational homology stratified pseudomanifolds with rationally isolated singularities. Our main objects of study are 6-dimensional toric varieties.
For these, we compare the Betti numbers of the associated 
generalized intersection spaces $IX$ and intersection homology in the following table.\\

\begin{table}[!h] 
	\begin{center}
		\begin{tabular}{|l||*{5}{c|}}\hline
			\backslashbox{$b_{\ast}$}{Theory}
			&\makebox[7em]{$H_{\ast}(IX)$}&\makebox[6em]{$IH_{\ast}(X)$}\\\hline\hline
			$b_{6}$ & $0$&$1$\\\hline
			$b_{5}$  & $m-1$ &$0$\\\hline
			$b_{4}$  &$f_{1}-3-b$&$f_{1}-3$\\\hline
			$b_{3}$  &$2(3f_{1}-f_{2}-b-6) $&$0$\\\hline
			$b_{2}$  & $f_{1}-3-b$ &$f_{1}-3$\\\hline
			$b_{1}$  & $m-1$&$0$\\\hline
			$b_{0}$  &$1$&$1$\\\hline
		\end{tabular}
		\caption{\label{table.hixihx}The Betti numbers associated to $HI_{\ast}$ and $IH_{\ast}$. }
	\end{center}
\end{table}
The left hand column of the table shows that the $b_{\ast}^{IX}$ are not combinatorially invariant. Indeed, we can determine the ordinary Betti numbers of $X$ up to and including the middle degree by using $b_{\ast}^{IX}$. However, the right hand column, as already observed by Fieseler in \cite{fieseler1991rational}, shows that the associated Betti numbers of intersection homology neither fix the combinatorial data of the fan nor the ordinary homology below the middle degree.

\end{document}